\documentclass{amsart}

\usepackage{amssymb}
\usepackage{amsmath}
\usepackage{amsthm}
\usepackage[mathscr]{eucal}

\newtheorem{thrm}{Theorem}[section]

\newtheorem{prop}[thrm]{Proposition}
\newtheorem{cor}[thrm]{Corollary}
\newtheorem{main}{Main Theorem}

\theoremstyle{definition}
\newtheorem{defn}[thrm]{Definition}

\theoremstyle{remark}

\numberwithin{equation}{section}

\newcommand{\dbar}{$\bar{\partial}$}
\newcommand{\mdbar}{\bar{\partial}}

\newcommand{\zb}{\bar{z}}

\newcommand{\Lb}{\overline{L}}

\newcommand{\Vb}{\overline{V}}
\newcommand{\Wb}{\overline{W}}

\newcommand{\omegab}{\bar{\omega}}

\newcommand{\lre}{\mathcal{E}}
\newcommand{\lrs}{\mathcal{S}}

\newcommand{\lrl}{\mathcal{L}}
\newcommand{\lrp}{\mathcal{P}}

\makeatletter
\newcommand{\opnorm}{\@ifstar\@opnorms\@opnorm}
\newcommand{\@opnorms}[1]{%
	\left|\mkern-1.5mu\left|\mkern-1.5mu\left|
	#1
	\right|\mkern-1.5mu\right|\mkern-1.5mu\right|
}
\newcommand{\@opnorm}[2][]{%
	\mathopen{#1|\mkern-1.5mu#1|\mkern-1.5mu#1|}
	#2
	\mathclose{#1|\mkern-1.5mu#1|\mkern-1.5mu#1|}
} \makeatother

\begin{document}
	
\title
[The \dbar-Neumann problem on intersection domains]
{
	Weighted estimates for the \dbar-Neumann
 problem on intersections of strictly pseudoconvex
 domains in $\mathbb{C}^2$}

\author{Dariush Ehsani}

\address{
	Hochschule Merseburg\\
	Eberhard-Leibnitz-Str. 2\\
	D-06217 Merseburg\\
	Germany}
\email{dehsani.math@gmail.com}


\begin{abstract}
	We obtain weighted estimates for
the \dbar-Neumann operator on intersections of
 two smooth strictly pseudoconvex domains in 
  $\mathbb{C}^2$.  The regularity estimates 
are described with the use of Sobolev norms with weights
which are powers of the defining functions of the 
 two domains. 

\end{abstract}

\maketitle


	
	\bibliographystyle{plain}

\section{Introduction}

This purpose of this work is to 
 study the \dbar-Neumann problem on transversal
intersections of strictly pseudoconvex domains.
Let $\Omega\subset\mathbb{C}^n$	be the 
 intersection of $m$ smoothly bounded strictly
 pseudoconvex domains, $\Omega_1,
 \ldots, \Omega_m$, which intersect
 (real) transversely, that is  
for all $1\le i_1 < \cdots < i_l \le m$, we have
 \begin{equation*}
 d\rho_{i_1}\wedge \cdots\wedge d\rho_{i_l} \neq 0
 \end{equation*}
on $\bigcap_{j=1}^l \partial\Omega_{i_j} 
  \cap \partial\Omega$, 
 where $\rho_j$ is a defining function for
 $\Omega_j$,
 or
 \begin{equation*}
 \Omega_j := \{ z\in\mathbb{C}^n | \rho_j<0 \},
 \end{equation*}
 for $j=1,\ldots, m$.
 Note in particular, we consider $m\le n$.
This definition is modeled on
  Range and Siu's description of
piecewise smooth domains \cite{RS}. 

We will be concerned with regularity estimates,
 as measured by Sobolev norms, of the
\dbar-Neumann operator, $N$, defined as the
 inverse to the $\square = \mdbar\mdbar^{\ast}
  + \mdbar^{\ast}\mdbar$ operator.  We should not
expect to prove regularity of the solution, 
 even if the given data form is smooth up to the
boundary, because of the singularities on the
 boundary of the domain.  Singularities of solutions
to elliptic equations 
 (the interior equations of the \dbar-Neumann problem
  are elliptic)
are well known in
 the situation of singular domains
\cite{Gr, K}.  
 However, limited regularity, up to a fixed order have
been shown by Michel and Shaw in \cite{MS98}:
\begin{thrm}
Let $\Omega\subset\subset \mathbb{C}^n$ be a piecewise
 smooth strictly pseudoconvex domain.  
The \dbar-Neumann operator is bounded as a 
 map of $(p,q)$-forms with
  $L^2$ coefficients, $N:L^2_{p,q}(\Omega)\rightarrow
  L^2_{p,q}(\Omega)$,
for $0\le p \le n$, $1\le q \le n-1$.  Furthermore,
 we have the following estimates:
\begin{equation*}
 \| N f\|_{W^{1/2}_{p,q}(\Omega)} \lesssim
  \|f\|_{L^2_{p,q}(\Omega)}
\end{equation*}
for $f\in L^2_{p,q}(\Omega)$.
\end{thrm}

We note that the symbol $\lesssim$ means 
 $\le c$ where $c$ is independent of the 
functions (or forms) being estimated.

In \cite{EhLi12}, Lieb and the current author
 considered singularities of another
 kind (arising on Henkin-Leiterer domains), and
obtained $L^p$
 estimates of the \dbar-Neumann operator  
 with the use of weights which vanish
 at the singularity of the domains. 
We use that idea of vanishing weights in the current
 paper.
 
We first define the weighted Sobolev norms
 on half-spaces
as in \cite{Eh18_halfPlanes}.  
 Let $\mathbb{H}_j^{2n}$ denote the
half-plane 
\begin{equation*}
 \mathbb{H}_j^{2n} = \{
   (x,\rho_j) \in \mathbb{R}^{2n}|
    \rho_j <0
 \}
\end{equation*}
and for multi-index, $I=(i_1,\ldots,i_l)$,
\begin{equation*}
 \mathbb{H}_I^{2n} := \bigcap_{i\in I} H_i^{2n}.
\end{equation*}
We use a non-standard notation to
 write powers of the $\rho_i$ functions
using the index notation to stress that 
 the powers (over each $\rho_i$) will be the same.
We write
\begin{equation*}
 \rho_I^{k\times l} :=
  \rho_{i_1}^k \rho_{i_2}^k \cdots \rho_{i_l}^k.
\end{equation*}

\begin{defn}
\label{defnWghtSpace}
Let $I$ be a multi-index, and
 $J$ a subset of $I$.
 For $\alpha\in \mathbb{R}$,
and
$s,k$ integers $\ge 0$ we have the
 spaces
\begin{multline*}
W^{\alpha,s} (\mathbb{H}_I^{2n},
\rho_J,k)
= \\
\left\{
f\in W^{\alpha}(\mathbb{H}_I^{2n}) 
\Big| \rho_J^{(sk - rk)\times |J| }f \in
W^{\alpha + s-r}(\mathbb{H}_I^{2n}) 
\mbox{ for each } 0\le r\le s
\right\} 
\end{multline*}
with norm
\begin{equation*}
\left\| f \right\|_{
	W^{\alpha,s} (\mathbb{H}_I^{2n},\rho_J,k) 
} 
= \sum_{0\le r \le s} 
\left\|\rho_J^{(sk - rk)\times |J|} f \right\|_{
	W^{\alpha + s -r} (\mathbb{H}_I^{2n}) 
}. 
\end{equation*}
In the case $J=I$ we will simply write
$W^{\alpha,s} (\mathbb{H}_I^{2n},
\rho,k):=W^{\alpha,s} (\mathbb{H}_I^{2n},
\rho_I,k)$.

\end{defn}

In dealing with boundary values
 (restriction to $\rho_j=0$ for some $j$),
a notation to deal with a missing index will be
 useful.
To indicate a missing index
(from the multi-index, $I$, which 
is to be known), we use 
$\rho_{\hat{i}_j} 
:= \rho_{i_1} \cdots \rho_{i_{j-1}}\rho_{i_{j+1}}
\cdots \rho_{i_l}^k$
and
\begin{equation*}
\rho_{\hat{i}_j}^{k\times (l-1)} :=
\rho_{i_1}^k \cdots \rho_{i_{j-1}}^k\rho_{i_{j+1}}^k
\cdots \rho_{i_l}^k.
\end{equation*}
Then a weighted Sobolev space 
 can be defined on 
$\mathbb{H}^{2n-1}_{I_{\hat{i}_j}}
 := \bigcap_{i\in I} \mathbb{H}_i^{2n}\big|_{\rho_{i_j} =0}$,
with norm
\begin{equation*}
\left\| f \right\|_{
	W^{\alpha,s} (\mathbb{H}^{2n-1}_{I_{\hat{i}_j}},
	\rho_{\hat{i_j}},k) 
} 
= \sum_{0\le r \le s} 
\left\|\rho_{\hat{i}_j}^{(sk - rk)\times (l-1)} f \right\|_{
	W^{\alpha + s -r} (\mathbb{H}^{2n-1}_{I_{\hat{i}_j}}) 
}. 
\end{equation*}

We can generalize the above spaces to 
 general intersection domains 
(of smooth domains) by localizing and using
 a coordinate system including the 
$\rho_j$ functions.

Our main result is the following
\begin{main}
	Let $\Omega\subset \subset \mathbb{C}^2$
be an intersection of two smoothly bounded
 strictly pseudoconvex domains.
	Let $f\in W^{0,2}_{(0,1)}(\Omega,\rho_j,2)$ for
	$j=1,2$.
	Let $N$ be solution operator 
	to the \dbar-Neumann problem.
	Then
\begin{equation*}
	\sum_j \|Nf\|_{W_{(0,1)}^{1/2,s}
		\left(
		\Omega,
		\rho,2
		\right)} \lesssim
	\sum_j \|f\|_{W_{(0,1)}^{0,s}
		\left(
		\Omega,
		\rho_j,2
		\right)}.
	\end{equation*}
\end{main}

We use a representation of the solution to the
 \dbar-Neumann problem as a sum of a
solution to a homogeneous Dirichlet problem
 (with the use of Green's operator) and the
solution to an inhomogeneous Dirichlet problem
 (with the use of a Poisson operator)
\cite{CNS92}.  Weighted estimates 
 on the boundary will be obtained by 
reducing the boundary conditions for the
 \dbar-Neumann problem to an equation, which to
highest order is just the Kohn boundary 
 problem, for $\square_b$. 
This study of the boundary condition
 is based on our earlier work
in \cite{Eh18_dno}, and is carried out 
 in Section \ref{secBndryEqns}.  
In particular the Dirichlet to Neumann operator
 (DNO),
 a boundary value operator expressing the
boundary values of normal derivatives in terms
of the given 
 boundary values, plays an important role, and
weighted estimates for the DNO are derived in
 Section \ref{secDNO}, which may be of interest
in its own right. 
 Estimates for the boundary solution 
 are obtained in Section \ref{secWeightedEst}
using standard integration by parts techniques.

In order to use these boundary estimates 
 to conclude estimates for the solution operator, 
the application of the Poisson operator to
 the boundary solution is studied.  
The Poisson operator is represented as 
 a sum of pseudodifferential operators
and combinations of pseudodifferential operators
 with restriction to boundary operators. 
Weighted estimates for such resulting
 operators used in 
Sections \ref{secPoisson} and
 \ref{secDNO} are taken from
\cite{Eh18_halfPlanes}.  In a similar manner
weighted estimates for the solution to a homogeneous
 Dirichlet problem are worked out in 
  Section \ref{secGreens}.

For the sake of simplicity, we work with 
 {\it generic corners} following the
terminology of \cite{CV13} in which
\begin{equation*}
 \partial\rho_{i_1} \wedge 
 \cdots \wedge \partial\rho_{i_m} \neq 0
\end{equation*}
 at points where
$\partial\Omega_{i_1}, \ldots
\partial\Omega_{i_m}$ intersect.
Although we 
 setup much of the work in 
$\mathbb{C}^n$, our results are restricted to
 $n=2$, mainly in order to make use of the
vanishing of a problematic term, which 
 has worse weighted regularity mapping
properties than those used to obtain our 
 Main Theorem.  We refer the reader
  to the simplification in 
Section \ref{secBndryEqns} to see where the problematic
 term can just be ignored.

\section{Notation for operators on intersection domains}

\label{secNotation}

We will use the index notation, $\hat{i}$
to refer to an index, $i$, which is to be
omitted in whatever variables, respectively, 
operators, are being considered.  Thus, for
instance, we use $F.T._{\hat{j}}$ to denote
the partial Fourier Transform in all 
variables other than the $j^{th}$ variable.
Similarly,
\begin{equation*}
x_{\hat{j}} :=
(x_1,\ldots, x_{j-1},x_{j+1},\ldots, x_n).
\end{equation*} 
In specifying a value for a particular 
coordinate, we mark that value
 with
the index.  Thus we write
\begin{align*}
\left( F.T._{\hat{j}} g \right) \Big|_{x_j=0}
=& \widehat{g}(\xi_1,\ldots,\xi_{j-1},0_j
,\xi_{j+1},\ldots, \xi_n).
\end{align*}
 This should not be confused with
$\widehat{g}(\xi_1,\ldots,\xi_{j-1},0
,\xi_{j+1},\ldots, \xi_n)
 = \widehat{g}(\xi)\Big|_{\xi_j=0}$.
With slight abuse of notation we will 
 rearrange the order of the arguments and write
the variable fixed to a specific value first; we
 write
$\widehat{g}(0_j,\xi_{\hat{j}})$ in place of
\newline $\widehat{g}(\xi_1,\ldots,\xi_{j-1},0_j,
 ,\xi_{j+1},\ldots, \xi_n)$.
Note that above we use $\widehat{g}$ to refer
to both the partial and full Fourier Transforms
of the function $g$, the particular transform
being clear from context. 

Boundary values naturally arise 
 in the Fourier representation of differential
equations (on domains with
 boundary) and it will be useful to have a 
 notation representing the restriction to
a given boundary.  We will use the operator
 $R_j$ to denote the restriction of a 
function (or form) to $\partial\Omega_j$:
\begin{equation*}
 R_j \phi = \phi \Big|_{\rho_j=0}.
\end{equation*}

We borrow notation from
 \cite{Eh18_dno} on pseudodifferential operators.
In particular, we write
 $\Psi^{k}(\Omega)$ to denote the class of pseudodifferential
operators of order $k$ on $\Omega$, and we reserve 
 the notation $\Psi^k$ to indicate an operator belonging
to class $\Psi^{k}(\Omega)$.  Often the meaning
 of $\Psi^k$ will change from one line to
the next.  We also use $\Psi_b^k$ to denote an operator
 belonging to class $\Psi^k(\partial\Omega)$.  
With slight abuse of
 notation, we will also use
$\Psi_{b}^k$ to denote an operator belonging
 to
 class $\Psi^k(\partial\Omega_j)$ for
any particular boundary 
 $\partial\Omega_j$.

As is customary, we write
 $\Psi^{-\infty}$ to denote a smoothing operator:
\begin{equation*}
  \Psi^{-\infty}: W^s(\Omega) \rightarrow W^{-\infty}(\Omega)
\end{equation*}
 for any $s$.

In writing a pseudodifferential operator 
 (on $\mathbb{R}^n$) applied to
 a distribution supported on a domain, we shall use the
convention that the distribution will be considered to be
 extended by zero to all of $\mathbb{R}^n$.  Thus,
for instance, if $\varphi \in L^2(\mathbb{H}^2_2)$,
 where 
\begin{equation*}
\mathbb{H}_{2}^2 := \{
(x_1,x_2) \in \mathbb{R}^2 : x_2 < 0
\},
\end{equation*}
and $A\in \Psi^2(\mathbb{R}^2)$, we write
$A\varphi$ to mean
\begin{equation*}
 A\varphi := \frac{1}{(2\pi)^2}
  \int \sigma(A) \widehat{\varphi^{E_2}}(\xi_1,\xi_2)
   e^{ix\xi} d\xi_1 d\xi_2,
\end{equation*}
 where $\varphi^{E_2}$ refers to extension by zero
over $x_2>0$.

In general if we have a distribution, 
 $\varphi$, defined
 on $\Omega = \cap_j \Omega_j$, we will 
use the notation $\varphi^{E_k}$ to denote
 the extension by zero of $\varphi$
to $\cap_{j\neq k}\Omega_j$.  
 Similar extensions will be used for boundary
distributions.  For example, in the case 
 $\Omega= \Omega_1\cap \Omega_2$, consider
$\varphi \in 
 L^2 (\partial\Omega_1\cap \partial\Omega)$.  
We denote $\varphi^{E_2}$
  to be the distribution
supported on all of 
 $\partial\Omega_1$ defined by 
  an extension by zero.
From Theorem 1.4.2.4 in \cite{Gr} 
  we have the following
 results concerning extension by zero:
for $\Omega = \cap_j \Omega_j$ and
 $0\le s\le 1/2$,
if $g\in W^{s}(\Omega)$ then
 $g^{E_k} \in W^{s}(\cap_{j\neq k}\Omega_j)$.

Operators mapping distributions supported on
 one boundary, $\partial\Omega_j$, to a distribution
supported on another boundary, 
 $\partial\Omega_k$, also arise in the
pseudodifferential analysis of operators on
 intersection domains.  On the one hand,
Fourier Transforms of derivatives lead to 
 boundary value terms, while restrictions to
the various boundaries of the intersection 
 domain produce new terms supported on the
respective boundaries.  It will be necessary 
 to study the Sobolev mapping properties of
such operators.

As an example consider the intersection of
 half-spaces, $\mathbb{H}_{1,2}^2$, defined by
\begin{equation*}
 \mathbb{H}_{1,2}^2 := \{
  (x_1,x_2) \in \mathbb{R}^2 : x_1, x_2 < 0
 \}.
\end{equation*}
and a function, $h(x_1,x_2)$, in a Sobolev space,
  $W^s(\mathbb{H}_{1,2}^2)$, for some 
   $s\ge 0$. A Fourier Transform
of $\partial_{x_1} h$ would lead to 
 partial transforms on $x_1=0$:
\begin{equation*}
 \widehat{\partial_{x_1} h}
  = F.T_{\hat{1}} h(0_1,\xi_2)
   + i\xi_1 \widehat{h}.
\end{equation*}
It is the first term on the right,
in combination with pseudodifferential 
 operators and restrictions,
which we will now discuss.

We illustrate with an example.  Let
 $\chi'(\xi_2)\in C^{\infty}(\mathbb{R})$
be such that $\chi'(\xi_2) \equiv 0$ in a neighborhood of
 the origin.  We
 look at the operator which sends
$h|_{x_1=0}$ to a function supported on
 $x_2=0$ via
\begin{align*}
 R_2 \circ \int \chi'(\xi_2)
\frac{\widehat{h}(0_1,\xi_2)}{\xi_1^2+\xi_2^2}
 e^{i x \xi} d\xi =&
  \int \chi'(\xi_2)
  \frac{\widehat{h}(0_1,\xi_2)}{\xi_1^2+\xi_2^2}
  e^{i x_1 \xi_1} d\xi \\
=&\pi \int \chi'(\xi_2)
\frac{\widehat{h}(0_1,\xi_2)}{|\xi_2|}
e^{x_1 |\xi_2|} d\xi_2,
\end{align*}
where 
$R_2$ is a restriction operator, 
denoting restriction to $x_2=0$.
Taking derivatives yields
\begin{equation*}
 \frac{\partial^k}{\partial x_1^k}
  \int \chi'(\xi_2)
  \frac{\widehat{h}(0_1,\xi_2)}{|\xi_2|}
  e^{x_1 |\xi_2|} d\xi_2
= \int \chi'(\xi_2)
|\xi_2|^{k-1}
\widehat{h}(0_1,\xi_2)
e^{x_1 |\xi_2|} d\xi_2,
\end{equation*}
and upon squaring and integrating
\begin{align*}
\Bigg\|
\frac{\partial^k}{\partial x_1^k}
\int \chi'(\xi_2)
\frac{\widehat{h}(0_1,\xi_2)}{|\xi_2|}
e^{x_1 |\xi_2|}& d\xi_2
\Bigg\|^2_{L^2(\mathbb{R}_{<0})}\\
\lesssim & 
\int \chi'(\xi_2)
|\xi_2|^{2k-2}
\left| \widehat{h}(0_1,\xi_2)
\right|^2	
e^{2 x_1 |\xi_2|}  d\xi_2 dx_1\\
\simeq &
\int \chi'(\xi_2)
|\xi_2|^{2k-3}
\left| \widehat{h}(0_1,\xi_2)
\right|^2	
  d\xi_2\\
\lesssim&
\Big\|
h|_{x_1=0}  
\Big\|^2_{W^{k-3/2}(\mathbb{R}_{<0})}.
\end{align*}
Interpolating for non-integer
Sobolev spaces, we get
\begin{equation*}
\Bigg\|
 R_2 \circ \int \chi'(\xi_2)
\frac{\widehat{h}(0_1,\xi_2)}{\xi_1^2+\xi_2^2}
e^{i x \xi} d\xi
\Bigg\|^2_{W^{s+3/2}\big(\{x_1<0\}\big)}\\
\lesssim 
\left\|
h|_{x_1=0}  
\right\|^2_{W^{s}
	\big(\{x_1<0\}\big)}
\end{equation*}
for $s\ge 0$.

Generalizing the above result to apply to 
 transversal intersections of domains, we
use the notation
 $\lre^{jk}_{-\alpha,\gamma}$, where
 $-\alpha-1/2 \in\mathbb{N}$,
for $\alpha\ge 1/2$, for
 certain operators (to be made precise shortly)
with the property 
\begin{equation}
 \label{errorMap}
\lre^{jk}_{-\alpha,\gamma}:  W^s(\partial\Omega_j)
 \rightarrow W^{s+\alpha}(\partial\Omega_k)
\end{equation}
valid for $s\ge \gamma$.
 In the case $\gamma = 0$ we will 
use the notation
\begin{equation*}
 \lre^{jk}_{-\alpha} :=
  \lre^{jk}_{-\alpha,0}.
\end{equation*}

To make precise which operators are to 
 be included in the
  $\lre^{jk}_{-\alpha,\gamma}$ operators,
we start with $\gamma=0$ and include
 in $\lre^{jk}_{-\alpha}$ any
 operators
which for any $N$, can be written in the form
\begin{equation*}
\lre_{-\alpha}^{jk} 
= R_k \circ A_{-\alpha-\frac{1}{2}} 
+ R_k \circ \Psi^{-N} ,
\end{equation*}
where
 $A_{-\alpha-1/2} \in \Psi^{-\alpha-1/2}(\Omega)$ is
such that its symbol, 
  $\sigma(A_{-\alpha-1/2})(x,\rho,\xi,\eta)$
has the property that it is meromorphic in the
$\eta$ variables with poles in $\eta_j$ which are 
 elliptic symbols in class $\lrs^1(\Omega)$ 
  (restricted to $\eta_j=0$).  
 We shall reserve the notation, $A_{-k}$
 for $k\ge 1$, for operators whose symbols satisfy the above
conditions.  As we will see this condition applied again,
 we refer to a pseudodifferential
operator of order $-k$, 
 $\Psi^{-k}$, which can be written
  for any $N\ge k$ in the form 
\begin{equation*} 
 \Psi^{-k} = A_{-k} + \Psi^{-N}
\end{equation*}
as {\it decomposable}.

Note that in the
 case $j=k$ above, we have
  included in 
$\lre_{-\alpha}^{jj}$ (for $\alpha \in \mathbb{N}$) 
 terms of the form
\begin{align*}
 \lre_{-\alpha}^{jj} =& R_j \circ \Psi^{-\alpha -1}\circ R_j\\
  =& \Psi^{-\alpha }_{bj}
\end{align*}
for $\alpha\ge 1$, where we 
 use the subscript $bj$ to specify
a pseudodifferential operator
 on the boundary $\partial\Omega_j$.
The last line is due to
 a restriction property of 
pseudodifferential operators
 as given in Lemma 2.7 
  of \cite{Eh18_halfPlanes}.

Lastly, we will include compositions 
 of such boundary operators, and here the
$\gamma$ value will be of importance.
We write
\begin{equation*}
\lre^{jk}_{-\alpha,\alpha_1} = 
\lre^{lk}_{-\alpha_1} \circ 
\lre^{jl}_{-\alpha_2}
\qquad \alpha_1,\alpha_2 \ge 1/2,
\quad \alpha=\alpha_1+\alpha_2
\end{equation*}
for $l\neq j$.

We will also write 
$\lre_{-\beta}^{jk} = \lre_{-\alpha}^{jk}$ for any
 $\alpha \le \beta$,
and
 $\lre_{-\beta,\gamma}^{jk} = \lre_{-\alpha}^{jk}$
for $\alpha = \beta-\gamma$.

We refer to 
 \cite{Eh18_halfPlanes} for estimates regarding the
$\lre_{-\alpha}^{jk}$ operators.
 From Corollary 4.7 in \cite{Eh18_halfPlanes},
we have
\begin{thrm}
\label{weightedLRE}
 Let 	$-1/2\le \delta \le 1/2$,
and
  $g_{bj} \in W^{\delta ,s}
 \left(
  \partial \Omega \cap \partial \Omega_j 
  ,\rho_{\hat{j}}, \lambda
  	\right)$.
Then
for $\beta\ge \gamma$ with
$\beta - \alpha\le \delta$, we have
\begin{equation*}
\left\|
\lre_{-\alpha,\gamma}^{jk} g_{bj}
\right\|_{W^{\beta,s}\left(
 \partial \Omega \cap \partial \Omega_k ,
	\rho_{\hat{k}},\lambda \right)}
\lesssim \| g_{bj}\|_{W^{\beta-\alpha,s}
	\left(
	 \partial \Omega \cap \partial \Omega_j,
	 \rho_{\hat{j}},\lambda \right)}.
\end{equation*}
\end{thrm}

To illustrate the importance of 
 the $\gamma$ value, we consider
 a composition of two
$\lre^{jk}_{-1/2}$ operators:
$
 \lre_{-1/2}^{12} \circ
  \lre_{-1/2}^{21}
$
cannot be written as $\lre_{-1}^{11}$
 due to the condition in 
Theorem \ref{weightedLRE} that the
 Sobolev norm on the left-hand side of 
the estimates be $\ge 0$.  The 
 condition in two applications of 
the Theorem to obtain estimates for 
 $ \lre_{-1/2}^{12} \circ
 \lre_{-1/2}^{21}
 $
will however be satisfied if we 
 try to estimate 
Sobolev $1/2$ estimates.
Thus, we can write
\begin{equation*}
  \lre_{-1/2}^{12} \circ
 \lre_{-1/2}^{21}
  = \lre_{-1,1/2}^{11}.
\end{equation*}
Most such operators in this
 article will involve a $\gamma$ value
of zero, but in a few places 
 a higher order will be needed due to 
various compositions of operators.  

We now give
 some results concerning interior estimates
involving decomposable operators. 
Suppose that the operator
$A\in \Psi^{-k}$ is decomposable.
Then we have the following theorem
concerning weighted estimates
 from \cite{Eh18_halfPlanes}:
\begin{thrm}
\label{thrmWghtdDecomp}
(Theorems 4.3 and 4.4
 in \cite{Eh18_halfPlanes})
	\label{weightedPsi}
	Let $A_{-\alpha}$ be 
decomposable operator
(of order $-\alpha\le -1$).
Then, for $1/2\le \beta\le\alpha+1/2$,
and \newline $g_{bj}\in W^{\max\{\beta-\alpha ,0\},s}
\left(
\partial\Omega\cap \partial\Omega_j,
 \rho_{\hat{j}},k
\right)$,
\begin{equation*}
\left\|
A_{-\alpha} g_{bj}
\right\|_{W^{\beta-1/2,s}
	\left( \Omega,
	\rho,k \right)}
\lesssim 
\| g_{bj}\|_{W^{\beta-\alpha,s}
	\left(
	\partial\Omega\cap \partial\Omega_j,
	\rho_{\hat{j}},k	\right)}.
\end{equation*}
\end{thrm}

 For pseudodifferential operators
  acting on a distribution supported on the
 interior of $\Omega$ (as opposed to the
  boundary as in the above theorems),
we have
\begin{thrm}[Theorem 4.6 in \cite{Eh18_halfPlanes}] 
\label{PsiInt}
Let $A_{-\alpha}\in \Psi^{-\alpha}(\Omega)$
 for $\alpha\ge 0$.  For
 $g \in W^{0,s} (\Omega,\rho_{\hat{j}},k)$
 \begin{equation*}
 \| A_{-\alpha} g
 \|_{W^{\alpha,s}(\Omega,\rho_{\hat{j}},k)}
 \lesssim
 \left\|
 g
 \right\|_{W^{0,s}
 	(\Omega,\rho_{\hat{j}},k)}.
 \end{equation*}
\end{thrm}

 We will later consider operators
which are matrices composed of 
 operators of the various above types.
Suppose, $M$ is an $n\times n$ matrix operator
 and $u$ is a vector,
\begin{equation*}
\left[ 
 \begin{array}{c}
   u_1 \\
   \vdots\\
   u_n
 \end{array}
\right]
\end{equation*}
 with
$u_j\in H_j$ for some (weighted Sobolev)
 space denoted $H_j$.  
Then we write
\begin{equation*}
 M: H_1 + \cdots+ H_n \rightarrow
  H_1' + \cdots +H_n',
\end{equation*}
where the $H_j'$ denote
  some (weighted Sobolev)
 spaces,
 to mean the $j^{th}$ component of 
$Mu$, or $(Mu)_j$ satisfies
\begin{equation*}
 \|(Mu)_j\|_{H_j'} \lesssim 
  \sum_k \|u_k\|_{H_k}.
\end{equation*}
In the case the $H_j$ 
 are all the same, we will omit the
summation signs.

Furthermore, regarding pseudodifferential 
operators, we will use the notation
$\Psi^{\alpha}_{\varepsilon}$ to refer to
pseudodifferential operators with small operator
norm, by which we mean, given some sufficiently small
neighborhood $U$, we have
\begin{equation}
\label{smallOps}
\| \Psi^{\alpha}_{\varepsilon} v \|_s \le \varepsilon \| v\|_{s+\alpha}
\end{equation}
for all $s\ge 0$ and for all $v$ with support in $U$. 

In Sections \ref{secBndryEqns} and
 \ref{secWeightedEst} we discuss 
the boundary equations related to the 
 \dbar-Neumann problem, and in obtaining estimates
for the solutions to given boundary equations, we
 isolate a particular problematic direction in
which to obtain a gain of regularity
 (in a weighted sense).  This is a normal
phenomenon in the analysis in the theory of
 the \dbar-Neumann problem in which certain
  operators behave as elliptic operators with 
the exception of their behavior in one particular
 microlocal region.  To describe this problematic
region we recall a microlocal decomposition as 
 given in \cite{Ch91, KoNi06, K85, Ni06}.  We describe the situation in 
  $\mathbb{R}^3$ (considered as the boundary
of a half-plane in $\mathbb{R}^4$).

We choose a smooth partition of 
 the two dimensional unit sphere $|\xi|=1$ 
with functions $\psi_k^+$, $\psi_k^0$, and
 $\psi_k^-$, with dependence on a 
  parameter $k$, in such a way that 
$\psi^+_k$ has support in 
 $\xi_3 > k \sqrt{\xi_1^2+\xi_2^2}$ and is
equal to 1 when
 $\xi_3 > (k+1) \sqrt{\xi_1^2+\xi_2^2}$.
$\psi^-_k$ is defined symmetrically, so that
 $\psi^-_k(\xi_1,\xi_2,-\xi_3) 
  = \psi^+_k(\xi_1,\xi_2,\xi_3)$,
and finally, 
 $\psi^+_k + \psi^0_k + \psi^-_k =1$ on the
unit sphere.

The functions are then extended to all of 
 $\mathbb{R}^3$ in the following way.  First,
they are extended radially 
 (so they are symbols of zero order pseudodifferential
operators) to everywhere
 outside a neighborhood
 of the origin.  A cutoff equivalently equal to
  1 in a neighborhood of the origin is then
included in (an extension of) the $\psi^0_k$ function
 so that on $\mathbb{R}^3$ we have a smooth
partition of unity from three order 0 symbols.  
 We refer the reader to the above mentioned 
papers for more details of the decomposition.  

\section{Setup of the \dbar-Neumann problem}
\label{secSetup}
While our final results are stated in the
 case of $n=2$, 
we can set up the \dbar-Neumann problem on
 intersection domains in $\mathbb{C}^n$
for $n\ge 2$.
 As in the Introduction, we set
 $\Omega = \bigcap_{i=1}^m\Omega_i$
where the $\Omega_i\subset\mathbb{C}^n$
 are smoothly bounded strictly
pseudoconvex domains
 intersecting real transversely.

The operator, $\square$, is defined according to
\begin{equation*}
\square=
\mdbar\mdbar^{\ast} + \mdbar^{\ast}\mdbar,
\end{equation*}
and for 
$f$ a $(0,q)$-form with components 
 in $L^2(\Omega)$, 
written $f\in L^2_{(0,q)}(\Omega)$,
the \dbar-Neumann problem is the boundary 
value problem:
\begin{equation}
 \label{dbarNEqn}
\square u = f \quad \mbox{in }\Omega
\end{equation}
with the boundary conditions
\begin{equation}
\label{bndryCndns}
\begin{aligned}
& \mdbar u \rfloor \mdbar \rho_j = 0,\\
&  u \rfloor \mdbar \rho_j = 0,
\end{aligned}
\end{equation}
on $\rho_j=0$, for $j=1,\ldots, m$.

We work in a neighborhood of
a given point, $p\in\partial \Omega$
at which all the domains intersect;
$
p\in \bigcap_{i\in I}\partial\Omega_i
$
for $I=\{1,\ldots,m\}$. 
 We will further assume that
at the point $p$ we have
\begin{equation*}
 \partial\rho_1 \wedge 
   \cdots \wedge \partial\rho_m \neq 0.
\end{equation*}
Such is the case for so-called Bell domains
 (see \cite{Ba95}), also called domains with
{\it generic corners} in \cite{CV13}.
The same procedure can be carried out for
 points at which a subset of 
domains intersect with obvious modifications.

We work with a metric so that 
$
 \omega_{1} = \partial \rho_{
	1}, \ldots, \omega_{m} = \partial \rho_{m}  
$ 
make up part of an orthonormal frame of
$(1,0)$-forms in a neighborhood of $p$.  
Let $L_1,\ldots,L_{m}$ be dual to
$\omega_1,\ldots,\omega_{m}$,
respectively.
In local coordinates we have
 the following
representations of the vector fields:
\begin{equation}
\label{lpt}
L_j = \frac{1}{\sqrt{2}}
\frac{\partial}{\partial
	\rho_{j}} +
i T_{j} \qquad 1\le j \le
m,
\end{equation}
where $T_j$ is tangential to 
 $\partial\Omega_j$,
 and in local coordinates will be written 
\begin{equation*}
T_j = \frac{\partial}{\partial x_j}
+O(\rho_j) \qquad 
1\le j \le m.
\end{equation*}

We use the convention as in
 \cite{CNS92} that the holomorphic vector fields
transverse to the boundary are written with 
positive imaginary part.

Without loss of generality we take
the singular boundary point
$p$ to be the origin.
Then, lastly, for
$j=1,\ldots, n-m$,
we write
\begin{equation}
\label{vxx}
V_{m+j}=\frac{1}{2}\left(
\frac{\partial}
{\partial x_{m+2j-1}}
-i \frac{\partial}
{\partial x_{m+2j}}
\right) +\sum_{k=1}^{2n-m}
\ell^{m+j}_k(x)
\frac{\partial}{\partial x_k} 
+\sum_k O(\rho_k).  
\end{equation}
%

We use the standard notation of forms with
indices, so that
\begin{equation*}
\omega_J = \omega_{j_1}\wedge
\cdots \wedge\omega_{j_{|J|}}.
\end{equation*}
Let $f$ be of the form $f_J\omegab_J$,
for a single index $J$.  If we can solve
the \dbar-Neumann problem for all $f$ of the form
$f=f_J\omegab_J$, for a single index, $J$,
 we can solve the
problem for any $(0,|J|)$-form.  We thus
look at
\begin{equation}
\label{interiorCndnJ}
\bar {\partial} \bar {\partial}^{\ast}u + 
\bar {\partial}^{\ast}
\bar{\partial}u=f_J\omegab_J.
\end{equation}
We want to find $u_K$ with $|K|=|J|$ so that
$u=\sum_K u_K \omegab_K$ gives the solution
to problem \eqref{interiorCndnJ} with
the boundary conditions in 
\eqref{bndryCndns}.

 In \cite{Eh18_dno} (Proposition
  3.1), we obtained an
expression for the
 $\omegab_J$ component of 
$\square u_K \omegab_K$ for any
 $K$ for which $|K| = |J|$ and
as sets $K\setminus J$ contains at most
 one index.  
As in \cite{Eh18_dno} we use the notation
 $c_J^{J'}$ to be the function which
satisfies
\begin{equation}
 \label{cJSet}
 \mdbar \omegab_{J'} = c_J^{J'} 
  \omegab_{J},
\end{equation} 
modulo forms orthogonal to 
 $\omegab_J$.
Note that with the above notation
 $|J| = |J'| +1$.  We also write
$d_j$ for functions arising in 
 integration by parts involving the
$\Lb_j$ fields: formally,
\begin{equation*}
 (\phi, \Lb_j \phi) = \big( 
 (-L_j + d_j)\phi,\phi
  \big),
\end{equation*}
where $\phi$ is a smooth function
 supported away from the boundary, $\partial\Omega$. 
 
With slight modification to accommodate
our notation of the fields 
$L_j$ and $V_j$, Proposition 3.1
of \cite{Eh18_dno}
 yields in the present situation
\begin{prop}
	\label{squareOp}
	Modulo the vector fields
	$V_j$ or $\Vb_j$ for 
	$j=m+1,\ldots n$,
	zero order terms, or forms orthogonal to 
	$\omegab_J$, we have
\begin{align*}
& i)\ \square \left(
	u_{J}  \omegab_{J}
	\right) =
	-\left(\sum_{l\in J \atop 1\le l \le m} \Lb_l L_l +
	\sum_{l\in J \atop m+ 1\le l \le n}\Vb_l V_l\right) u_{J}\omegab_J\\
&\qquad \qquad
	-\left( \sum_{l\notin J\atop 1\le l \le m}  
	L_l\Lb_l +\sum_{l\notin J\atop m+1\le l \le n} 
	  V_l\Vb_l \right) u_{J}\omegab_J
	\\
&\qquad \qquad
	+
	\sum_{k=1}^{m} \left[
	(-1)^{|J\cup \{ k\}|}  \left(
	c^{J\setminus \{k\}}_{J\setminus \{k\} \cup \{ k\}} L_k
	-\overline{c}^{J\setminus \{k\}}_{
		J\setminus \{k\} \cup \{ k\}} \Lb_k
	\right) 
	+  d_k
	\Lb_k \right] u_J \omegab_J
	\\
& ii)\ \square u_{J_{\hat{k}}l}
	\omegab_{J_{\hat{k}}\cup\{l\}}=
	- \varepsilon^{lJ_{\hat{k}} }_{J_{\hat{k}}\cup \{l\}} \varepsilon^{kJ_{\hat{k}} }_{J}
	[ \Wb_k, W_l]
	u_{J_{\hat{k}}l}\omegab_J,
\end{align*}
	where
	$W_j=L_j$ for 
	$j=1,\ldots, m$, and
	$W_j=V_j$ for $j=m+1,\ldots, n$.
\end{prop}

The boundary  condition,
$ u \rfloor
\mdbar\rho_j = 0$,
for $1 \le j
\le m$ can be expressed by
\begin{equation}
\label{dirBndry1}
u_K = 0 \qquad \mbox{ for } j\in K,
\end{equation}
on $\partial\Omega_j$.   

We also define the (smooth) functions 
$c^K_{jK}$ by
\begin{equation*}
c^K_{Kj} = \mdbar \omegab_K \rfloor 
\left(
\omegab_{K} \wedge \omegab_j
\right).
\end{equation*}
 Note that this definition can
  be reconciled with that using \eqref{cJSet},
if the subscript, $Kj$, is thought of as an
 ordered set, with $j$ the last entry.
This definition eliminates much of the need for
 permutation sign functions,
$ \varepsilon^J_K $, giving the sign of 
 the permutation between 
  ordered indices, $J$ and $K$,
and defined to be zero if $J$ and $K$ do not
 contain the same indices.

We write, for fixed $j$,
\begin{align*}
\mdbar u 
\rfloor \mdbar\rho_j =&
\left(   \sum_{ K \not\owns j} \Lb_j u_K  \omegab_j
\wedge \omegab_K + 
\sum_{ K\owns j \atop {k\neq j}} \Psi_{t_{\rho_j}}^1u_{K}
\omegab_{K\cup\{k\}}
\right)
\rfloor \mdbar\rho_j\\ 
&+\sum_{j\notin K} c_{Kj}^K u_K \omegab_K
-\sum_{l\notin K} \sum_{j\in K} 
 \varepsilon^K_{(K\setminus j)j} c_{Kl}^K u_K 
 \omegab_{K\setminus j} \wedge \omegab_l
\end{align*}
where $\Psi_{t_{\rho_j}}^1$ refers to 
first order operators, tangential to
$\partial \Omega_j$.
 Then
we can use \eqref{dirBndry1}
to express the boundary condition
$ \mdbar u \rfloor \mdbar
\rho_{j} = 0$ on
$\rho_{j}=0$ by
\begin{equation}
\label{NeumannA}
\Lb_j u_K +
(-1)^{|K|}c_{Kj}^K u_K=0
\end{equation}
for $j\notin K$, and
$1\le j\le m$.

As the operator on the left-hand side of
\eqref{dbarNEqn} is elliptic, we use the 
Poisson and Green's
operators to describe
a general solution
and then insert it into the boundary conditions 
 \eqref{NeumannA}
to get the properties of the specific solution
to the \dbar-Neumann problem.  We thus first
derive some properties of the Poisson 
and Green's operators on the intersection
domains.

\section{The Poisson operator on intersection domains}
\label{secPoisson}

We will follow the
 calculus of pseudodifferential operators
in describing the Poisson operator 
 (a good reference for pseudodifferential
  operators on smooth manifolds is
 \cite{Tr}; we follow our own work in
\cite{Eh18_halfPlanes} in the presentation of 
 pseudodifferential operators on domains
  with boundary).  As we will see, the
non-smooth boundary presents problems in this
approach, and so we will need to define 
 boundary operators, which, although derived
from pseudodifferential operators, do not
lend themselves to the calculus (for example, due to
 the $\lre^{jk}_{\alpha,\gamma}$ operators which
arise below).  Nonetheless,
the Sobolev mapping properties of such problematic
 operators can be
characterized and suffice to obtain estimates
for the final solution. 

We consider the homogeneous
Dirichlet problem for the operator
$2\square$ on $\Omega$, with
prescribed boundary values, $g_b$:
\begin{equation}
 \label{homoDirProb}
\begin{aligned}
& 2\square u =0 \qquad \mbox{in } \Omega\\
& u = g_b \qquad \mbox{on }
\partial\Omega.
\end{aligned}
\end{equation}
 From \cite{JK95} (see Section 5 in
  \cite{JK95};
see also Theorem 5.1 of \cite{Ve84} and 
Theorem 1.4.3 of \cite{Sh05}), 
given $g_b\in W^{s}(\partial\Omega)$,
 for $0\le s\le 1$,
 there exists a unique solution, 
 $u\in W^{s+1/2}(\Omega)$,
to the homogeneous Dirichlet problem,
 such that $u \rightarrow g_b$ almost 
  everywhere, where the limits are taken
non-tangentially.  We call the unique solution
 $u$ to be the Poisson solution for
$g_b$ associated with the domain $\Omega$.

 We look at the Poisson solution locally, in 
a neighborhood of an intersection point on 
 the boundary, which we take to be the
origin as in Section \ref{secSetup}.  At such a
point there are several boundaries of 
 different domains, and to distinguish the 
boundary values, we use an index.
 Thus,
\begin{equation*}
   g_{bj}:= g_b \Big|_{\partial\Omega_j}.
\end{equation*}
 
 Our goal in this section is to obtain an
expression for the Poisson solution in terms
 of order $-1$ pseudodifferential operators
acting on $g_b$,
  modulo lower order error terms.  The 
 operator and its errors will be expressed
in terms of the data boundary function
 (or form), $g_b$, and the non-tangential
limits of the normal derivatives of the unique
 solution, $\partial_{\rho_j} v|_{\partial\Omega_j}$.  The operator 
mapping $g_b$ to the solution, $u$, is called
 the Poisson operator.  
 
To keep track of (smooth) error terms, we use the
 following notation:
 we write $R^{-\infty}$ to mean
 $ \Psi^{-\infty}$ applied to 
  $u$ or to 
 $g_{bj}\times \delta(\rho_j)$ or to 
$\partial_{\rho_j}u
  \big|_{\partial\Omega_j} \times \delta(\rho_j)$.
Furthermore, related to a single boundary,
 $\partial\Omega_j$, we use the notation
$R_{bj}^{-\infty}$ to denote 
 $ \Psi^{-\infty}_{bj}$ applied to 
 $g_{bj}$ or to 
 $\partial_{\rho_j}u
 \big|_{\partial\Omega_j} $,
and also to denote
 $R_j$ composed with a term from $R^{-\infty}$:
$ R_j\circ R^{-\infty}=R_{bj}^{-\infty}$.
 And finally,
we also write $R^{-\infty}$ to include
any term which can be written in the form
\begin{equation*}
\Psi^{-\alpha} \left(R_{bj}^{-\infty}
 \times \delta(\rho_j) \right)
\end{equation*}
for $\alpha\ge 1$, where the 
$\Psi^{-\alpha}$ is decomposable.

We note that for estimates of the smooth terms, 
 using Theorems \ref{thrmWghtdDecomp} and
  \ref{PsiInt}, we have
 for any $\alpha,s,k \ge0$
\begin{align*}
\nonumber
  \| R^{-\infty} \|_{W^{\alpha,s}(\Omega, \rho,k)} 
 \lesssim& 
 \|u\|_{W^{-\infty ,s}(\Omega, \rho,k)}
  + \sum_j \|g_{bj}\|_{W^{-\infty ,s}
  	 (\partial \Omega \cap \partial\Omega_j, \rho_{\hat{j}},k)}
   \\
\nonumber
  &+ \sum_j \|\partial_{\rho_j}u
  \big|_{\partial\Omega_j}
  \|_{W^{-\infty ,s}
  	(\partial \Omega \cap \partial\Omega_j, \rho_{\hat{j}},k)}\\
\nonumber
\lesssim& 
\sum_j \|g_{bj}\|_{L^2
	(\partial \Omega \cap \partial\Omega_j)} +
\sum_j \|\partial_{\rho_j}u
\big|_{\partial\Omega_j}
\|_{W^{-\infty}
	(\partial \Omega \cap \partial\Omega_j)}.
\end{align*}

We can estimate boundary values of 
 a term, $\partial_{\rho_j}u
 \big|_{\partial\Omega_j}$ by 
  assuming support in a neighborhood of
   $\partial\Omega$ intersected with $\Omega$
 and writing
\begin{align*}
 \partial_{\rho_j}u
 \big|_{\rho_j =0 } 
  =& \int_{-\infty}^0 \partial_{\rho_j}^2 u d\rho_j \\
  =& \int_{-\infty}^0 \Lambda_{t_j}^2 u
    d\rho_j + \Lambda_{t_j}^1 g_{bj},
\end{align*}
where $\Lambda_{t_j}^k$, $k=1,2$,
 is a $k^{th}$ order tangential (to 
 $\partial\Omega_j$) operator.
 Applying a tangential smoothing operator to
  both sides and integrating yields
\begin{align*}
 \|\partial_{\rho_j}u
 \big|_{\partial\Omega_j}
 \|_{W^{-\infty}
 	(\partial \Omega \cap \partial\Omega_j)}
  \lesssim& \|u\|_{L^2(\Omega)}
  + \|g_b\|_{L^2(\partial\Omega)}\\
  \lesssim& \|g_b\|_{L^2(\partial\Omega)}.
\end{align*}

We thus have
\begin{equation*}
\label{inftyEst}
\| R^{-\infty} \|_{W^{\alpha,s}(\Omega, \rho,k)} 
\lesssim  \|g_b\|_{L^2(\partial\Omega)}.
\end{equation*}
Similarly, we have
\begin{align*}
 \| R_{bj}^{-\infty} \|_{W^{\alpha,s}
	(\partial \Omega \cap \partial\Omega_j, 
	\rho_{\hat{j}},k)} 
\lesssim& \|g_b\|_{L^2(\partial\Omega)}
.
\end{align*}

To obtain our expression for 
 the Poisson solution, we
 assume $u$ is supported in 
a small neighborhood of the origin in
 $\Omega$ (we can multiply
  the solution with a smooth cutoff function)
 with boundary values
(as non-tangential limits) given by
 $g_b$ (also with compact support),
  and study the operator
$2\square$ applied to 
 (the cutoff multiplied by) $u$.
 
We use extensions by zero to consider 
 $\square u $ on all of $\mathbb{R}^n$.
Similarly the result of 
 other operators applied to $u$ will be extended
by zero when taking Fourier Transforms.  
Thus, writing 
 $g_{b}:= u|_{\partial\Omega}$ and
 $g_{bj}:= u_{bj}$, we have
\begin{align*}
	\nonumber \widehat{
\frac{\partial^2 u}{\partial \rho_j^2}}
 (\eta,\xi) =&
	\int\int_{-\infty}^0 
\frac{\partial^2 u}{\partial \rho^2_j} 
	e^{-i\rho_j\eta_j}d\rho_j
 e^{-i\rho_{\hat{j}} \cdot
  \eta_{\hat{j}}}e^{-ix\xi} 
d\rho_{\hat{j}} dx\\
	\nonumber
=& F.T._{\hat{j}}
 \frac{\partial u}{\partial \rho_j}
	(0_j,\eta_{\hat{j}},\xi) +i\eta_j
 F.T._{\hat{j}}u (0_j,\eta_{\hat{j}},\xi)
	-\eta^2_j \widehat{u}(\eta,\xi)\\
	\label{2ndder}
=& \Bigg(\left.\frac{\partial u}
 {\partial \rho_j}\right|_{\rho_j=0}
	\times \delta(\rho_j) \widehat{\Bigg)}
+i\eta_j \Bigg( g_{bj}\times
  \delta(\rho_j)\Big)\widehat{\Bigg)} -\eta_j^2
	\widehat{u}(\eta, \xi),\\
	\nonumber
\widehat{\frac{\partial u}
	 {\partial \rho_j}}(\eta, \xi)
=&F.T._{\hat{j}} u (0_j,\eta_{\hat{j}},\xi)
  +i\eta_j \widehat{u}(\eta,\xi)\\
\nonumber =& \Big( g_{bj}\times
 \delta(\rho_j)\widehat{\Big)} +
	i\eta_j \widehat{u}(\eta,\xi).
\end{align*}
 Above, $u$ is also extended by zero
  in writing 
 Fourier Transforms (and partial transforms) of 
  $u$.
For instance, a zero order 
pseudodifferential 
 operator acting on $\frac{\partial u}
  {\partial{\rho_j}}$ 
  can
be written in the form
\begin{equation}
 \label{pseudoOnDer}
\Psi^1 u + 
 \Psi^0 \left(g_{bj} \times \delta_j \right),
\end{equation}
where $\delta_j:= \delta(\rho_j)$.

We consider the above expressions in the sense of
 distributions (see for example \cite{BN08}
for details of 
 the Dirichlet problem in the sense of distributions), and one of our first tasks in obtaining estimates for
the solution will be to obtain a formula and estimates
 in the sense of distributions for the boundary 
  values of the derivative.

We recall the expressions in local coordinates
 of the vector fields dual to the
$\omega_j$ forms to write their symbols 
 according to the convention that
$\eta_j$ is the variable dual to
 $ \rho_j$ for $1\le j \le m$,
and $\xi_j$ is the variable dual to
 $x_j$,
for $1\le j \le  2n-m$. 
The symbols of the 
 vector fields $L_j$ and $V_j$
 are given by
\begin{align*}
&\sigma(L_j) = \frac{1}{\sqrt{2}}
i\eta_{j} 
- \xi_{j}
 + O(\rho_j)
\qquad 1\le j \le m\\
& \sigma
\left(
V_{m+j}
\right)
 = \frac{1}{2}\left(
  i\xi_{m+2j-1}
  +\xi_{m+2j}\right)
 +O(x)  +\sum_k O(\rho_k)
 \qquad 1\le j \le n-m
.
\end{align*} 
  We then have
  as principal symbol
of the second order operator in 
Proposition \ref{squareOp}
\begin{align}
\nonumber
-\sigma_2&\left(L_1\Lb_1 + \cdots +
L_{m}\Lb_{m}
+ V_{m+1}\Vb_{m+1} + \cdots V_n\Vb_n\right)\\
\nonumber
 &\qquad = \frac{1}{2} 
\left(
	\eta_1^2 + \cdots + \eta_m^2  
\right)
 + \xi_{1}^2 + \cdots +\xi_{m}^2
 + \frac{1}{4} \left(
 \xi_{m+1}^2 + \cdots +\xi_{2n-m}^2 
\right)\\
 \label{symSec}
&\qquad \quad
+
 \sum l_{jk}(x)
\xi_j
 \xi_k +\sum_j O(\rho_j),
\end{align} 
where 
 $l_{jk}$ is $O(x)$.

We use the vector notation of forms
 where each component of a $(0,q)$-form
corresponds to an entry of an
 $\binom{n}{q}$ vector.  Symbols and
differential operators will accordingly be 
 matrices.  Thus, a symbol such as
$\sigma_2\left(
 L_1\Lb_1 + \cdots + V_n\Vb_n
\right)$ above is a matrix with
 diagonal
 entries given by
the right hand side of \eqref{symSec}.

 We also use the notation
$\partial_{\rho_j}u$ to denote
\begin{equation*}
\frac{\partial}{\partial \rho_j}
\left( \sum_{|K|=q} u_K\omegab_K 
\right)
\end{equation*}
written in vector notation.
 For shorthand notation, we use
$\partial_{\rho_j} : = 
 \frac{\partial}{\partial\rho_j}$, with
a similar notation holding for the
 $x$ coordinates.

With the use of the expressions in local
 coordinates given in 
 \eqref{lpt} and
 \eqref{vxx}, we use
Proposition \ref{squareOp} to write
 $2\square u = 0$ in the local
 form
(see also \cite{Eh18_dno})
\begin{align}
\nonumber
-
 \bigg[\partial_{\rho_1}^2
 +& \cdots + \partial_{\rho_m}^2
 + 2\Big(
\partial_{x_{1}}^2 + \cdots
 +  \partial_{x_{m}}^2
\Big)\\
\label{localSAT}
 &+ \frac{1}{2}\Big(
  \partial_{x_{m+1}}^2 + \cdots
  +  \partial_{x_{2n-m}}^2
 \Big)
+ 2 \sum_{ij} l_{ij}
 \partial_{x_i}\partial_{x_j}
\bigg]u
\\
\nonumber
& + \sqrt{2}\sum_{j=1}^{m} S_j
\circ(\partial_{\rho_j} u \times \delta_j)
+ A(u)+ \sum_{j=1}^m \rho_j \tau_j u = 0,
\end{align}
where the $S_k$ operators are 
 diagonal zero order operators, arising from the
$\partial_{\rho}$ components of 
\begin{equation}
\label{forS}
 	\sum_{k=1}^{m} 
 (-1)^{|J\cup \{ k\}|}  \left(
 c^{J\setminus \{k\}}_{J\setminus \{k\} \cup \{ k\}} L_k
 -\overline{c}^{J\setminus \{k\}}_{
 	J\setminus \{k\} \cup \{ k\}} \Lb_k
 \right) 
 +  d_k
 \Lb_k,
\end{equation}
$A$ is a matrix comprised of all
 first order tangential operators
  (tangential to all boundaries simultaneously;
  note that $L_i$ is orthogonal
 to $\partial_{\rho_j}$ for $i\neq j$
 as are the $V_i$ vector fields),
and $\tau_i$ are second order operators,
 arising from the $O(\rho)$ terms in 
$  L_1\Lb_1 + \cdots + V_n\Vb_n$.
 The relation in \eqref{localSAT} is
to be understood modulo smoothing terms which
 are arise due to the local cutoffs
introduced in order to study the problem
 locally.  Thus \eqref{localSAT} holds 
in a small neighborhood contained in the 
 support of $u$, modulo $R^{-\infty}$.

 For the purposes of the Poisson 
operator we will group the $O(\rho_j)$ terms
(the last summation on the left-hand side of
 \eqref{localSAT})
 with the principal second order operator.
 Then,
using the symbols for the vector fields as
 above, and using the notation
$g_{b}:= u|_{\partial\Omega}$, 
  we write \eqref{localSAT} as
\begin{align}
\nonumber
\frac{1}{(2\pi)^{2n}}	\int \bigg(
	&\eta_1^2 + \cdots + \eta_m^2  
+2\left( \xi_{1}^2 + \cdots +\xi_{m}^2
 \right)
	+ \frac{1}{2} \left(
	\xi_{m+1}^2 + \cdots +\xi_{2n-m}^2 
	\right)\\
\nonumber
&
+
2	\sum l_{jk}
	\xi_j \xi_k
 +\sum_j O(\rho_j) \bigg)  \widehat{u}(\eta,\xi)
	e^{i \rho \cdot \eta} e^{ix\xi} d\eta  d\xi\\
\nonumber
	&-
\frac{1}{(2\pi)^{2n}}	\sum_{j=1}^m \int
	\bigg(F.T._{\hat{j}} \partial_{\rho_j}
	u(0_j,\eta_{\hat{j}},\xi) +
i\eta_j 
 \widehat{g_{bj}} (\eta_{\hat{j}},\xi) 
  \bigg)
	e^{i \rho \cdot \eta} e^{ix\xi} d\eta  d\xi\\
\label{origPoissonFourier}
&
+ \sqrt{2}\sum_{j=1}^{m} S_j
(\partial_{\rho_j} u) +	A(u)  = 0.
\end{align}

Let us define the symbols
\begin{align*}
 \Xi(x,\rho,\xi) = \bigg(
 2\left( \xi_{1}^2 + \cdots \xi_{m}^2
 \right)
 +& \frac{1}{2} \left(
 \xi_{m+1}^2 + \cdots \xi_{2n-m}^2 
 \right)\\
 & +
2 \sum l_{jk}
 \xi_j \xi_k
+\sum_j O(\rho_j)
\bigg)^{1/2},
\end{align*}
where the $O(\rho_j)$ terms come from the 
 second order operator in
  \eqref{origPoissonFourier}.
 We further use the notation
\begin{equation*}
 \Xi_{bj}(x,\rho_{\hat{j}},\xi) 
  := \Xi (x,\rho,\xi) \big|_{\rho_j =0 }.
\end{equation*}
We will also write
\begin{align*}
&\eta^2 = \eta_1^2 + \cdots + \eta_m^2\\
&\eta_{\hat{j}}^2 = 
\eta_1^2 + \cdots +\eta_{j-1}^2
+ \eta_{j+1}^2+ \eta_m^2.
\end{align*}

For ease of notation we will omit 
 the delta distributions when applying
pseudodifferential operators to distributions 
 supported on the boundaries.  Thus, for 
$\phi_b \in  L^2(\partial\Omega)$, as a 
 shorthand 
notation we will write
\begin{align*}
& \Psi^{\alpha} \phi_b := 
     \sum_j\Psi^{\alpha}\left( \phi_{bj} \times \delta_j \right),\\
 & \Psi^{\alpha} \phi_{bj} := 
  \Psi^{\alpha}\left( \phi_{bj} \times \delta_j \right).
\end{align*}

In what follows we will make repeated use of 
 the fact that multiplying an elliptic
operator of negative order acting on a
 distribution supported on a boundary 
  $\partial\Omega_j$ with a factor
$\rho_j$ yields lower order terms; thus,
 for instance, with
$\Psi^{-s}$ for $s>0$ 
 denoting a generic pseudodifferential operator
  in the class $\Psi^{-s}(\Omega)$, we have
   for $k\ge 0$
\begin{equation}
 \label{lowerOrder}
 \rho_j^k \Psi^{-s} g_{bj}
  \equiv \Psi^{-s-k} g_{bj}.
\end{equation}
See \cite{Eh18_halfPlanes} for details.
 From the same reference we have the following
restriction property of pseudodifferential
 operators acting on distributions supported on
a boundary:
\begin{equation}
\label{restrictIncrease}
 R_j\circ \Psi^{-s} u_{bj}
  = \Psi^{-s+1}_{bj} u_{bj},
\end{equation}
 for $s\ge 2$, 
where $\Psi^{-s} \in\Psi^{-s}(\Omega)$,
  $\Psi^{-s}_{bj} \in\Psi^{-s}
   (\partial\Omega_j)$, and 
$u_{bj}$ is a distribution supported on
 $\partial\Omega_j$.
 
We recall from Theorem 4.1 in
\cite{Eh18_dno},
the principal operator,
 denoted by $\Theta^+_j$, of the
 Poisson operator on the 
 (smooth) domain $\Omega_j$
(but in a neighborhood of 
 the origin, in which  
$\rho_1,\ldots,\rho_m, x_{1},
 \ldots,x_{2n-m}$ forms a coordinate 
system) has
as 
 symbol
 \begin{equation}
\label{symlam}
\sigma(\Theta^+_j):= \frac{i}
{\eta_j + i
\sqrt{\eta_{\hat{j}}^2+ 
	\Xi_{bj}^2}}.
\end{equation}
The corresponding operator maps
 $W^{s}(\partial\Omega_j)$ into 
 $W^{s+1/2}(\Omega_j)$,
modulo operators which lead to 
 errors of more smooth type, i.e. which map
$W^{s_1}(\partial\Omega_j)
  \rightarrow W^{s_2}(\Omega_j)$ 
for $s_2 > s_1+1/2$.  As we will show, the 
 same operators (for each of the domains)
arise in the Poisson operator for the 
 intersection domain.

Applying an inverse of the 
 principal second order 
elliptic operator on the left-hand side of
\eqref{origPoissonFourier}
  to both sides of \eqref{origPoissonFourier},
and recalling that
 $ S_j
 \circ(\partial_{\rho_j} u)$
can be written as in
 \eqref{pseudoOnDer},
yields
	\begin{align}
\nonumber
	u=& 
	\sum_{j=1}^m \frac{1}{(2\pi)^{2n}}\int \frac{
\partial_{\rho_j}   F.T._{\hat{j}}u (0_j,\eta_{\hat{j}},\xi)
  + i\eta_j \widehat{g_{bj}} (\eta_{\hat{j}},\xi)}
{\eta^2+ \Xi^2(x,\rho,\xi)}
	e^{i \rho \cdot \eta} e^{ix\xi} d\eta  d\xi\\
\label{errorPoisson}
	& + 
\Psi^{-3}\left(\partial_{\rho} u \big|_{\partial\Omega}
 \right) +\Psi^{-2} \left(g_b \right)+ \Psi^{-1}u.
	\end{align}
To be precise we can rewrite 
 \eqref{origPoissonFourier} 
 by adding term, $u$, to both sides and
then invert the operator with symbol
 $1+\eta^2+\Xi^2$ in order to avoid complications
with zeros in the denominators of the symbols of
 inverse operators.  If we consider then 
a resulting term of the form
\begin{equation*}
 \int \frac{
 	\widehat{h_{bj}} (\eta_{\hat{j}},\xi)}
 {1+\eta^2+ \Xi^2(x,\rho,\xi)}
 e^{i \rho \cdot \eta} e^{ix\xi} d\eta  d\xi,
\end{equation*}
we can 
integrate over the 
 $\eta_j$ variable using the residue 
at $\eta_j = i\sqrt{1+\eta_{\hat{j}}^2+\Xi^2}$.
 Alternatively, set
$\chi_j (\eta_{\hat{j}},\xi) \in C^{\infty}_0
 (\mathbb{R}^{2n-1})$ such that
$\chi_j \equiv 1$ near the origin, and set
 $\chi_j' = 1-\chi_j$.
 We could then use an expansion
\begin{align*}
 \frac{1}{1+\eta^2+ \Xi^2}
   =& \frac{\chi_j(\eta_{\hat{j}},\xi)}
    {1+\eta^2+ \Xi^2}
    + \frac{\chi_j'(\eta_{\hat{j}},\xi)}
    {1+\eta^2+ \Xi^2}\\
   =& \frac{\chi_j(\eta_{\hat{j}},\xi)}
   {1+\eta^2+ \Xi^2}
   + \frac{\chi_j'(\eta_{\hat{j}},\xi)}
   {\eta^2+ \Xi^2} + \cdots,
\end{align*}
where the remainder terms are symbols in class
 $\lrs^{-3}(\Omega)$.
The term
\begin{equation}
\label{smthCutoff}
\int \chi_j(\eta_{\hat{j}},\xi) \frac{
	\widehat{h_{bj}} (\eta_{\hat{j}},\xi)}
{1+\eta^2+ \Xi^2(x,\rho,\xi)}
e^{i \rho \cdot \eta} e^{ix\xi} d\eta  d\xi
\end{equation}
is smoothing (on $\partial\Omega_j$), which can be seen by integrating 
 over the $\eta_j$ variable, using the residue calculus,
while the term
\begin{equation*}
\int \chi_j'(\eta_{\hat{j}},\xi) \frac{
	\widehat{h_{bj}} (\eta_{\hat{j}},\xi)}
{\eta^2+ \Xi^2(x,\rho,\xi)}
e^{i \rho \cdot \eta} e^{ix\xi} d\eta  d\xi
\end{equation*}
can also be analyzed using the residue calculus
 without any resulting singular terms.  
 We will implicitly adopt this approach in what follows, 
but for simplicity we will omit the $\chi_j'$ factors.
 In our use of symbols which are singular at the origin,
we can use the above approach to reduce the application of
 such symbols to distributions which vanish at the
  singularities.  

	We now return to 
\eqref{errorPoisson}.  
Expanding $\Xi^2(x,\rho,\xi)$ in each
 $\rho_j$, we get 
	\begin{align}
\nonumber
	u=& 
	\sum_{j=1}^m \frac{1}{(2\pi)^{2n}} \int \frac{
  F.T._{\hat{j}}  \partial_{\rho_j} 
	u(0_j,\eta_{\hat{j}},\xi) +
i\eta_j \widehat{g_{bj}}
 (\eta_{\hat{j}},\xi)}
 {\eta^2 +\Xi_{bj}^2}
	e^{i \rho \cdot \eta} e^{ix\xi} d\eta  d\xi\\
\nonumber
	& + \Psi^{-1} u +
\sum \rho_j
 \Psi^{-2}\left(\partial_{\rho_j}
  u \big|_{\partial\Omega_j}
 \right) +
\sum \rho_j \Psi^{-1} \left(g_{bj} \right)\\
\nonumber
&+
	\Psi^{-3}\left(\partial_{\rho}
	 u \big|_{\partial\Omega}
	  \right) +\Psi^{-2} \left(g_{b}\right).
	\end{align}

Now, using the property stated 
 in 
\eqref{lowerOrder} above
of 
$\rho_j$ multiplied with an elliptic operator,
we can solve for $u$ and get
\begin{align}
\nonumber
u=& 
\sum_{j=1}^m \frac{1}{(2\pi)^{2n}} \int \frac{
	F.T._{\hat{j}} \partial_{\rho_j}   
	u(0_j,\eta_{\hat{j}},\xi) +
	i\eta_j \widehat{g_{bj}}
 (\eta_{\hat{j}},\xi)}
{\eta^2 + \Xi_{bj}^2}
e^{i \rho \cdot \eta} e^{ix\xi} d\eta  d\xi\\
& 
 \label{vPoissonBasic}
+
\Psi^{-3}\left(\partial_{\rho}
u \big|_{\partial\Omega}
\right) +\Psi^{-2} \left(g_{b}\right)+
 \Psi^{-\infty} u.
 \end{align}

Taking limits as
	$\rho_k \rightarrow 0^+$, 
and using \eqref{restrictIncrease}, we get
\begin{align*}
	0=& 
	\frac{1}{(2\pi)^{2n}} \int \frac{
 F.T._{\hat{k}} \partial_{\rho_k}
  u(0_k,\eta_{\hat{k}},\xi) +
i\eta_k \widehat{g_{bk}}
  (\eta_{\hat{k}},\xi)}
{\eta^2 +
		\Xi_{bk}^2}
	e^{i \rho_{\hat{k}} \cdot \eta_{\hat{k}}}
 e^{ix\xi} d\eta  d\xi\\
& + \frac{1}{(2\pi)^{2n}}
 \sum_{j\neq k}  \int \frac{
  F.T._{\hat{j}}\partial_{\rho_j} 
  u (0_j,\eta_{\hat{j}},\xi) +
		i\eta_j \widehat{g_{bj}}
(\eta_{\hat{j}},\xi)}
 {\eta^2 +
		\Xi_{bj}^2}
	e^{i \rho_{\hat{k}} \cdot \eta_{\hat{k}}}
 e^{ix\xi} d\eta    d\xi\\
	&  
	 +
 \Psi^{-2}_{bk}\left(\partial_{\rho_k} 
 u \big|_{\partial\Omega_k} 
 \right)
+
\Psi^{-1}_{bk}\left(g_{bk}\right)
 + R_{bk}^{-\infty}\\
 &+\sum_{j\neq k} R_k \circ 
  \Psi^{-3}\left(\partial_{\rho_j}
u \big|_{\partial\Omega_j} \right) +
 \sum_{j\neq k}
 R_k \circ \Psi^{-2} \left(g_{bj} \right).
\end{align*}

	We perform
integrations in the $\eta$ variables
  in the integrals in \eqref{vPoissonBasic}
 for $\rho_k>0$ 
  above and then let $\rho_k\rightarrow 0^+$:
\begin{align}
\nonumber
	0=& 
\frac{i}{(2\pi)^{2n-1}} \int \frac{
		   F.T._{\hat{k}} \partial_{\rho_k}
 u(0_k,\eta_{\hat{k}},\xi) -
		\sqrt{\eta_{\hat{k}}^2+ 
\Xi_{bk}^2}  \widehat{g_{bk}}
 (\eta_{\hat{k}},\xi)}
{2i\sqrt{\eta_{\hat{k}}^2+ 
		 \Xi_{bk}^2} }
	e^{i \rho_{\hat{k}} \cdot \eta_{\hat{k}}} e^{ix\xi} 
	d\eta_{\hat{k}}  d\xi\\
\nonumber
	& + 
\frac{i}{(2\pi)^{2n-1}}
 \sum_{j\neq k}  \int \frac{   
 	 F.T._{\hat{j}} \partial_{\rho_j} 
u(0_j\eta_{\hat{j}},\xi)
	+
	\sqrt{\eta_{\hat{j}}^2+
		 \Xi_{bj}^2} 
 \widehat{g_{bj}}(\eta_{\hat{j}},\xi)}
	{2i\sqrt{\eta_{\hat{j}}^2+ 
			\Xi_{bj}^2}} \times\\
\nonumber
&\qquad \qquad \qquad \qquad \qquad \qquad
\qquad \qquad\qquad \qquad
e^{ \rho_{j} \sqrt{
 \eta_{\hat{j}}^2+ \Xi_{bj}^2}}	
	e^{i \rho_{\hat{k}} \cdot \eta_{\hat{k}}} e^{ix\xi} 
	d\eta_{\hat{j}}    d\xi\\
\nonumber
	&  
+
\Psi^{-2}_{bk}\left(\partial_{\rho_k} u \big|_{\partial\Omega_k} \right)
+
\Psi^{-1}_{bk}\left(g_{bk}  \right)
+ R_{bk}^{-\infty}\\
 \label{intCross}
&+\sum_{j\neq k} R_k \circ 
\Psi^{-3}\left(\partial_{\rho_j}
u \big|_{\partial\Omega_j} \right) +
\sum_{j\neq k}
R_k \circ \Psi^{-2} \left(g_{bj} \right).
\end{align}

 We note that the terms,
\begin{equation*}
 \int \frac{
 	 F.T._{\hat{j}}\partial_{\rho_j} u
 	   (0_j, \eta_{\hat{j}},\xi)
 	+
 	\sqrt{\eta_{\hat{j}}^2+
 		\Xi_{bj}^2} 
 	\widehat{g_{bj}}(\eta_{\hat{j}},\xi)}
 {2i\sqrt{\eta_{\hat{j}}^2+ 
 		\Xi_{bj}^2}}
 e^{ \rho_{j} \sqrt{\eta_{\hat{j}}^2+ \Xi_{bj}^2}}	
 e^{i \rho_{\hat{k}} \cdot \eta_{\hat{k}}} e^{ix\xi} 
 d\eta_{\hat{j}}    d\xi
 \end{equation*}
can be thought of as mappings from 
 distributions on 
$\partial\Omega_j$ to distributions on
  $\partial\Omega_k$.  Using our notation
from Section \ref{secNotation},
  we will write the operators defined by
\begin{equation*}
  h|_{\rho_j=0}  
  \mapsto
  \int \frac{
  	h(0_j,\eta_{\hat{j}},\xi)	
  }
  {\sqrt{\eta_{\hat{j}}^2+ \Xi_{bj}^2}}
  e^{ \rho_{j} \sqrt{\eta_{\hat{j}}^2+ 
  		\Xi_{bj}^2}}	
  e^{i \rho_{\hat{k}} \cdot \eta_{\hat{k}}} e^{ix\xi} 
  d\eta_{\hat{j}}    d\xi
\end{equation*}
  as 
 $\lre^{jk}_{-3/2}\left(h|_{\rho_j=0} \right)$.
We also note the 
 $ \Psi^{-3}$ 
 and $ \Psi^{-2}$ operators stem from a 
symbol expansion of the inverse operator
 to the principal operator on the left-hand side of 
\eqref{origPoissonFourier},
 and so are decomposable.
Then, from Section \ref{secNotation}, we 
 have
\begin{align}
\label{min52}
& R_k \circ 
 \Psi^{-3}\left(\partial_{\rho_j}
 u \big|_{\partial\Omega_j}\right)
 = \lre^{jk}_{-5/2}\left(
 \partial_{\rho_j} u \big|_{\partial\Omega_j}
 \right)\\
 \nonumber
& 
 R_k \circ \Psi^{-2} \left(g_{bj} \right)
= \lre^{jk}_{-3/2}\left( g_{bj}
\right),
\end{align}
for $j\neq k$.

We define an operator,
$\Gamma^{\sharp}$ by the symbol
\begin{equation*}
\sigma\left(
\Gamma^{\sharp}
\right)
= \frac{1}{\eta^2+\Xi^2(x,\rho,\xi)},
\end{equation*}
and $\Gamma^{\sharp}_{j} = 
 \Gamma^{\sharp} \Big|_{\rho_j=0}$.
Let us also define the operators 
 $|D_{bj}|$ by the symbols
\begin{equation*}
\sigma\left(
 |D_{bj}|
\right) = 
 \sqrt{\eta_{\hat{j}}^2+ \Xi_{bj}^2}.
\end{equation*}
From \eqref{intCross}
we can now write
\begin{align*}
	0=& \frac{1}{2}\frac{1}{(2\pi)^{2n-1}} \int \frac{
 F.T._{\hat{k}}
 \partial_{\rho_k} u(0_k,\eta_{\hat{k}},\xi)}
{\sqrt{\eta_{\hat{k}}^2+ \Xi_{bk}^2} } 
	e^{i \rho_{\hat{k}} \cdot \eta_{\hat{k}}}
 e^{ix\xi} d\eta
	d\xi\\
&- \frac{1}{2}
 \frac{1}{(2\pi)^{2n-1}} \int 
  g_{bk}(\eta_{\hat{k}},\xi)
e^{i \rho_{\hat{k}} \cdot \eta_{\hat{k}}}
  e^{ix\xi} d\eta\\
	&+ \sum_{j\neq k}
R_k \circ \Gamma^{\sharp}_{j}
 \left(
 \partial_{\rho_j} u 
\big|_{\partial\Omega_j}
 + |D_{bj}|g_{bj}
 \right) \\
&+\sum_{j\neq k}
\lre_{-5/2}^{jk}\left(\partial_{\rho_j} u
\big|_{\partial\Omega_j} \right)
+\sum_{j\neq k} \lre_{-3/2}^{jk}\left(g_{bj} \right)
	\\	&  
	+
\Psi^{-2}_{bk}\left(\partial_{\rho_k} u 
 \big|_{\partial\Omega_k}\right)
	+
\Psi^{-1}_{bk}\left(g_{bk} \right)
	+ R_{bk}^{-\infty}.
\end{align*}

We now 
solve for 
$\partial_{\rho_k}u(0_k,\rho_{\hat{k}},\xi)$
 by inverting the operator
with symbol
$1/ \left(2 
 \sqrt{\eta_{\hat{k}}^2+ \Xi^2_{bk}} \right)$.
Note that 
 the $\lre_{-\alpha}^{jk}$ terms above
 are of the form $R_k \circ A_{-(\alpha+1/2)}$
 where $A_{-(\alpha+1/2)}$ is decomposable.  Thus,
$|D_{bj}| \circ \lre_{-\alpha}^{jk}
 = \lre_{-\alpha+1}^{jk}$
for the  $\lre_{-\alpha}^{jk}$ above.  
We have
\begin{align}
\nonumber
	\partial_{\rho_k}   u \big|_{\rho_k =0}
	=& |D_{bk}| g_{bk}
-2 \sum_{j\neq k} 
 |D_{bk}| \circ R_k \circ \Gamma^{\sharp}_j
\left(
\partial_{\rho_j} u 
\big|_{\partial\Omega_j}
+ |D_{bj}|g_{bj}
\right)
\\
\nonumber
& 
 +\sum_{j\neq k} \lre_{-3/2}^{jk} 
 \left(
 \partial_{\rho_j} u \big|_{\partial\Omega_j}
 \right) 
 + \sum_{j\neq k}
\lre_{-1/2}^{jk}  g_{bj}\\
 \label{derivToIterate}
&
	+
\Psi^{-1}_{bk}\left(\partial_{\rho_k} u 
\big|_{\partial\Omega_k}\right)
+
\Psi^{0}_{bk}\left(g_{bk} \right)
+ R_{bk}^{-\infty} .
\end{align}

We now iterate 
\eqref{derivToIterate}
to get
\begin{align}
\nonumber
\partial_{\rho_k}   u \big|_{\rho_k =0}
=& |D_{bk}| g_{bk}
-4 \sum_{j\neq k} |D_{bk}|\circ R_k \circ \Gamma^{\sharp}_j
\circ
|D_{bj}|g_{bj}\\
\nonumber
& + 8 \sum_{{j\neq k} \atop 
	{l\neq j}}|D_{bk}|\circ R_k \circ 
\Gamma^{\sharp}_{j}
\circ |D_{bj}| \circ 
R_j \circ \Gamma^{\sharp}_l\circ 
|D_{bl}| g_{bl}
\\
\nonumber
& 
+
\Psi^{0}_{bk}\left(g_{bk}\right)
+\sum_j \lre_{-1/2}^{jk} 
g_{bj}\\
\label{derWithZeroIt}
&+ \Psi^{-1}_{bk}\left(
\partial_{\rho_k} u \big|_{\partial\Omega_k} \right)
+ \sum_j \lre_{-3/2}^{jk} 
\left(
\partial_{\rho_j} u \big|_{\partial\Omega_j}
\right) 
+R^{-\infty}_{bk} .
\end{align}

Note that
$2 \Gamma^{\sharp}_j
\circ
|D_{bj}| \circ R_j
\equiv 
\Theta^+_j\circ R_j$, with
$\Theta^+_j$ defined as in 
\eqref{symlam}.  
We thus have
from \eqref{derWithZeroIt}
\begin{align}
\nonumber
\partial_{\rho_k}   u \big|_{\rho_k =0}
=& |D_{bk}| g_{bk}
- 2\sum_{j\neq k} 
|D_{bk}|\circ R_k \circ \Theta_j^+ g_{bj}\\
\nonumber
&+ 2
\sum_{{j\neq k} \atop 
	{l\neq j}} |D_{bk}|\circ R_k \circ 
\Theta_j^+ \circ R_j \circ 
\Theta_l^+ g_{bl}
\\
\nonumber
&
+
\Psi^{0}_{bk}\left(g_{bk}\right)
+\sum_j \lre_{-1/2}^{jk} 
g_{bj}\\
\label{withTheta}
&+ \Psi^{-1}_{bk}\left(
\partial_{\rho_k} u \big|_{\partial\Omega_k} \right)
+ \sum_j \lre_{-3/2}^{jk} 
\left(
\partial_{\rho_j} u \big|_{\partial\Omega_j}
\right) 
+R^{-\infty}_{bk}.
\end{align}

When estimating the term
 $\partial_{\rho_k}   u \big|_{\rho_k =0}$
it suffices to write the first two sums on the
 right-hand side simply as a summation of terms
  of the form
$|D_{bk}|\circ \lre_{-1/2}^{jk}g_{bj}$:
\begin{align}
\nonumber
\partial_{\rho_k}   u \big|_{\rho_k =0}
=& |D_{bk}| g_{bk}
+ \sum_{j} |D_{bk}|\circ \lre^{jk}_{-1/2} g_{bj}
+
\Psi^{0}_{bk}\left(g_{bk}\right)
+\sum_j \lre_{-1/2}^{jk} 
g_{bj}\\
\label{derWithZeroEst}
&+ \Psi^{-1}_{bk}\left(
\partial_{\rho_k} u \big|_{\partial\Omega_k} \right)
+ \sum_j \lre_{-3/2}^{jk} 
\left(
\partial_{\rho_j} u \big|_{\partial\Omega_j}
\right) 
+R^{-\infty}_{bk} .
\end{align}

The expression \eqref{withTheta} 
 for the normal derivatives
leads to an expression
 for the solution, $u$, in 
\eqref{vPoissonBasic}.  
Recall from our convention in
 Section \ref{secNotation} that
the boundary operator,
$\Psi_{bk}^{-1}$, when acting
on $\partial_{\rho_k}u|_{\rho_k=0}$
above can be written
$\lre_{-1}^{kk}$:
\begin{align*}
\Psi^{-1}_{bk}
\left( \partial_{\rho_k} u \big|_{\partial\Omega_k}
\right)
=& \lre^{kk}_{-1} \left(
\partial_{\rho_k} u \big|_{\partial\Omega_k}
\right).
\end{align*}
 From
 \eqref{vPoissonBasic}, we thus have
 the expression for the Poisson solution
as
\begin{align}
\nonumber
u
=& \sum_j  \Theta^+_j g_{bj}
+ \sum_{j, k}\Psi^{-1} \circ \lre_{-1/2}^{kj} g_{bk}
 +
\sum_{j} \Psi^{-2}
g_{bj}\\
 \label{vFormDNO}
&+
\sum_{j,k} \Psi^{-2}
\circ \lre_{-1}^{kj} 
\left(
\partial_{\rho_k} u \big|_{\partial\Omega_k}
\right) 
+
\Psi^{-3}\left(\partial_{\rho}
u \big|_{\partial\Omega}\right) 
+ R^{-\infty}.
\end{align}
In addition,
 \eqref{derWithZeroEst}
 above gives an expression
for $\partial_{\rho_j} u|_{\partial\Omega_j}$ 
 (recall the boundary values are to be understood
as non-tangential limits to the boundary).
We note the above relation for future use.
 We also note that all the 
$\Psi^{-2}$ and $\Psi^{-3}$ operators are
 decomposable, as they arise from
the inverse to the Laplacian.

We can now derive (weighted) estimates for
 the Poisson solution from
\eqref{vPoissonBasic};
we show
\begin{thrm}
\label{thrmWeightedPoisson}
Let $u$ be the solution to the homogeneous Dirichlet
 problem \eqref{homoDirProb} with boundary data 
$g_b$ satisfying
$g_{bj} \in L^{2}(\partial\Omega\cap\partial\Omega_j,
\rho_{\hat{j}},\lambda)$ for some
 $\lambda \ge 0$ and for all 
$1\le j \le m$.  Then
\begin{equation*}
  \left\|
u \right\|_{W^{1/2,s}
	\left( \Omega, \rho, \lambda
	\right)}
 \lesssim \sum_j \left\|
 g_{bj} 
 \right\|_{W^{0,s}
 	\left(
 	\partial\Omega_j\cap \partial\Omega,
 	\rho_{\hat{j}}, \lambda 
 	\right) }.
\end{equation*}
\end{thrm}
\begin{proof}
Weighted estimates for
 $	\partial_{\rho_j}   u \big|_{\rho_j =0}$
can be read from \eqref{derWithZeroEst}.
 From 
Theorem 5.6 of \cite{JK95} 
 (see also \cite{JK81}, Theorem 5.1 in \cite{Ve84}, and
  Theorem 1.4.3 in \cite{Sh05})
  we have that
$\partial_{\rho} u |_{\partial\Omega} \in 
  L^{2}(\partial\Omega)$ in the case
 $g_b \in W^1(\partial\Omega)$, 
  and the first sub goal 
of the proof is to extend these
  estimates for $g_b\in W^{\gamma}
  (\partial\Omega)$,
 with $0\le \gamma\le 1$,  using
\eqref{derWithZeroEst} and 
Theorem \ref{weightedLRE}.
 We note the $R_{bk}^{-\infty}$ term stems from
  terms of the form $R_k\circ \Psi^{-\infty}$ 
 (see \eqref{vPoissonBasic}) in addition
  to any terms of the form
\eqref{smthCutoff}  
  resulting from our handling of the singularities in the
inverse to the Laplacian operator.
 The former can be 
estimated by 
 $\|u\|_{-\infty}:= \|u\|_{W^{-\infty}(\Omega)}$, while the
  latter by
 $\left\|
 \partial_{\rho}   u \big|_{\partial\Omega}
 \right\|_{-\infty}+ \|
 g_b \|_{-\infty}$ 
We have, for $-1 \le \beta \le 0$,
\begin{align*}
\left\|
\partial_{\rho_j}   u \big|_{\partial\Omega_j}
\right\|&_{W^{\beta,s}
	\left(
	\partial\Omega_j\cap \partial\Omega,
	\rho_{\hat{j}}, \lambda
	\right)
}\\
\lesssim&
\left\|
g_{bj}
\right\|_{W^{1+\beta,s}
	\left(
	\partial\Omega_j\cap \partial\Omega,
	\rho_{\hat{j}}, \lambda
	\right)}+\sum_{k}\left\|
g_{bk} 
\right\|_{W^{1/2+\beta,s}
	\left(
	\partial\Omega_k\cap \partial\Omega,
	\rho_{\hat{k}}, \lambda
	\right) }\\
&+\sum_{k }\left\|
\partial_{\rho_k}   u \big|_{\partial\Omega_k}
\right\|_{W^{\beta-1/2,s}
	\left(
	\partial\Omega_k\cap \partial\Omega,
	\rho_{\hat{k}}, \lambda
	\right) }\\
& + \| u \|_{-\infty}
+\left\|
\partial_{\rho}   u \big|_{\partial\Omega}
\right\|_{-\infty}+ \|
g_b \|_{-\infty}.
\end{align*}
In applying weighted estimates to \eqref{derWithZeroEst}
 we use 
\begin{equation*}
 \big\| |D_{bk}| h_{bk} \big\|_{W^{\beta,s}(\partial\Omega_k
 	\cap\partial\Omega,\rho_{\hat{k}},\lambda)}
  \simeq \big\| h_{bk} \big\|_{W^{1+\beta,s}(\partial\Omega_k
  	\cap\partial\Omega,\rho_{\hat{k}},\lambda)},
\end{equation*}
 for a distribution 
$h_{bk}\in L^{2}(\partial\Omega\cap\partial\Omega_k,
\rho_{\hat{k}},\lambda)$,
which follows by the product rule of 
 differentiation.
 Then a direct application
 of Theorems \ref{weightedLRE}, \ref{thrmWghtdDecomp},
  and \ref{PsiInt} yields the inequality.

Then summing over $j$
 and bringing lower order estimates of boundary
 values of derivatives to the left-hand side yields 
\begin{align}
\nonumber
\sum_j \left\|
\partial_{\rho}   u \big|_{\partial\Omega_j}
\right\|_{W^{\beta,s}
	\left(
	\partial\Omega_j\cap \partial\Omega,
	\rho_{\hat{j}}, \lambda
	\right) }
\lesssim& 
\sum_j \left\|
g_{bj}
\right\|_{W^{1+\beta,s}
	\left(
	\partial\Omega_j\cap \partial\Omega,
	\rho_{\hat{j}}, \lambda
	\right)}\\
\label{weightedDvBeta}
&+ \| u \|_{-\infty}+\left\|
\partial_{\rho}   u \big|_{\partial\Omega}
\right\|_{-\infty} 
 + \|
 g_b \|_{-\infty}.
\end{align}
In particular,
\begin{equation}
\label{weightedDv}
 \left\|
 \partial_{\rho_j}   u \big|_{\partial\Omega_j}
 \right\|_{W^{-1,s}
 	\left(
 	\partial\Omega_j\cap \partial\Omega,
 	\rho_{\hat{j}}, \lambda
 	\right) }
 \lesssim
\sum_j \left\|
 g_{bj}
 \right\|_{W^{0,s}
 	\left(
 	\partial\Omega_j\cap \partial\Omega,
 	\rho_{\hat{j}}, \lambda
 	\right)}
.
\end{equation}

If we use \eqref{derWithZeroEst} 
(solving first for
$\partial_{\rho_k}u\big|_{\rho_k=0}$ )
in 
 \eqref{vFormDNO}, we get the expression
\begin{align*}
\nonumber
u
=& \sum_j \Psi^{-1} g_{bj}
+ \sum_{j, k}\Psi^{-1} \circ \lre_{-1/2}^{kj} g_{bk}
+
\sum_{j,k} \Psi^{-2}
\circ \lre_{-1}^{kj} 
\left(
\partial_{\rho_k} u \big|_{\partial\Omega_k}
\right) 
,
\end{align*}
modulo smoothing terms.
The $\Psi^{-2}$ and $\Psi^{-1}$
operators are decomposable so 
 Theorem \ref{weightedPsi} applies.  
Using 
 \eqref{weightedDv} 
and the above 
expression as well as
the estimates of Theorem \ref{weightedPsi},
we can conclude the
 estimates
\begin{align*}
  \left\|
u  \right\|_{W^{1/2,s}
 	\left( \Omega, \rho, \lambda
 	\right)}
 \lesssim& 
\sum_j \left\|
g_{bj} 
\right\|_{W^{0,s}
	\left(
	\partial\Omega_j\cap \partial\Omega,
	\rho_{\hat{j}}, \lambda
	\right) }  
\\
&+
\sum_j
\left\|
\partial_{\rho_j}   u \big|_{\rho_j =0}
\right\|_{W^{-1,s}
	\left(
	\partial\Omega_j\cap \partial\Omega,
	\rho_{\hat{j}}, \lambda
	\right)}\\
&+ \| u \|_{-\infty}+\left\|
\partial_{\rho}   u \big|_{\partial\Omega}
\right\|_{-\infty} 
+ \|
g_b \|_{-\infty}\\
\lesssim& \sum_j \left\|
g_{bj} 
\right\|_{W^{0,s}
	\left(
	\partial\Omega_j\cap \partial\Omega,
	\rho_{\hat{j}}, \lambda
	\right) } .
\end{align*} 
\end{proof}
We remark that higher order estimates,
 for instance
\begin{equation*}
\left\|
u \right\|_{W^{1,s}
	\left( \Omega, \rho, k
	\right)}
\lesssim \sum_j \left\|
g_{bj} 
\right\|_{W^{1/2,s}
	\left(
	\partial\Omega_j\cap \partial\Omega,
	\rho_{\hat{j}}, k 
	\right) },
\end{equation*}
follow by taking
 weighted Sobolev $1$
  estimates from the form
\begin{equation*}
u=
\sum_{j}
\Psi^{-2} 
\left(
\partial_{\rho_j} u \big|_{\partial\Omega_j}
\right) +
\sum_j \Psi^{-1}( g_{bj} )
\end{equation*}
from
\eqref{vPoissonBasic}.
 However it is the estimates with
base-level (by which we mean
 the Sobolev level whereby $s=0$)
equal to $1/2$ of the Poisson solution
 which we will use
 for our Main Theorem.

In particular, if the boundary data is
in $W^{s}(\partial\Omega)$,
 for instance as the restriction of 
  a function in $W^{s+1/2}$
 in some neighborhood
of $\cup_j \Omega_j$
 to each piece of the boundary, 
 then
the solution $u$ can be estimated by
\begin{thrm}
\label{specCase}
	\begin{equation*}
	\| u
	\|_{W^{1/2,s}(\Omega,\rho)}
	\lesssim \sum_j 
	\|g_{bj}\|_{W^{s}
		(\partial\Omega)}.
	\end{equation*}
\end{thrm}

Compare the estimates in 
 Theorem \ref{thrmWeightedPoisson}
to the estimates in \cite{JK95};
 Sobolev $\alpha$ estimates are concluded in
\cite{JK95}, where it is shown
 $\|u\|_{W^{\alpha}(\Omega)} \lesssim \|g_b\|_{W^{\alpha-1/2}(\partial\Omega)}$
for $1/2\le\alpha\le3/2$.

We also note for future reference 
 an extension of the estimates for the
normal derivatives.  Under the assumptions of
Theorem \ref{thrmWeightedPoisson}, so 
 that in particular we know that
the boundary values of the 
 normal derivatives 
  (defined as non-tangential limits)
exist and are in $L^2(\partial\Omega_j)$,
for $0\le\beta\le3/2$, we can take Sobolev $-\beta$ 
 estimates of
\eqref{derWithZeroEst}:
\begin{align*}
\nonumber
\left\|
\partial_{\rho_j}   u \big|_{\partial\Omega_j}
\right\|&_{W^{-\beta,s}
	\left(
	\partial\Omega_j\cap \partial\Omega,
	\rho_{\hat{j}}, \lambda
	\right)
}\\
\nonumber
\lesssim&
\left\|
g_{bj}
\right\|_{W^{1-\beta,s}
	\left(
	\partial\Omega_j\cap \partial\Omega,
	\rho_{\hat{j}}, \lambda
	\right)}+\sum_{k}\left\|
g_{bk} 
\right\|_{W^{\max(1/2-\beta,-1/2),s}
	\left(
	\partial\Omega_k\cap \partial\Omega,
	\rho_{\hat{k}}, \lambda
	\right) }\\
&+s.c.\sum_{k }\left\|
\partial_{\rho_k}   u \big|_{\partial\Omega_k}
\right\|_{W^{-3/2,s}
	\left(
	\partial\Omega_k\cap \partial\Omega,
	\rho_{\hat{k}}, \lambda
	\right) }\\
& + \| u \|_{-\infty}+\left\|
\partial_{\rho}   u \big|_{\partial\Omega}
\right\|_{-\infty} 
+ \|
g_b \|_{-\infty}.
\end{align*} 
Summing over the boundaries yields,
in the same manner as 
 \eqref{weightedDv} above,
the estimates
\begin{equation}
\label{minus12Est}
\left\|
\partial_{\rho_j}   u \big|_{\partial\Omega_j}
\right\|_{W^{-\beta,s}
	\left(
	\partial\Omega_j\cap  \partial\Omega,
	\rho_{\hat{j}}, \lambda
	\right) }
\lesssim
\sum_j \left\|
g_{bj}
\right\|_{W^{1-\beta,s}
	\left(
	\partial\Omega_j\cap  \partial\Omega,
	\rho_{\hat{j}}, \lambda
	\right)}
,
\end{equation}
 modulo (estimates of) smoothing terms.

We conclude this section
 by illustrating how the
above analysis can be used to 
 obtain an expression for
the Poisson operator.
A Poisson operator, $P$, associated with
$2\square$ on $\Omega$, with
prescribed boundary values, $g_b$,
is the solution operator to a homogeneous
Dirichlet problem
\begin{align*}
& 2\square P(g_b) =0 \qquad \mbox{in } \Omega\\
& P(g_b) = g_b \qquad \mbox{on }
\partial\Omega.
\end{align*}
As seen from \eqref{vFormDNO} the principal
 terms in the Poisson operator are
\begin{equation*}
 \sum_j  \Theta^+_j \circ R_j.
\end{equation*} 
 And an expression for the Poisson operator
follows from \eqref{vFormDNO},
\begin{align*}
\nonumber
u
=& \sum_j  \Theta^+_j g_{bj}
+ \sum_{j, k}\Psi^{-1} \circ \lre_{-1/2}^{kj} g_{bk}
+
\sum_{j,k} \Psi^{-2}
g_{bj}\\
&+
\sum_{j,k} \Psi^{-2}
\circ \lre_{-1}^{kj} 
\left(
\partial_{\rho_k} u \big|_{\partial\Omega_k}
\right) 
+
\Psi^{-3}\left(\partial_{\rho}
u \big|_{\partial\Omega}\right) 
+ R^{-\infty},
\end{align*}
 to any desired degree by inserting local 
expressions for the normal derivatives,
$\partial_{\rho}
u \big|_{\partial\Omega}$,
 as in 
\eqref{derWithZeroEst}, and iterating.

\section{DNO}
\label{secDNO}

The Dirichlet to Neumann operator is defined
 as the operator which maps boundary values
of the solution to the homogeneous Dirichlet
 problem to the boundary values of the
 (outward) normal derivatives
of the solution to the homogeneous Dirichlet 
 problem.

We note as a corollary from our
 proof of Theorem \ref{thrmWeightedPoisson},
in particular the inequalities given in 
 \eqref{weightedDvBeta}
and \eqref{minus12Est},
 the following estimates for the 
 DNO:
\begin{thrm}
	 \label{thrmDNOWeight}
Let $-3/2\le \beta \le 0$.
Let $u$ be the solution to \eqref{homoDirProb} with
 $g_{bj} \in W^{1+\beta}(\partial\Omega\cap\partial\Omega_j,
  \rho_{\hat{j}},k)$ for all 
 $1\le j \le m$.  Then
 \begin{multline*}
   \left\|
   \partial_{\rho_j}   u \big|_{\rho_j =0}
   \right\|_{W^{\beta,s}
   	\left(
   	\partial\Omega_j\cap \partial\Omega,
   	 \rho_{\hat{j}},\lambda
   	\right) }
   \lesssim
   \left\|
   g_{bj}
   \right\|_{W^{1+\beta,s}
   	\left(
   	\partial\Omega_j\cap \partial\Omega,
   	\rho_{\hat{j}},\lambda
   	\right)}\\
  + \| u \|_{-\infty}+\left\|
  \partial_{\rho}   u \big|_{\partial\Omega}
  \right\|_{-\infty} 
  + \|
  g_b \|_{-\infty}.
 \end{multline*}
\end{thrm}
 In the case of $\beta=0$ and $s=0$
  we obtain
 the known 
 estimates on 
 Lipschitz
domains: 
\begin{equation*}
\left\| \partial_{\rho_j} u|_{\partial\Omega_j}
\right\|_{L^2(\partial\Omega \cap 
	\partial\Omega_j)}
\lesssim \| g_b \|_{W^1(\partial\Omega)},
\end{equation*}
 where the boundary values are to be 
understood in the sense of non-tangential
 limits,
for each $j$ (see Theorem 5.1 in \cite{Ve84}
and Theorem 1.4.3 in \cite{Sh05}). 

We start with a simplification of the
 expression for the normal derivative
along a boundary as in \eqref{derivToIterate}:
\begin{align}
\nonumber
\partial_{\rho_k}   u \big|_{\rho_k =0}
=& |D_{bk}| g_{bk}
-2 \sum_{j\neq k} 
|D_{bk}| \circ R_k \circ \Gamma^{\sharp}_j
\left(
\partial_{\rho_j} u 
\big|_{\partial\Omega_j}
+ |D_{bj}|g_{bj}
\right)
\\
\nonumber
& 
+\sum_{j\neq k} \lre_{-3/2}^{jk} 
\left(
\partial_{\rho_j} u \big|_{\rho_j =0}
\right) 
+ \sum_{j\neq k}
\lre_{-1/2}^{jk}  g_{bj}\\
\label{derForm-1}
&
+
\Psi^{-1}_{bk}\left(\partial_{\rho_k} u 
\big|_{\partial\Omega_k}\right)
+
\Psi^{0}_{bk}\left(g_{bk} \right)
+ R_{bk}^{-\infty} .
\end{align}

Our aim in this section is to calculate
 the zero order term, written
as $\Psi^{0}_{bk}\left(g_{bk} \right)$ in the expression 
 \eqref{derForm-1}, which is the same
  zero order term in \eqref{derWithZeroEst}.  
 We first include the
zero order operators on $g_{bk}$ coming from
$ \Psi^{-1}_{bk}
(\partial_{\rho_k} u \big|_{\rho_k
	=0})$ with the 
 $\Psi^{0}_{bk}\left(g_{bk} \right) $ term
in \eqref{derWithZeroEst}.  
 Let us denote the zero order 
boundary pseudodifferential operator 
 acting on $g_{bk}$ by
 $\Lambda_{bk}^0$ so that \eqref{withTheta} now 
reads
\begin{align}
\nonumber
\partial_{\rho_k}   u \big|_{\rho_k =0}
=& |D_{bk}| g_{bk}
- 2\sum_{j\neq k} 
|D_{bk}|\circ R_k \circ \Theta_j^+ g_{bj}\\
\label{derFormFirst}
&+ 2
\sum_{{j\neq k} \atop 
	{l\neq j}} |D_{bk}|\circ R_k \circ 
\Theta_j^+ \circ R_j \circ 
\Theta_l^+ g_{bl}
+
\Lambda^{0}_{bk}\left(g_{bk}\right)
+B_kg_b,
\end{align}
where we write $B_kg_b$ to denote the error
 terms
\begin{equation*}
 B_k g_b :=
  \sum_j \lre_{-1/2}^{jk} 
  g_{bj}
  + \sum_j \lre_{-3/2}^{jk} 
  \left(
  \partial_{\rho_j} u \big|_{\partial\Omega_j}
  \right) +R^{-\infty}_{bk}.
\end{equation*}
We note for future reference the form
 of the $\lre^{jk}_{-3/2}$ operators
is given by
\begin{equation}
\label{form3/2}
 \lre^{jk}_{-3/2}
  =
   R_k\circ \Psi^{-2} \circ R_j +
  |D_{bk}|\circ R_k \circ \Psi^{-1} \circ
   R_l \circ \Psi^{-1} \circ
    R_q \circ \Psi^{-2} \circ R_j,
\end{equation}
for $k\neq l$, $l\neq q$, $q\neq j$,
modulo lower order operators.
 This will be useful in 
Section \ref{secWeightedEst}.

With \eqref{derFormFirst}
 in \eqref{vPoissonBasic},
we can improve the expression for 
 the Poisson solution in \eqref{vFormDNO}:
\begin{align}
\nonumber
u
=& \sum_j  \Theta^+_j g_{bj}
+ \sum_{j, k}\Psi^{-1} \circ \lre_{-1/2}^{kj} g_{bk}
+
\sum_{j} \Psi^{-2}
g_{bj}\\
\label{vFormDNO2}
&+
\sum_{j,k} \Psi^{-2}
\circ \lre_{-3/2}^{kj} 
\left(
\partial_{\rho_k} u \big|_{\partial\Omega_k}
\right) 
+
\Psi^{-3}\left(\partial_{\rho}
u \big|_{\partial\Omega}\right) 
+ R^{-\infty}.
\end{align}

If we return to the derivation of 
 \eqref{withTheta}, we see the
$\Lambda^{0}_{bk} g_{bk} $ comes from
$i)$ $-2|D_{bk}|\circ \Psi^{-2}_{bk}$
applied to $\partial_{\rho_k}u \big|_{\rho_k
	=0}$,
where the 
$\Psi^{-2}_{bk}$ operator itself comes from
the restriction to $\partial\Omega_k$ of the
operator of order $-3$ in the symbol
expansion of the inverse to $\Gamma$,
 $ii)$ $-2|D_{bk}|\circ \Psi^{-1}_{bk}$
applied to $g_{bk}$, where the 
 $\Psi^{-1}_{bk}$ operator comes from
the restriction to $\partial\Omega_k$ of the
 operator of order $-3$ in the symbol
expansion of the inverse to $\Gamma$
 (composed with the operator with symbol
  $i\eta_k$),
   and $iii)$ from
$2|D_{bk}|\circ R_k\circ \Psi^{-1}
 \circ \Theta_k^+$
 applied to $g_{bk}$, where the $\Psi^{-1}$ 
operator is the same $\Psi^{-1}$ operator
 in \eqref{errorPoisson}, coming from
\begin{equation*}
 \Gamma^{-1} \circ \left(
 \sqrt{2}\sum_{j=1}^{m} S_j
 \circ(\partial_{\rho_j} u)
 + A(u)
 \right),
\end{equation*}
as well as $-2|D_{bk}|\circ \Psi^{-1}_{bk}g_{bk}$
 terms from
the operator of order $-2$ stemming from
 $\Gamma^{-1}\circ S_j$ 
in the expression
\begin{equation}
\label{contribII}
\Gamma^{-1} \circ S_j (\partial_{\rho_j}u)
=  \Psi^{-1} u + \Gamma^{-1}\circ S_j (g_b)
\end{equation}
using \eqref{pseudoOnDer}.

Regarding the terms from cases $i)$ and
 $ii)$ above we need to look at
 the symbol expansion of the inverse
of the operator $\Gamma$. 
Recall the second order operator,
$\Gamma$, in
\eqref{localSAT}:
\begin{align}
\nonumber
\Gamma:=
-
\bigg[\partial_{\rho_1}^2
+& \cdots + \partial_{\rho_m}^2
+ 2\Big(
\partial_{x_{m+1}}^2 + \cdots
+  \partial_{x_{m+l}}^2
\Big)\\
\label{GammaDefn}
&+ \frac{1}{2}\Big(
\partial_{x_{m+l+1}}^2 + \cdots
+  \partial_{x_{2n}}^2
\Big)
+ 2 \sum_{ij} l_{ij}
\partial_{x_i}\partial_{x_j}
\bigg]+ \sum_{j=1}^m \rho_j \tau_j .
\end{align}
For the 
second order operator 
$\sum \rho_j \tau_j$ in $\Gamma$, we write 
\begin{equation*}
\tau_k = -  \sum_{i,j}\tau_k^{ij} 
\frac{\partial^2}{\partial x_i \partial x_j},
\end{equation*}
%
%
and 
\begin{equation*}
\tau_k^{ij} = \tau_k^{ij}(\rho_{\hat{k}},x),
\end{equation*}
modulo $O(\rho_k)$.

We use the expansion
\begin{align}
 \nonumber
  \sigma\left( \Gamma^{-1} \right)
 =& \frac{1}{\eta^2+\Xi^2} \\
\label{orderMin3}
 &-i \frac{\sum_j \partial_{\xi_j}
 (\eta^2 +\Xi^2)  \partial_{x_j} 
 (\eta^2+\Xi^2)
+ \sum_j \partial_{\eta_j}
(\eta^2 +\Xi^2)  \partial_{\rho_j} 
(\eta^2+\Xi^2)
}{(\eta^2+\Xi^2)^3}.
\end{align}
 modulo lower order symbols.
We again remind the reader the above expansion is 
 just formal.  To avoid the singularities arising
at $\eta=\xi=0$, we could work instead with the
 operator $\Gamma +I$ and use cutoffs 
in the expansion \eqref{orderMin3};
 see the discussion following Equation \ref{errorPoisson}.

Recall that, for given $1\le j\le m$,
  we denote
by $\eta_{\hat{j}}$ the dual to the 
tangential (with respect to
$\partial\Omega_j$) coordinates $\rho_k$ for
$1
\le k \le m$.  Thus, by
$|\eta_{\hat{j}}|$ we mean
\begin{equation*}
|\eta_{\hat{j}}| = \sqrt{\sum_{1
		\le k \le m \atop k\neq j}
	\eta_k^2}.	 
\end{equation*}
Similarly, 
\begin{equation*}
|\xi_{\hat{j}}| = \sqrt{\sum_{1
		\le k \le 2n-m \atop k\neq j}
	\xi_k^2},
\end{equation*}
for $1\le j\le m$.
We also define a notation which gives importance
 to the vector fields $V_j$ for
$m+1\le j \le n$:
\begin{equation*}
 |\xi_V| = \sqrt{\xi_{m+1}^2 + \cdots \xi_{2n-m}^2}.
\end{equation*}

We extend to 
 $\mathbb{R}^{2n-1}$ the microlocal 
neighborhoods described in Section
 \ref{secNotation}
for each boundary, $\partial\Omega_j$.
 Namely, $\psi^-_{N,bj}$ will be
defined in analogy with $\psi^-_N$ with 
 support in the region
\begin{equation*}
 \xi_j < - N \sqrt{\eta_{\hat{j}}^2
  + \xi_{\hat{j}}}
\end{equation*}
for $1\le j\le m$.

  We note
\begin{align}
\nonumber
 \partial_{x_j}
 (\eta^2 +\Xi^2) =& 
  \partial_{x_j}
  \Xi^2 \\ 
\label{dxXi}
 =& O(x,\rho)
  O(\xi^2+\eta^2) + O(|\xi_{V}|)
   O(\xi,\eta),
\end{align}  
 for any
$1
\le j \le 2n-m $
(see Section \ref{secSetup} above),
and
\begin{align*}
 \partial_{\rho_j}
(\eta^2 +\Xi^2) =& 
\partial_{\rho_j}
\Xi^2 \\ 
=& \sum_{k,l} \tau_j^{kl} \xi_k \xi_l
  + O(\rho_j),
\end{align*}
 for
$1
\le j \le m $.

We thus have
\begin{equation*}
 \frac{\partial_{\xi_j}\Xi^2 \partial_{x_j}
 	 \Xi^2}{(\eta^2+\Xi^2)^3}
= O(x,\rho) + O\left(
 \frac{|\xi_{V}|}
  {(\eta^2+\Xi^2)^2} \right)
 ,
\end{equation*}
while
\begin{equation*}
 \frac{\partial_{\eta_j}\Xi^2 \partial_{\rho_j}
	\Xi^2}{(\eta^2+\Xi^2)^3}=
 2 \tau_j^{kl}\frac{\eta_j \xi_k \xi_l}
 {(\eta^2+\Xi^2)^3}
  +O(\rho_j)
\end{equation*}
which is all we will need to know of this 
 operator.   
 The contribution of the
last symbol 
 to the operators
written as $\Psi^{-3}$ in 
\eqref{errorPoisson} is given by
\begin{align}
\label{sumTerm}
-\frac{i}{(2\pi)^{2n}}\sum_{j,k,l } & \int
 \tau_{j}^{kl} \frac{2\eta_{j}\xi_k\xi_l}
 {\left(\eta^2+\Xi^2\right)^3} 
    F.T._{\hat{j}}\partial_{\rho_j}
 u(0_j,\eta_{\hat{j}},\xi)
 e^{i\rho\cdot\eta} e^{ix\xi} d\eta d\xi,
\end{align}
modulo the $O(\rho)$ terms, and modulo
 terms with symbols
of order
\begin{equation*}
O\left(\frac{|\eta_{\hat{j}}|}
 {(\eta^2+\Xi^2)^2}
 \right)
\end{equation*}
acting on
  $\partial_{\rho_j}u|_{\partial\Omega_j}$.
Note that such terms lead to 
 operators with arbitrarily small
norm in microlocal neighborhoods defined by
 the support of $\psi^-_{N,bj}$ for 
large $N$.
Upon integrating with respect to
 $\eta_j$, the integrals of the summation
  term in \eqref{sumTerm}
  are $O(\rho_j)$.  On the other hand
 restricting to a boundary $\Omega_k$ for
  $k\neq j$ would lead to 
 $\lre_{-5/2}^{jk} \left(
  \partial_{\rho_j} u 
  \big|_{\rho_j=0}
 \right)$ terms (see \eqref{min52}).

  The operator
associated with the symbol of order -3 in 
 \eqref{orderMin3} contributes (upon composition with
  the operators $2|D_k|$) to the
$\Psi^{-1}_{bk}
(\partial_{\rho_k} u \big|_{\rho_k
	=0})$ in 
\eqref{derForm-1}.   
 Denote this
 $\Psi^{-1}_{bk}$ operator by
  $A_{bk}^{-1}$. 
  To handle error terms from the
order -3 symbol in \eqref{orderMin3}, when 
 used as operators, we use the notation
 $\Psi^{\alpha}_{\varepsilon}$ 
 introduced in 
\eqref{smallOps} to refer to
  pseudodifferential operators with small operator
norm.
We work
 in a microlocal neighborhood,
with respect to $\partial\Omega_k$, 
 that is with symbols with support in 
the support of $\psi^-_{N,bk}$ with 
 large $N$.
In particular,
$|\xi_{\hat{k}}| \ll \sqrt{\xi^2+\eta_{\hat{k}}^2 }$,
and so for example, a symbol, given by
\begin{equation*}
 O\left(
  \frac{\xi_{\hat{k}}^2 +
   \eta_{\hat{k}}^2 }
 {\xi^2 + \eta_{\hat{k}}^2}
 \right),
\end{equation*}
of an operator
 on $\partial\Omega_k$,
will be denoted
 $\Psi^0_{\varepsilon,bk}$.
Symbols which are 
 $O(x,\rho)$ will also be included
in $\Psi^0_{\varepsilon,bk}$ as
 we can restrict
  to a small neighborhood 
 of the point on the boundary 
under consideration.

We have
\begin{align*}
A^{-1}_{bk}
 =& - 4 i \sum_{j,l} |D_{bk}|\circ R_k 
  \circ Op \left(
 \tau_k^{jl} \frac{\eta_k\xi_j\xi_l}
 {\left(\eta^2+\Xi^2_{bk}\right)^3} 
\right) \circ R_k + \Psi^{-1}_{\varepsilon,bk}\\
 =&  \Psi^{-1}_{ \varepsilon,bk},
\end{align*}

  For the terms $-2|D_{bk}|\circ \Psi^{-1}_{bk}$ arising in case $ii)$ above,
we note the $\Psi^{-1}_{bk}$ operator is 
 just the restriction to the boundary of
the operator in
$\Psi^{-2} \left(g_{bk}\right)$, from \eqref{errorPoisson}.
Let us denote this operator of 
 order -2 by $A^{-2}$.
As stated earlier, the $A^{-2}$ operator 
itself is just the operator of order $-3$ 
 given by the symbol
expansion of the inverse to $\Gamma$
 with symbol as in \eqref{orderMin3}
composed with the operator with symbol
$i\eta_k$.

 We note that
 \begin{align*}
\frac{i}{(2\pi)^n}&\sum_{\alpha,j,l}
 \int
 \tau_{\alpha}^{jl} \frac{2\eta_{\alpha}\xi_j\xi_l}
 {\left(\eta^2+\Xi^2\right)^3} 
 i\eta_k
 \widehat{g}_{bk}(\eta_{\hat{k}},\xi)
 e^{i\rho\cdot\eta} e^{ix\xi} d\eta d\xi\\
  =&- \frac{1}{(2\pi)^n}\sum_{j,l}
  \int
  \tau_{k}^{jl} \frac{2\eta_{k}^2\xi_j\xi_l}
  {\left(\eta^2+\Xi^2\right)^3} 
  \widehat{g}_{bk}(\eta_{\hat{k}},\xi)
  e^{i\rho\cdot\eta} e^{ix\xi} d\eta d\xi
 + O(\rho_k),
\end{align*}
using that
\begin{equation*}
 \int
 \frac{\eta_{k}}
 {\left(\eta^2+\Xi^2\right)^3} 
 \widehat{g}_{bk}(\eta_{\hat{k}},\xi)
 e^{i\rho\cdot\eta} e^{ix\xi} d\eta d\xi
  =O(\rho_k)
\end{equation*}
as above.

We thus have
\begin{align*}
R_k\circ A^{-2} g_{bk}
=&-\frac{i}{(2\pi)^n}R_k\circ \sum_{\alpha,j,l}
 \int
\tau_{\alpha}^{jl} \frac{2\eta_{\alpha}\xi_j\xi_l}
{\left(\eta^2+\Xi^2\right)^3} 
i\eta_k
\widehat{g}_{bk}(\eta_{\hat{k}},\xi)
e^{i\rho\cdot\eta} e^{ix\xi} d\eta d\xi\\
=&\frac{2}{(2\pi)^n}R_k\circ
 \sum_{ij} \int
\tau_k^{ij} \frac{\eta_k^2\xi_i\xi_j}
{\left(\eta^2+\Xi^2_{bk}\right)^3} 
\widehat{g}_{bk}(\eta_{\hat{k}},\xi)
e^{i\rho\cdot\eta} e^{ix\xi} d\eta d\xi \\
=&  \frac{1}{(2\pi)^{n-1}}
\frac{1}{8} \int \tau_k^{kk} 
 \frac{\xi_k^2}{(\eta_{\hat{k}}^2 + \Xi_{bk}^2)^{3/2} } 
\widehat{g}_{bk}(\eta_{\hat{k}},\xi)
e^{i\rho_{\hat{k}}\cdot\eta_{\hat{k}}}
  e^{ix\xi} d\eta_{\hat{k}} d\xi
,
\end{align*}
modulo 
terms (as in case $i)$ above) 
 which are $\Psi_{\varepsilon,bk}^{-1}$ in the 
  microlocal neighborhood,
  with respect to $\partial\Omega_k$, in which
  $|\xi_{\hat{k}}|\ll \sqrt{\xi^2+\eta_{\hat{k}}^2 }$.

Thus the term
 in
 $-2|D_{bk}|\circ \Psi^{-1}_{bk} 
  \left(g_{bk}\right)$ stemming
  from
case $ii)$ 
 in a microlocal neighborhood defined by
 the support of $\psi^{-}_{N,bk}$
can be
 written as
	\begin{align}
\nonumber
-2 |D_{bk}|\circ \Psi^{-1}_{bk} &
 g_{bk}\\
\nonumber
 =& -2|D_{bk}|\circ
Op\left( 
	\frac{1}{8}
	\tau_k^{kk} 
	\frac{\xi_k^2}{(\eta_{\hat{k}}^2 + \Xi_{bk}^2)^{3/2}} 
\right)
 \left(g_{bk}\right)\\
 \label{DRT}
=& -
\frac{1}{4}
\frac{1}{(2\pi)^{n-1}}
 \int \tau_k^{kk} 
\frac{\xi_k^2}{\eta_{\hat{k}}^2 + \Xi_{bk}^2} 
\widehat{g}_{bk}(\eta_{\hat{k}},\xi)
e^{i\rho_{\hat{k}}\cdot\eta_{\hat{k}}}
e^{ix\xi} d\eta_{\hat{k}} d\xi
,
\end{align}
modulo $\Psi^0_{\varepsilon, bk} g_b$
 as well as lower order terms.

We now handle case $iii)$ and the terms
from
\begin{equation*}
\Gamma^{-1} \circ \left(
\sqrt{2}\sum_{j=1}^{m} S_j
\circ(\partial_{\rho_j} u)
+ A(u)
 \right).
\end{equation*}
We first look at 
$\Gamma^{-1} \circ S_j
\circ(\partial_{\rho_j} u)$.
Let the symbol of $S_j$ be given by
\begin{equation*}
 \sigma(S_j)
  = s_j(\rho,x).
\end{equation*}
We will also use the notation
\begin{equation*}
 s_{0_kj}
(\rho_{\hat{k}},x) =
s_j(0_k,\rho_{\hat{k}},x),
\end{equation*}
and, in the case $j=k$, simply
\begin{equation*}
 s_{0j}
 (\rho_{\hat{j}},x) =
   s_j(0_j,\rho_{\hat{j}},x).
\end{equation*}
Then,
modulo lower order terms, we have
\begin{align*}
\Gamma^{-1} \circ S_j
\circ(\partial_{\rho_j} u) =
\frac{1}{(2\pi)^n}\int 
s_j(\rho,x)
\frac{
	\widehat{g}_{bj} + i\eta_j \widehat{u}
}
{\eta^2+\Xi^2}
e^{i \rho \cdot \eta} e^{ix\xi} d\eta  d\xi.
\end{align*}
The integral involving $g_{bj}$ can
be calculated by integrating with 
respect to $\eta_j$:
\begin{align*}
\int 
s_j(\rho,x)
\frac{
	\widehat{g}_{bj}}
{\eta^2+\Xi^2}
e^{i \rho \cdot \eta}& e^{ix\xi} d\eta  d\xi
 \\
=&
 \int 
 s_{0j}(\rho_{\hat{j}},x)
 \frac{
 	\widehat{g}_{bj}}
 {\eta^2+\Xi^2_{bj}}
 e^{i \rho \cdot \eta} e^{ix\xi} d\eta  d\xi
\\
=&2\pi \int 
s_{0j}(\rho_{\hat{j}},x)
\frac{
	\widehat{g}_{bj}}
{2\sqrt{\eta^2_{\hat{j}}+
		\Xi^2_{bj}}}
e^{\rho_j\sqrt{\eta^2_{\hat{j}}+
		\Xi^2_{bj}}}
e^{i \rho_{\hat{j}} \cdot \eta_{\hat{j}}}
e^{ix\xi} d\eta  d\xi,
\end{align*}
modulo lower order terms.
 Restricting to $\partial\Omega_k$ 
and applying $2\sqrt{2}
 |D_{bk}|$ yields a term
\begin{equation}
\label{opSk}
 \sqrt{2} Op(s_{0k}(\rho_{\hat{k}},x))
\end{equation}
which is to be included in the 
 $\Lambda^0_b$ operator.

For the integral involving $u$ we use
the expression for the Poisson 
solution in \eqref{vFormDNO2}.
We have
\begin{align}
\nonumber
\int 
s_j(\rho,x)&
\frac{
	i\eta_j \widehat{u}
}
{\eta^2+\Xi^2}
e^{i \rho \cdot \eta} e^{ix\xi} d\eta  d\xi
\\
\nonumber
=& 
-\sum_l \int 
s_{0_lj}(\rho_{\hat{l}},x)
\frac{
	\eta_j 
}
{\eta^2+\Xi^2_{bl}}
\frac{1}{\eta_l+i\sqrt{
		\eta_{\hat{l}}^2 + \Xi_{bl}^2
}}
\widehat{g}_{bl}
e^{i \rho \cdot \eta} e^{ix\xi} d\eta  d\xi\\
\nonumber
&+  \sum_{j,l} 
\Psi^{-2} \circ \lre_{-1/2}^{lj} g_{bl}
+ \Psi^{-3}g_b\\
\label{sTuTerm}
&+
\sum_{j,l} \Psi^{-3}\circ 
\lre_{-3/2}^{lj} 
\left(
\partial_{\rho_l} u \big|_{\partial\Omega_l}
\right) 
+
\Psi^{-4}\left(\partial_{\rho}
u \big|_{\partial\Omega}\right) 
+R^{-\infty}.
\end{align}
We will restrict the above relation
 to the boundary, $\partial\Omega_k$, and
we analyze the terms in the
 first summation on the
 right according to the cases
of $l=k$ or $l\neq k$, and
 according to $j=k$ or $j\neq k$.
In the case $l\neq k$ restricting to
 $\partial \Omega_k$ yields a term
$\lre_{-3/2}^{lk} g_{bl}$.  

In the case $l=k$, and 
$j\neq k$, we have
\begin{align*}
-R_k \circ \int 
s_{0_kj}(\rho_{\hat{k}},&x)
\frac{
	\eta_j 
}
{\eta^2+\Xi^2_{bk}}
\frac{1}{\eta_k+i\sqrt{
		\eta_{\hat{k}}^2 + \Xi_{bk}^2
}}
\widehat{g}_{bk}
e^{i \rho \cdot \eta} e^{ix\xi} d\eta  d\xi
\\
&=
\frac{2\pi i}{4}\sum_{k\neq j} \int 
s_{0_kj}(\rho_{\hat{k}},x)
\frac{
	\eta_j 
}
{\eta_{\hat{k}}^2 + \Xi_{bk}^2}
\widehat{g}_{bk}
e^{i \rho_{\hat{k}} \cdot \eta_{\hat{k}}}
e^{ix\xi} d\eta  d\xi  \\
&= 
 \Psi^{-1}_{\varepsilon, bk} g_{bk},
\end{align*}
where the symbol of the
 $\Psi^{-1}_{\varepsilon}$ operator is
of the form
\begin{equation*}
 O\left(\frac{|\eta_{\hat{k}}|}
  {\eta_{\hat{k}}^2 + \Xi_{bk}^2}\right)
\end{equation*}
and thus 
 can be made arbitrarily small
  in the support of
$\psi^-_{N,bk}$ 
(for large $N$).

Finally, in the case $l=k$ and 
 $j=k$, we have
\begin{align*}
-R_k \circ \int 
s_{0_k}(\rho_{\hat{k}},&x)
\frac{
	\eta_k 
}
{\eta^2+\Xi^2_{bk}}
\frac{1}{\eta_k+i\sqrt{
		\eta_{\hat{k}}^2 + \Xi_{bk}^2
}}
\widehat{g}_{bk}
e^{i \rho \cdot \eta} e^{ix\xi} d\eta  d\xi
\\
&=
-\frac{2\pi}{4}\int 
\frac{s_{0k}(\rho_{\hat{k}},x)}
{\sqrt{
		\eta_{\hat{k}}^2 + \Xi_{bk}^2
}}
\widehat{g}_{bk}
e^{i \rho_{\hat{k}} \cdot \eta_{\hat{k}}}
e^{ix\xi} d\eta_{\hat{k}}  d\xi.
\end{align*}
We can now restrict
 \eqref{sTuTerm} to $\partial\Omega_k$
 and write
\begin{align*}
\nonumber
R_k \circ
 \sum_j \int 
s_j(\rho,x)&
\frac{
	i\eta_j \widehat{u}
}
{\eta^2+\Xi^2}
e^{i \rho \cdot \eta} e^{ix\xi} d\eta  d\xi
\\
\nonumber
=& 
-\frac{2\pi}{4}\int 
\frac{s_{0k}(\rho_{\hat{k}},x)}
{\sqrt{
		\eta_{\hat{k}}^2 + \Xi_{bk}^2
}}
\widehat{g}_{bk}
e^{i \rho_{\hat{k}} \cdot \eta_{\hat{k}}}
e^{ix\xi} d\eta_{\hat{k}}  d\xi
 + \Psi^{-1}_{\varepsilon,bk} g_{bk}
\\
\nonumber
&+  \sum_{j,l} 
R_k\circ \Psi^{-2}\circ
  \lre_{-1/2}^{lj} g_{bl}
+ R_k\circ \Psi^{-3} g_{b}\\
&+\sum_{j,l} 
R_k\circ \Psi^{-2}\circ \lre_{-3/2}^{lj} 
\left(
\partial_{\rho_l} u \big|_{\partial\Omega_l}
\right) \\
&+
R_k\circ\Psi^{-4}\left(\partial_{\rho}
u \big|_{\partial\Omega}\right) 
+R_{bk}^{-\infty}.
\end{align*}
Applying $2\sqrt{2}|D_k|$ yields other
 terms
\begin{equation}
\label{opSk2}
 -\frac{\sqrt{2}}{2}
  Op(s_{0k}) + \Psi^0_{\varepsilon,bk}
\end{equation}
to be added to 
$\Lambda^0_b$.

We also note that the error terms arising  
 from 
$2\sqrt{2}|D_{bk}| \Gamma^{-1} \circ
 S_j \circ (\partial_{\rho_j}u)$
are of the form
\begin{equation*}
\sum_{l} 
 \lre_{-1/2}^{lk} g_{bl}
 + \Psi^{-1}_{bk}g_{bk}+
\sum_{l} 
 \lre_{-3/2}^{lk} 
 \left(
 \partial_{\rho_l} u \big|_{\partial\Omega_l}
 \right) 
+R_{bk}^{-\infty}
\end{equation*}
and are thus already included in the
 formula \eqref{derFormFirst}.

Putting 
\eqref{opSk} and \eqref{opSk2} together,
 we see the 
 terms $2\sqrt{2}|D_{bk}|\circ R_k \circ
\Gamma^{-1}\circ S_j\circ \partial_{\rho_j} u$,
yield a 
\begin{equation}
\label{DRS}
 \frac{\sqrt{2}}{2}\frac{1}{(2\pi)^{n-1}}
\int  s_{0k}(\rho_{\hat{k}},x)
 \widehat{g}_{bk}
 e^{i \rho_{\hat{k}} \cdot \eta_{\hat{k}}} 
  e^{ix\xi} d\eta  d\xi,
\end{equation}
which is to be included in 
 the $\Lambda^0_{b}$ operator.

We next look at 
$\Gamma^{-1} \circ A(u)$.  $A$ is
 a first order differential operator 
(tangential to all boundaries $\partial\Omega_j$
 for $1\le j\le m$).  
 Denote the
symbol of $A$ by
\begin{align*}
\sigma(A) =& \alpha(\rho,x,\xi,\eta)\\
=& \sum \alpha_j(\rho,x) \xi_j.
\end{align*}
In analogy with the symbol
 $s_{0j}$ we define 
\begin{align*}
&\alpha_{0_k}(\rho_{\hat{k}},x) :=
 \alpha(0_k,\rho_{\hat{k}},x)\\
& \alpha_{0_kj}(\rho_{\hat{k}},x):=
\alpha_j(0_k,\rho_{\hat{k}},x).
\end{align*}
We again use the expression in 
\eqref{vFormDNO} for $u$ to look at the
 action of 
  $\Gamma^{-1}\circ A$ on $u$,
  up to error terms.  
  
 We use
 the notation $\cdots$ in the following
 representation to indicate terms
 which upon being operated by
 $|D_{bk}|\circ R_k$ 
lead to terms of the form
\begin{equation}
\Psi^{-1}_{bk} g_{bk}+
\lre_{-1/2}^{jk}\left(g_{bj} \right)
+ \lre_{-3/2}^{jk} 
\left(
\partial_{\rho_j} u\big|_{\rho_j =0}
\right) 
+R^{-\infty}_{bk}
\label{modA}
\end{equation}
 in  
 \eqref{derFormFirst}.
   
We have
\begin{align*}
\Gamma^{-1} \circ& A(u) =
\frac{1}{(2\pi)^n} \int 
\alpha(\rho,x,\xi)
\frac{
	\widehat{u}
}
{\eta^2+\Xi^2}
e^{i \rho \cdot \eta} e^{ix\xi} d\eta  d\xi
 + \Psi^{-2}u\\
=&  \frac{1}{(2\pi)^n} \sum_j
 \int \alpha_{0_j}(\rho_{\hat{j}},x,\xi)
\frac{1 }
{\eta^2+\Xi^2_{bj}}
\frac{i\widehat{g}_{bj}}
{\eta_j+i\sqrt{\eta_{\hat{j}}^2+
		\Xi_{\hat{j}}^2}}
e^{i \rho \cdot \eta} e^{ix\xi} d\eta  d\xi+\cdots\\
=&   \frac{1}{4}\frac{1}{(2\pi)^{n-1}}
\sum_j
 \int 
\alpha_{0_j}(\rho_{\hat{j}},x,\xi)
\frac{\widehat{g}_{bj}}
{\eta_{\hat{j}}^2+
	\Xi_{\hat{j}}^2}
e^{\rho_j\sqrt{\eta_{\hat{j}}^2+
		\Xi_{bj}^2}}
e^{i \rho_{\hat{j}} \cdot \eta_{\hat{j}}}
e^{ix\xi} d\eta_{\hat{j}}  d\xi+\cdots.
\end{align*}

Then, the terms $2|D_{bk}|\circ R_k \circ
\Gamma^{-1}\circ A u$ yield
\begin{equation}
\label{DRA}
\frac{1}{2}\frac{1}{(2\pi)^{n-1}} \int 
\frac{\alpha_{0k}(\rho_{\hat{k}},x,\xi)}{
 \sqrt{\eta_{\hat{k}}^2+
 	\Xi_{bk}^2}}
\widehat{g}_{bk}
e^{i \rho_{\hat{k}} \cdot \eta_{\hat{k}}} 
e^{ix\xi} d\eta  d\xi
\end{equation}
to highest order,
 i.e. modulo terms 
of the form \eqref{modA}.

We are now ready to put together the 
 $\Lambda_{bk}^0$ operator according to
the terms from cases $i)$, $ii)$ and $iii)$
above.  
From \eqref{DRT}, \eqref{DRS}, and \eqref{DRA} above
we have
\begin{align}
\nonumber
\Lambda_{bk}^0 g_{bk} =& 
  -\frac{1}{4} \frac{1}{(2\pi)^{n-1}}
  \int \tau_k^{kk} 
  \frac{\xi_k^2}{\eta_{\hat{k}}^2
 +\Xi_{bk}^2} 
  \widehat{g}_{bk}(\eta_{\hat{k}},\xi)
  e^{i\rho_{\hat{k}}\cdot\eta_{\hat{k}}}
  e^{ix\xi} d\eta_{\hat{k}} d\xi\\
 \nonumber
   &+ \frac{\sqrt{2}}{2}\frac{1}{(2\pi)^{n-1}}
   \int  s_{0k}(\rho_{\hat{k}},x)
   \widehat{g}_{bk}
   e^{i \rho_{\hat{k}} \cdot \eta_{\hat{k}}} 
   e^{ix\xi} d\eta  d\xi\\
 \label{replace}
 &+ \frac{1}{2}\frac{1}{(2\pi)^{n-1}} \int 
 \frac{\alpha_{0k}(\rho_{\hat{k}},x,\xi)}{
 	\sqrt{\eta_{\hat{k}}^2+
 		\Xi_{bk}^2}}
 \widehat{g}_{bk}
 e^{i \rho_{\hat{k}} \cdot \eta_{\hat{k}}} 
 e^{ix\xi} d\eta  d\xi
\end{align}
 modulo $\Psi^{0}_{\varepsilon,bk}g_b$ (from 
cases $i)$ and $ii)$).  Recall that 
\begin{equation*}
 \| \Psi^{0}_{\varepsilon,bk}g_{bk} 
\|_{W^{\gamma,s}(\partial\Omega_k\cap\partial\Omega,
	 \rho_{\hat{k}},\lambda )}
  \lesssim s.c. \|g_{bk} \|_{W^{\gamma,s}(\partial\Omega_k\cap\partial\Omega,
  	\rho_{\hat{k}},\lambda )}.
\end{equation*}

We also note the error terms of the form
\begin{equation*}
C_kg_b=
\Psi^{-1}_{bk} g_{bk}+
 \sum_{j} \lre_{-1/2}^{jk} 
 g_{bj}
 +  \sum_{j} 
  \lre_{-3/2}^{jk} 
 \left(
 \partial_{\rho_j} u\big|_{\rho_j =0}
 \right) 
 +R^{-\infty}_{bk}
\end{equation*}
in the $B_kg_b$ terms in 
 \eqref{derFormFirst} and those 
 resulting from
 the above expansions.
The $\lre_{-3/2}^{jk}$ terms
 remain in the form
\eqref{form3/2} with additional 
 terms of the form
\begin{equation*}
|D_k|\circ R_k \circ \Psi^{-2} \circ 
  \lre_{-3/2}^{lj}.
\end{equation*}

We thus have from \eqref{derFormFirst}
\begin{thrm}
\label{thrmDNO}
 Let $N^-_k$ be the DNO operator mapping the
boundary values of the homogeneous 
 Dirichlet problem \eqref{homoDirProb}
 to the boundary values on 
$\partial\Omega_k \cap
 \partial\Omega$ of the outward normal 
 derivative of the solution.
Then
\begin{multline*}
\nonumber
N_k^-g_b
=|D_{bk}| g_{bk}
- 2\sum_{j\neq k} 
|D_{bk}|\circ R_k \circ \Theta_j^+ g_{bj}\\
+ 2
\sum_{{j\neq k} \atop 
	{l\neq j}} |D_{bk}|\circ R_k \circ 
\Theta_j^+ \circ R_j \circ 
\Theta_l^+ g_{bl}
+
\Lambda^{0}_{bk}\left(g_{bk}\right)
+C_kg_b,
\end{multline*}
where
 in the microlocal support of
a cutoff, $\psi^{-}_{N,bk}$, we have
 modulo operators of the form
$\Psi^{0}_{\varepsilon,bk}$ for large $N$,
\begin{align*}
 \nonumber
 \Lambda_{bk}^0  =& 
 -\frac{1}{4} Op\left( \tau_k^{kk} 
 \frac{\xi_k^2}{\eta_{\hat{k}}^2
 	+\Xi_{bk}^2}\right) + 
\frac{\sqrt{2}}{2}Op
 \left( s_{0k}(\rho_{\hat{k}},x)
\right)+
 \frac{1}{2}Op\left(
 \frac{\alpha_{0k}(x,\rho_{\hat{k}},\xi)}{
 	\sqrt{\eta_{\hat{k}}^2+
 		\Xi_{bk}^2}}
\right),
\end{align*}
and, for $0 \le \gamma \le 1$
 and $\lambda \ge 0$,
\begin{multline*}
 \big\|  C_k g_b 
 \big\|_{W^{\gamma,s}
 	\left(
 	\partial\Omega_k\cap \partial\Omega,
 	\rho_{\hat{k}},\lambda
 	\right)} 
  \lesssim
   \sum_j \left\|
g_{bj}
\right\|_{W^{\gamma-1/2,s}(\partial\Omega_j\cap\partial\Omega,
	\rho_{\hat{j}},\lambda )}
\\
+ \| u \|_{-\infty}+\left\|
\partial_{\rho}   u \big|_{\partial\Omega}
\right\|_{-\infty} 
+ \|
g_b \|_{-\infty}.
\end{multline*}
\end{thrm}

\section{Green's operator}
\label{secGreens}

The Green's operator 
 for the $\square$ operator 
  is defined as the
 solution operator, $G$, to
\begin{align*}
& 2\square G(g) =g \qquad \mbox{in } \Omega\\
& G(g) = 0 \qquad \mbox{on }
\partial\Omega.
\end{align*}
As we did with the Poisson operator, we will
 find an expression for the
Green's operator, modulo some smoothing terms.  For
 this purpose, we use again the notation
  $R^{-\infty}$ to refer to smoothing terms, but in this
section $R^{-\infty}$ will 
 mean
 $ \Psi^{-\infty}$ applied to 
 $u=G(g)$, or
$ \Psi^{-\infty}$ applied to the boundary terms
$\partial_{\rho} u 
\big|_{\partial\Omega} $
 Furthermore, $R_{bk}^{-\infty}$ will denote 
$\Psi_{bk}^{-\infty}$ applied to
 $\partial_{\rho_k} u 
\big|_{\partial\Omega_k} $ 
 and to denote terms
described by $R_k\circ R^{-\infty}$.
 For instance, we have, for any $\alpha\ge 0$,
\begin{align}
\nonumber
\| R_{bk}^{-\infty}
 &\|_{W^{\alpha,s}(\partial\Omega_k \cap \partial\Omega,
 	\rho_{\hat{k}},\lambda)}\\
 \nonumber
& \lesssim 
\| R_{k}\circ
 \Psi^{-\infty}u
\|_{W^{\alpha+s}(\partial\Omega_k)}   +
 \| R_{k}\circ
 \Psi^{-\infty}\left(
\partial_{\rho} u 
\big|_{\partial\Omega} 
 \right)
 \|_{W^{\alpha+s}(\partial\Omega_k)}\\
\nonumber
&\qquad +  \| 
\partial_{\rho_k} u 
\big|_{\partial\Omega_k} 
\|_{W^{-\infty}(\partial\Omega_k)}\\
 \nonumber
&\lesssim
 \|
 \Psi^{-\infty}u
 \|_{W^{\alpha+\epsilon+1/2+s}(\Omega_k)}+
 \|
 \Psi^{-\infty}\left(
 \partial_{\rho} u 
 \big|_{\partial\Omega} 
 \right)
 \|_{W^{\alpha+\epsilon+1/2+s}(\Omega_k)}\\
\nonumber
&\qquad +  \| 
\partial_{\rho_k} u 
\big|_{\partial\Omega_k} 
\|_{W^{-\infty}(\partial\Omega_k)}\\
 \label{RbInftyEst}
& \lesssim
 \| u
 \|_{L^2(\Omega)} +   
 \|
 \partial_{\rho} u 
 \big|_{\partial\Omega} 
 \|_{W^{-\infty}(\partial\Omega)},
\end{align}
where $\epsilon\ge 0$ is chosen so that
 $\alpha+\epsilon >0$,
and thus so that the Trace Theorem
 applies.
The $L^2$ norm of $u$ can be bounded by
 $\|g\|_{L^2(\Omega)}$ (see Theorem
  .5 of \cite{JK95}).  To obtain
estimates for the boundary values of
 the normal derivative, we argue as in 
  Section \ref{secPoisson}:
\begin{align}
\nonumber
\partial_{\rho_j}u
\big|_{\rho_j =0 } 
=& \int_{-\infty}^0 \partial_{\rho_j}^2 u d\rho_j \\
\label{normalInhomoInt}
=& \int_{-\infty}^0 \Lambda_{t_j}^2 u
d\rho_j
 +\int_{-\infty}^0 \Lambda_{t_j}^1 g
 d\rho_j
,
\end{align}
where $\Lambda^i_{t_j}$ is an operator 
 of order $i$ tangential to $\partial\Omega_j$.
Then applying a smoothing tangential
 operator yields
\begin{equation*}
 \|
\partial_{\rho} u 
\big|_{\partial\Omega} 
\|_{W^{-\infty}(\partial\Omega)}
 \lesssim \|u\|_{L^2(\Omega)}
 +\|g\|_{L^2(\Omega)}.
\end{equation*}
Putting all this together yields
\begin{align*}
 \| R_{bk}^{-\infty}
 \|_{W^{\alpha,s}(\partial\Omega_k \cap \partial\Omega,
 	\rho_{\hat{k}},\lambda)}
\lesssim&   \|u\|_{L^2(\Omega)}
 +\|g\|_{L^2(\Omega)}\\
\lesssim& \|g\|_{L^2(\Omega)}
\end{align*}
for any $\lambda\ge 0$.

Using the notation of Section
\ref{secPoisson}, with 
 the $\Gamma$ operator defined as in 
\eqref{GammaDefn}, 
for the solution using Green's operator, 
 we write the interior equation as
\begin{equation*}
\Gamma u + 
 \sqrt{2}S \partial_{\rho} u +
  A u = g,
\end{equation*}
where $\sqrt{2}S\circ \partial_{\rho} u $
 is the short-hand notation for
  the sum of terms
\begin{equation*}
\sqrt{2}S\circ \partial_{\rho} u :=
  \sqrt{2}\sum_{j=1}^{m} S_j
 \circ(\partial_{\rho_j} u)
\end{equation*}
as in \eqref{origPoissonFourier}, and
with boundary conditions $u=0$ on $\rho_i=0$,
 $i=1,\ldots,m$.

Applying the operator,
 $\Gamma^{\sharp}$, with symbol
$
\frac{1}{\eta^2 +
	\Xi^2}
$
we get, in a similar manner to
 \eqref{errorPoisson} above,
\begin{equation*}
u= \sum_j \Gamma^{\sharp}
\left(\partial_{\rho_j} u \big|_{\rho_j =0} \right)
+ \Psi^{-3}
\left(\partial_{\rho} u 
\big|_{\partial\Omega} \right)
+  \Gamma^{\sharp} g
+  \Psi^{-3} g +\Psi^{-1}u.
\end{equation*}
Solving for $u$ yields
\begin{equation}
\label{vExp}
u= \sum_j \Gamma^{\sharp}
\left(\partial_{\rho_j} u
\big|_{\rho_j =0} \right)
+  \Psi^{-3}
\left(\partial_{\rho} u 
\big|_{\partial\Omega} \right)
+  \Gamma^{\sharp} g
+  \Psi^{-3} g + R^{-\infty}.
\end{equation}

We now obtain an expression
  for $\partial_{\rho_k}  
u|_{\rho_k=0}$.  We use
\begin{align*}
R_k\circ\partial_{\rho_k}  \circ \sum_j
 \Gamma^{\sharp} \left(\partial_{\rho_j} u 
\big|_{\rho_j =0} \right)
= \frac{1}{2}\partial_{\rho_k}u 
\big|_{\rho_k =0}
+ \sum_{j\neq k}
 \lre_{-1/2}^{jk}\left(\partial_{\rho_j} u 
\big|_{\rho_j =0} \right)
\end{align*}
and
\begin{align*}
\partial_{\rho_k}
\circ &\Gamma^{\sharp} g
= \frac{1}{(2\pi)^n}\int
\frac{i\eta_k}{\eta_k^2+\left(
	\eta_{\hat{k}}^2+\Xi^2\right)}
\widehat{g} (\eta,\xi)  
e^{i x\xi}e^{i\rho\eta} d\xi d\eta\\
=& 
\frac{1}{(2\pi)^n}\int
\frac{i\eta_k}{\eta_k^2+\left(
	\eta_{\hat{k}}^2+\Xi^2\right)}
\int_{-\infty}^0 \widetilde{g}
(t_k,\eta_{\hat{k}},\xi) e^{-it_k \eta_k }
dt_k e^{i x\xi}e^{i\rho\eta} d\xi d\eta\\
=& \frac{1}{(2\pi)^n}\int
\int_{-\infty}^0 
\frac{i\eta_k}{\eta_k^2+\left(
	\eta_{\hat{k}}^2+\Xi^2\right)}
\widetilde{g}
(t_k,\eta_{\hat{k}},\xi) 
e^{i(\rho_k-t_k) \eta_k }
dt_k e^{i x\xi}
e^{i\rho_{\hat{k}}\eta_{\hat{k}}}
d\xi d\eta\\
=& \frac{1}{(2\pi)^{n-1}}
\frac{1}{2}\int
\int_{-\infty}^0 
\mbox{sgn}(t_k-\rho_k)
\widetilde{g}
(t_k,\eta_{\hat{k}},\xi) \times \\
& \qquad \qquad 
\qquad \qquad \qquad \qquad \qquad  
e^{-|\rho_k-t_k|
	\sqrt{\eta_{\hat{k}}^2  + \Xi^2}} 
dt_k e^{i x\xi}
e^{i\rho_{\hat{k}}\eta_{\hat{k}}}
d\xi d\eta_{\hat{k}}\\
\end{align*}
which, in the limit $\rho_k\rightarrow 0$, tends to
\begin{equation}
\label{limitDer}
- \frac{1}{(2\pi)^{n-1}}\frac{1}{2}
\int \int_{-\infty}^0 
\widetilde{ g}(t_k,\eta_{\hat{k}},\xi)
e^{t_k\sqrt{\eta_{\hat{k}}^2+\Xi^2_{bk}}} 
dt_k e^{i x\xi} 
e^{i\rho_{\hat{k}}\eta_{\hat{k}}}
d\xi d\eta_{\hat{k}}
.
\end{equation}
Let $\Theta^{-}_j\in \Psi^{-1}(\Omega)$ be the operator with symbol
\begin{equation*}
\sigma(\Theta^{-}_j) = \frac{i}
{\eta_j-i 
	\sqrt{\eta_{\hat{j}}^2+\Xi^2_{bj}}}.
\end{equation*}
We can then rewrite the term in 
\eqref{limitDer} as
$ \frac{1}{2}R_k\circ \Theta_k^- g$, and
\begin{equation*}
R_k\circ \partial_{\rho_k}
\circ \Gamma^{\sharp} g
= \frac{1}{2}R_k\circ\Theta_k^- g.
\end{equation*}

Returning to \eqref{vExp} and applying
$R_k\circ \partial_{\rho_k}$ yields
\begin{align*}
\partial_{\rho_k} u|_{\rho_k=0} =&
\frac{1}{2} \partial_{\rho_k}  
u|_{\rho_k=0} 
+\frac{1}{2} R_k \circ \Theta_k^-g +
\sum_{j\neq k} \lre_{-1/2}^{jk}\left(\partial_{\rho_j} u 
\big|_{\rho_j =0} \right)\\
&+ 
\Psi_{bk}^{-1}(\partial_{\rho_k} u|_{\rho_k=0} )
+R_k\circ \Psi^{-2} g +R_{bk}^{-\infty}.
\end{align*}
Thus
\begin{equation}
\label{RNormTheta}
\partial_{\rho_k}  
u|_{\rho_k=0} = R_k\circ \Theta_k^-g
+ \sum_{j\neq k}\lre^{jk}_{-1/2}\left(\partial_{\rho_j}
u \big|_{\rho_j =0} \right)
+ R_k\circ \Psi^{-2} g +R_{bk}^{-\infty}.
\end{equation}

Our aim is to provide (weighted) estimates for
 $G(g)$.  To deduce
  these with the help
  of \eqref{vExp}, we
   need estimates for boundary values of 
  the normal derivatives, and we start with estimating the
term $\Theta^-g$ in \eqref{RNormTheta}
 (see also Theorem \ref{PsiInt}, whose
  proof can be read from that below).
In the following Theorems (and Corollaries), we 
 will assume the data function (form) satisfies
\begin{equation*}
 g\in W^{0,s}(\Omega,\rho_{\hat{j}},
 \lambda)
\end{equation*}
for some $s,\lambda \ge 0$ and for all
 $j=1,\ldots m$.
We start with
 \begin{thrm}
	\label{ThetaInt}
\begin{equation*}
 \| \Theta^-_j g
	\|_{W^{1,s}(\Omega,\rho_{\hat{j}},
		\lambda)}
 \lesssim
 \left\|
	g
 \right\|_{W^{0,s}(\Omega,\rho_{\hat{j}},
 	\lambda)}.
\end{equation*}
\end{thrm}
\begin{proof}
 We recall the convention that
$\Theta^-_j g$ for $g$ defined on $\Omega$
 refers to $\Theta^-_j g^{E_I}$, where
the superscript $E_I$ denotes extensions 
(defined locally) by zero
 across $\rho_i=0$
  for $i\in I :=\{1, \dots, m\}$.

Since
 $\Theta^-_j g=\Psi^{-1}g^{E_I}$, we have
\begin{equation*}
 \| \Theta^-_j g \|_{W^{1,s}(\Omega,\rho_{\hat{j}},
 	\lambda)}
 \lesssim \sum_{r\le s} 
  \left\|\rho_{\hat{j}}^{r\lambda\times(m-1)}
   \Psi^{-1}g^{E_I}
 \right\|_{W^{1+r}(\mathbb{R}^n)},
\end{equation*}
whereas for each $r\le s$ it holds that
\begin{align*}
\left\| 
 \rho_{\hat{j}}^{r\lambda\times(m-1)}
\Psi^{-1}g^{E_I}
\right\|_{W^{1+r}(\mathbb{R}^n)}
 \lesssim&
\sum_{l\le r}
\left\|
 \Psi^{-1-(r-l)} 
  \rho_{\hat{j}}^{l\lambda\times(m-1)}g^{E_I}
\right\|_{W^{1+r}(\mathbb{R}^n)} \\
\lesssim&
\left\|
 g^{E_I}
\right\|_{W^{0,r}(\mathbb{R}^n,\rho_{\hat{j}},
	\lambda)}
  .
\end{align*}
Summing over $r\le s$ yields
\begin{align*}
 \| \Theta^-_j g \|_{W^{1,s}(\Omega,\rho_{\hat{j}},
 	 \lambda)}
 \lesssim&
  \left\|
  g^{E_I}
\right\|_{W^{0,s}(\mathbb{R}^n,\rho_{\hat{j}},
	\lambda)}\\
\lesssim& 
 \left\|
 g
 \right\|_{W^{0,s}(\Omega,\rho_{\hat{j}},
 	\lambda)}
.
\end{align*}
\end{proof}

As a corollary of the Sobolev Trace Theorem
 (applied to
  each smooth domain $\Omega_j$) we obtain
\begin{cor}
\label{RThetaEst}
\begin{equation*}
 \|  R_j\circ \Theta_j^-g
\|_{W^{1/2,s}
	(\partial\Omega_j\cap \partial\Omega,
 \rho_{\hat{j}},\lambda)}
\lesssim
\left\|
g
\right\|_{W^{0,s}(\Omega,\rho_{\hat{j}},
	\lambda)}.
\end{equation*}
\end{cor}
\begin{proof}
	 For $\varphi$, a function defined on
$\mathbb{R}^n$, 
let $Z_{\Omega,j}(\varphi)$ denote the
  extension from $\Omega$ to
 $\Omega_j$ defined by
\begin{equation*}
Z_{\Omega,j}(\varphi) :=
  \begin{cases}  \varphi &
   \mbox{in } \Omega  \\ 
 0 & \mbox{in } \Omega_j\setminus \Omega.
 \end{cases}
\end{equation*}
We then have
\begin{align}
\nonumber
  \|  R_j\circ \Theta_j^-g
\|_{W^{1/2,s}
	(\partial\Omega_j\cap \partial\Omega,
	\rho_{\hat{j}},\lambda)}
\simeq&
 \sum_{r\le s}
 \left\|    
 R_j\circ Z_{\Omega,j} \left( 
 \rho_{\hat{j}}^{r\lambda\times(m-1)}  \Theta_j^-g
 \right)
 \right\|_{W^{1/2+r}
 	(\partial\Omega_j)} 
\\
 \label{RThetag}
 \lesssim&
 \sum_{r\le s}
\left\|    
 Z_{\Omega,j} \left( 
\rho_{\hat{j}}^{r\lambda\times(m-1)}  \Theta_j^-g
\right)
\right\|_{W^{1+r}
	(\Omega_j)}  .
\end{align}
 Furthermore,
since
\begin{equation*}
 \rho_{\hat{j}}^{r\lambda\times(m-1)}  \Theta_j^-g
  \in W^{1+r}
  (\Omega)
\end{equation*}
by Theorem \ref{ThetaInt}, we have
\begin{equation*}
 Z_{\Omega,j} \left( 
\rho_{\hat{j}}^{r\lambda\times(m-1)}  \Theta_j^-g
\right)\in W^{1+r}
(\Omega_j)
\end{equation*}
with
\begin{equation*}
 \left\|    
 Z_{\Omega,j} \left( 
\rho_{\hat{j}}^{r\lambda
	 \times(m-1)}  \Theta_j^-g
\right)
\right\|_{W^{1+r}(\Omega_j)} 
\lesssim 
  \left\|    
 \rho_{\hat{j}}^{r\lambda\times(m-1)} 
  \Theta_j^-g
 \right\|_{W^{1+r}(\Omega)} 
\end{equation*}
by Lemma 4.2 in 
 \cite{Eh18_halfPlanes}.

Inserting these last estimates into \eqref{RThetag}
 gives
\begin{align*}
  \|  R_j\circ \Theta_j^-g
\|_{W^{1/2,s}
	(\partial\Omega_j\cap \partial\Omega,
	\rho_{\hat{j}},\lambda)}
 \lesssim& 
\sum_{r\le s}
  \left\|    
\rho_{\hat{j}}^{r\lambda\times(m-1)}  \Theta_j^-g
\right\|_{W^{1+r}(\Omega)} 
\\
\lesssim&
 \|  \Theta_j^-g
\|_{W^{1,s}
	(\Omega,
	\rho_{\hat{j}},\lambda)}.
\end{align*}
The Corollary now follows from 
another application of Theorem \ref{ThetaInt}.
\end{proof}

A similar proof can be used to estimate 
 the terms which appear as 
$R_j \circ \Psi^{-2} g$ in \eqref{RNormTheta}
 above.  We obtain
\begin{align*}
\|  R_j\circ \Psi^{-2} g
 \|_{W^{1/2,s}
 	(\partial\Omega_j\cap \partial\Omega,
 	\rho_{\hat{j}},\lambda)}
 \lesssim&
 \|  g
 \|_{W^{0,s}
 	(\Omega, \rho_{\hat{j}},\lambda)}.
\end{align*}

We can now establish estimates for the boundary 
 values of normal derivatives of the 
Green's solution.
\begin{cor}
\label{corGreenNormDerEst}
\begin{equation*}
\sum_j
\|  
\partial_{\rho_j} u|_{\rho_j=0}
\|_{W^{1/2,s}
	(\partial\Omega_j\cap \partial\Omega,
	\rho_{\hat{j}},\lambda)}
\lesssim
\sum_j
\left\|
g
\right\|_{W^{0,s}(\Omega,\rho_{\hat{j}},
	\lambda)}.
\end{equation*}
\end{cor}
\begin{proof}
First, we note that with the 
 assumption that $g\in L^2
 (\Omega)$, we have
$\partial_{\rho}u|_{\partial\Omega}
 \in W^{-1}(\partial\Omega)$ by
\eqref{normalInhomoInt}
 and interior estimates for
for the solution to the inhomogeneous
 Dirichlet problem \cite{JK95}.
Using a bootstrapping
 argument, we can use \eqref{RNormTheta}
 to show
$\partial_{\rho} u|_{\partial\Omega}
 \in W^{1/2}(\partial\Omega)$.
Applications of \eqref{RNormTheta} can then
 be used to get weighted estimates.  First,
to show $\rho_{\hat{j}} \partial_{\rho_j} 
 u|_{\partial\Omega_j} \in W^{1}(
 \partial\Omega_j \cap \partial\Omega)$,
and with this, that
 $\rho_{\hat{j}} \partial_{\rho_j} 
 u|_{\partial\Omega_j} \in W^{3/2}(
 \partial\Omega_j \cap \partial\Omega)$.
Then multiplication with 
 $\rho_{\hat{j}}^2$ can be used to show
first $\rho_{\hat{j}}^2 \partial_{\rho_j} 
u|_{\partial\Omega_j} \in W^{2}(
\partial\Omega_j \cap \partial\Omega)$, and then
 $\rho_{\hat{j}}^2 \partial_{\rho_j} 
 u|_{\partial\Omega_j} \in W^{5/2}(
 \partial\Omega_j \cap \partial\Omega)$, and so forth.
	
For the estimates, we take weighted estimates 
of 
	\eqref{RNormTheta}, applying the
estimates for 
$R_j\circ \Theta_j^-g $ and 
$R_j\circ \Psi^{-2} g$ above:
\begin{align*}
 \|  
  \partial_{\rho_j} u|_{\rho_j=0}
 &\|_{W^{1/2,s}
 	(\partial\Omega_j\cap \partial\Omega,
 	\rho_{\hat{j}},\lambda)}\\
&\lesssim 
 \left\|
  g
\right\|_{W^{0,s}(\Omega,\rho_{\hat{j}},
	\lambda)}
 +
\sum_k  \|  
\partial_{\rho_k} u|_{\rho_k=0}
\|_{W^{0,s}
	(\partial\Omega_k\cap \partial\Omega,
	\rho_{\hat{k}},\lambda)}.
\end{align*}
 We used the estimates
  in \eqref{RbInftyEst} to estimate
the smooth terms, $R_{bj}^{-\infty}$ from
 \eqref{RNormTheta}.

Summing over all boundaries $\partial\Omega_j$
 and
 using
\begin{multline*}
  \|  
 \partial_{\rho_k} u|_{\rho_k=0}
 \|_{W^{0,s}
 	(\partial\Omega_k\cap \partial\Omega,
 	\rho_{\hat{k}},\lambda)}
  \lesssim\\ s.c. 
  \|  
 \partial_{\rho_k} v|_{\rho_k=0}
 \|_{W^{1/2,s}
 	(\partial\Omega_k\cap \partial\Omega,
 	\rho_{\hat{k}},\lambda)}+
 \|  
\partial_{\rho_k} u|_{\rho_k=0}
\|_{W^{-\infty}(\partial\Omega)}
\end{multline*}
  to bring the last sum on the right
to the left hand side yields
\begin{equation*}
\sum_j
  \|  
 \partial_{\rho_j} v|_{\rho_j=0}
 \|_{W^{1/2,s}
 	(\partial\Omega_j\cap \partial\Omega,
 	\rho_{\hat{j}},\lambda)}
 \lesssim
\sum_j
 \left\|
 g
 \right\|_{W^{0,s}(\Omega,\rho_{\hat{j}},\lambda)}.
\end{equation*}
\end{proof}

As a corollary we can prove the 
\begin{thrm}
\label{weightedGreen}
For 
any $k$,
\begin{equation*}
\| G (g) \|_{W^{2,s}(\Omega,\rho_{\hat{k}},
	 \lambda)} 
\lesssim
\|g\|_{W^{0,s}(\Omega,\rho_{\hat{k}},
	\lambda)},
\end{equation*}
for $s\ge 0$.
\end{thrm}
\begin{proof}
As above, let $u=G(g)$.
 Estimates come by taking
$W^{2,s}$ norms of the terms on the right-hand side of
 \eqref{vExp}:
\begin{equation*}
u= 
\sum_j \Psi^{-2}
\left(\partial_{\rho_j} u 
\big|_{\rho_j =0} \right)
+  \Psi^{-2} g + R^{-\infty}.
\end{equation*}
 The first $\Psi^{-2}$ operator is decomposable,
arising as the inverse to the Laplacian.
 Therefore the estimates of 
Theorem \ref{weightedPsi} can be applied to the term
$ \Psi^{-2}  \left(\partial_{\rho_j} u 
\big|_{\rho_j =0} \right) $.
Estimates for $\Psi^{-2}g$ 
 follow as in Theorem \ref{ThetaInt},
and those for
 $R^{-\infty}$ follow as in \eqref{RbInftyEst}.
\end{proof}
Thus in the case of intersections of smooth
 domains we obtain a (weighted) gain
of two derivatives.  This
 is also the case with
  Lipschitz domains 
however with Lipschitz domains the level of
 Sobolev norms for the solution
 is restricted between 
$W^{1/2}$ and $W^{3/2}$
  (Theorem .5 \cite{JK95}).

\section{Boundary equations}
\label{secBndryEqns}

	We now return to the conditions
in 
\eqref{NeumannA}:
for $j\notin K$, 
\begin{equation*}
\Lb_j u_K + (-1)^{|K|}c_{jK}^K u_K=0
\end{equation*}
on  $\rho_{j}=0$ 
for 
$1\le j\le m$.

We write
	$ u= G(2f) + P(u_b)$.  	
From Section \ref{secDNO}, we have 
\begin{equation*}
 \partial_{\rho_j}  
 P_K(u_b)|_{\rho_j=0} =
 |D_{bj}| u_{K,bj}
 +E u_{b}
 + \Lambda_{bj}^0 u_{K,bj},
\end{equation*}
$E$ is given by the
 $|D_{bj}| \circ 
  \sum_{k} \lre^{kj}_{-1/2} u_{K,bk}
  + C_j u_{b}$ terms
in Theorem \ref{thrmDNO} 
and $\Lambda_{bj}^0$ is also as in 
Theorem \ref{thrmDNO}.

On the boundary $\partial\Omega_j$ we use the notation
 $X_{k,bj}$ to denote the complex tangential (to
  $\Omega_j$) vector fields:
for 
$1\le j\le m$,
and $k\neq j$, we set
\begin{equation*}
X_{k,bj} =  
\begin{cases}
L_{k} \Big|_{\rho_j=0} & \mbox{ if } 1\le k
\le m\\
V_{k} \Big|_{\rho_j=0} &  \mbox{ if } 
m+1\le k \le n.
\end{cases}
\end{equation*}

 $\Lambda^0_{bj}$ is a matrix of pseudodifferential operators, 
and
 we write $\sigma(\Lambda^0_{bj})_{K,K}$ to refer to the
symbol in the $(K,K)^{th}$ entry of the 
 matrix symbol $\sigma(\Lambda^0_{bj})$.

 We will use 
$ | \cdot |_{\lrl_j}$ to 
 denote the Levi-norm length of 
  a vector field: the Levi norm is given
 by
 Levi metric, which 
is
defined by
\begin{equation*}
ds^2_j = 
\sum \frac{\partial^2 \rho_j}{\partial z_k \zb_l} dz_k d\zb_l,
\end{equation*}
 and the norm of a vector 
 field,
\begin{equation*}
 Z = \sum \gamma_j 
 \frac{\partial}{\partial z_j}
\end{equation*}
 with respect to this metric will be written
  as
\begin{equation*}
 |Z|_{\lrl_j} := \sum \frac{\partial^2 \rho_j}{\partial z_k \zb_l}
 \gamma_k  \overline{\gamma}_l.
\end{equation*}

 We can use Proposition 5.1 of \cite{Eh18_dno},
which relates the symbols of the $S_j$, $A$, and 
 $\tau_j$ operators of \eqref{localSAT} with the
terms $c^{K}_{jK}$ as well as the Levi-norms 
 (on $\partial\Omega_j$) of tangential vectors,
to find the symbol of $\Lambda_{bj}^0$.  
 On the diagonal of $\Lambda_{bj}^0$
 we have
\begin{align*}
\sigma(\Lambda^0_{bj})_{K,K}=& -(-1)^{|K|} \sqrt{2} c^K_{jK}
+  
\sum_{k\notin K} |X_{k,bj}|_{\lrl_j}^2
-  \sum_{k\in K} |X_{k,bj}|_{\lrl_j}^2\\
&
+O(x,\rho_{\hat{j}}) +O\left(
\frac{\sqrt{\eta_{\hat{j}}^2+
 \xi_{\hat{j}}^2  }}
 {\sqrt{\eta_{\hat{j}}^2+\Xi_{bj}^2}}
\right)  
\end{align*}
in a microlocal neighborhood 
 defined by the set product of a neighborhood
of the origin with the support of the
 $\psi^{-}_{N, bj}$ symbol 
defined in Section \ref{secDNO}. 
 There will also be some
 entries off the diagonal for $\sigma(\Lambda_{bj}^0)$,
 which arise from the 
contribution of the first order operators,
$
- \varepsilon^{kK_{\hat{j}} }_{K_{\hat{j}}\cup \{k\}}
  \varepsilon^{jK_{\hat{j}} }_{K}
[ \Wb_j, W_k]
u_{K_{\hat{j}}k}\omegab_K, $
 ($j\in K$)
 in 
 Proposition \ref{squareOp} $ii)$ to the  
 terms off the
diagonal in the symbol
 $\sigma(A) = \alpha(\rho,x,\xi)$.  We note
 for now that
the commutators, $[ \Wb_j, W_k]$, are tangential with respect
 to $\partial\Omega_j$.

Thus, for $1\le j\le m$, we can write
\begin{align}
 \nonumber
 \frac{1}{\sqrt{2}}\left(\Lambda_{bj}^0\right)_{K,K}& 
  + (-1)^{|K|}c_{jK}^K =\\
  \label{forZeroDNO} 
 \frac{1}{\sqrt{2}}
  \sum_{k\notin K} &|X_{k,bj}|_{\lrl}^2
 -  \frac{1}{\sqrt{2}} \sum_{k\in K} |X_{k,bj}|_{\lrl}^2
 +O(x,\rho_{\hat{j}}) +O\left(
\frac{\sqrt{\eta_{\hat{j}}^2+
		\xi_{\hat{j}}^2  }}
 {\sqrt{\eta_{\hat{j}}^2+\Xi_{bj}^2}}
 \right)
\end{align}
	on $\rho_j = 0$ 
in a microlocal neighborhood in which
 $\xi_{j} \ll -
 \sqrt{\eta_{\hat{j}}^2+\xi_{\hat{j}}^2}$.

From this point forward, we 
 work with $(0,1)$-forms in 
 $\mathbb{C}^2$.  We set 
  $n=m=2$, and look at resulting simplifications in 
the boundary equations.  We first deal with the non-diagonal
 tangential operators contained in 
$\Lambda_{bj}^0$.  Without loss of generality 
 we work on the particular boundary
 $\partial\Omega_1$.  We write
\begin{equation*}
 u = u_1 \omegab_1 + u_2\omegab_2.
\end{equation*}
  On 
 $\partial\Omega_1$, 
 \eqref{dirBndry1} leads to
$u_1|_{\rho_1=0} =0$.   We will also use the 
 notation
$u_{bj} = u_{1,bj} \omegab_1 + u_{2,bj} \omegab_2$.
From $u = P(u_b) + G(2f)$ we write
\begin{equation*}
 u_j = P_j(u_b) + G_j(2f)
\end{equation*}
for $j=1,2$.

Turning to \eqref{NeumannA} on
 $\partial\Omega_1$, we examine in more
detail the normal derivative,
\begin{align*}
 R_1 \circ \Lb_1 u_2 =& R_1 \circ \Lb_1 (P_2(u_b) +G_2(2f) )\\
  =& \frac{1}{\sqrt{2}} \left( N_1^- u_b \right)_2
  -iT_1 u_{2,b1} 
   +\sqrt{2} \partial_{\rho_1} G_2 f,
\end{align*}
where $N_1^-$ is given as in 
 Theorem \ref{thrmDNO},
  and  we use a subscript 
around the $N_1^-u_b$ term to denote which
  component (of the vector result) we are taking.  Thus
\begin{equation*}
 \left( N_1^- u_b \right)_2
  := N_1^- u_b \rfloor \omegab_2.
\end{equation*}
Taking the $\omegab_2$ components, 
 we obtain
\begin{multline*}
R_1\circ \left(N_1^-u_{b}\right)_2
=\\ |D_{b1}| u_{2,b1} + 
  2
 |D_{b1}|\circ R_1 \circ 
 \Theta_2^+ \circ R_2 \circ 
 \Theta_1^+ u_{2,b1}
+ \left(\Lambda_{b1}^0u_{b1}\right)_2
 + R_1\circ \left(C_1 u_b\right)_2,
\end{multline*}
and $C_1 u_b$ can be estimated as in the
 theorem
(assuming a microlocal neighborhood 
 in which $|\xi_1| \gg \sqrt{\eta_2^2+\xi_2^2}$).
For convenience, we group the
 second term on the right with the
last error term, and write
\begin{equation}
\label{defnE}
Eu_{b}
= 2
|D_{b1}|\circ R_1 \circ 
\Theta_2^+ \circ R_2 \circ 
\Theta_1^+ u_{b}
+ R_1\circ C_1u_b.
\end{equation}
so that
\begin{equation*}
 \left(Eu_{b}\right)_2
  = 2
  |D_{b1}|\circ R_1 \circ 
  \Theta_2^+ \circ R_2 \circ 
  \Theta_1^+ u_{2,b1}
   + R_1\circ \left( C_1u_b\right)_2.
\end{equation*}

Furthermore, as mentioned above, the terms off the 
 diagonal of the operator 
  $\Lambda_{b1}^0$ are tangential with respect to
$\partial\Omega_1$, and so 
\begin{align*}
 \Lambda_{b1}^0 u_{b1} 
   =&  \left[\left(
     \Lambda_{b1}^0 
   \right)_{1,1} u_{1,b1} 
   +
   \left(
   \Lambda_{b1}^0 
   \right)_{1,2} u_{2,b1}
    \right] \omegab_1\\
 &+\left[\left(
 \Lambda_{b1}^0 
 \right)_{2,1} u_{1,b1} 
 +
 \left(
 \Lambda_{b1}^0 
 \right)_{2,2} u_{2,b1}
 \right] \omegab_2 \\
=&      \left[
\left(
\Lambda_{b1}^0 
\right)_{1,2} u_{2,b1}
\right] \omegab_1
 +\left[
\left(
\Lambda_{b1}^0 
\right)_{2,2} u_{2,b1}
\right] \omegab_2.
\end{align*} 
Hence,
\begin{equation*}
\left(  \Lambda_{b1}^0 u_{b1} \right)_2 
  =  \left(\Lambda_{b1}^0 
  \right)_{2,2} u_{2,b1}.
\end{equation*}

We let $\zeta(x,\rho_2)$ have compact support in a neighborhood of 0
on $\partial\Omega_1$
 which provides a coordinate as in Section \ref{secSetup},
with $\zeta\equiv 1 $ near 0. 
Dividing Fourier space into regions as
 described in Section \ref{secNotation}, with
symbols $\psi^+_k$ with support in
$\xi_{1}>k \sqrt{\eta_2^2 + \xi_2^2}$,
 $\psi^-_k$ with support in
 $\xi_{1}<-k \sqrt{\eta_2^2 + \xi_2^2}$,
and $\psi^0_k$ defined by 
 $\psi^+_k + \psi^0_k+ \psi^-_k =1$, we write
\begin{equation*}
 u_b = u_b^+ + u_b^0 + u_b^-,
\end{equation*}
in a small neighborhood of
 $0\in \partial\Omega_1$,
where $u_b^+ := \Psi_k^+ u_b$,
 $\Psi^+_k$ being the operator with symbol 
\begin{equation}
\label{symMic}
 \sigma(\Psi^+_k) = \zeta(x,\rho_2)\psi^+_k(\xi,\eta_2),
\end{equation}
 and with
$u^0_b$ and $u^-_b$ defined similarly, with the
 same cutoff, $\zeta(x,\rho_2)$, in base space.

 We apply the operators
$\Psi^+_k$, respectively $\Psi^0_k$ to
 both sides of the boundary condition 
(on $\partial\Omega_1$)
 \begin{equation*}
 \Lb_1 u_2-c_{12}^2 u_2=0
 \end{equation*}
to write
\begin{multline}
 \label{bndryPlus}
 \left(\frac{1}{\sqrt{2}}|D_{b1}| -i T_1^0
\right) u_{2,b1}^{+} + \Psi^0_{b1}u_b+
E^{+}_2 u_b
= - \sqrt{2} R_1\circ \Psi^+_k\circ \partial_{\rho_1} \circ  G_2f,
\end{multline}
 where $T_1^0=R_1\circ T_1$,
 and $E^+_2 u_b:=\Psi_k^+\circ (Eu_b)_2$, $E$ being defined above in \eqref{defnE}.
  
 Similar notation is used to write
  $u_{2,b1}^0$.  
We use the notation 
 $u_{2,b1}^{\ast}$ to denote either 
  $u_{2,b1}^+$ or $u_{2,b1}^0$.
The first order operator, 
$\frac{1}{\sqrt{2}}|D_{b1}| -iT_1^0$ 
is elliptic in the regions of 
support $\psi^0_k$ and $\psi^+_k$, and so
 leads to (weighted) estimates with
 a gain of a derivative:
 \begin{equation}
\label{goodEst}
 \begin{aligned}
 \| u_{2,b1}^{\ast} &\|_{W^{1,s}(\partial\Omega_1,\rho_{2},\lambda)}\\
& \lesssim 
  \| u_{2,b1} \|_{W^{0,s}(\partial\Omega_1,\rho_{2},
  	\lambda)}
 +  \| u_{1,b2} \|_{W^{0,s}(\partial\Omega_2,\rho_{1},
 	\lambda)}
\\
 &\qquad+ \| R_1 \circ \partial_{\rho_1}\circ G_2f 
 \|_{W^{0,s}(\partial\Omega_1,\rho_{2},\lambda)}
 +\| E u_b
 \|_{W^{0,s}(\partial\Omega_1,\rho_{2},\lambda)}  \\
& \lesssim 
 \| u_{2,b1}
 \|_{W^{0,s}(\partial\Omega_1,\rho_{2},\lambda)}
 +  \| u_{1,b2} \|_{W^{0,s}(\partial\Omega_2,\rho_{1},\lambda)}
 + \| f 
 \|_{W^{0,s}(\Omega,\rho_{2},\lambda)},
\end{aligned}
 \end{equation}
 where we use Corollary 
 \ref{corGreenNormDerEst} in the
 last step.  In fact, the same reasoning shows
for operators 
 $\Psi^+_{k_1}$ and
 $\Psi^0_{k_2}$, defined by
symbols   $\psi^+_{k_1}$ and
$\psi^0_{k_2}$, respectively,
 with $k_1< k < k_2$, and with
 the properties
$\psi^+_{k} \equiv 1 $
 on the support of  $\psi^+_{k_2}$,
and similarly,
$\psi^0_{k} \equiv 1 $
on the support of  $\psi^0_{k_1}$,
  we have
 \begin{equation}
  \label{EstPlusZero}
  \begin{aligned} 
\|\Psi^+_{k_1} u_{2,b1} \|_{W^{1,s}(\partial\Omega_1,\rho_{2},\lambda)}
\lesssim &
\| u_{2,b1}
\|_{W^{0,s}(\partial\Omega_1,\rho_{2})}
+  \| u_{1,b2} \|_{W^{0,s}(\partial\Omega_2,\rho_{1},\lambda)}\\
& + \| f 
\|_{W^{0,s}(\Omega,\rho_{2})}\\
\|\Psi^0_{k_2} u_{2,b1} \|_{W^{1,s}(\partial\Omega_1,\rho_{2},\lambda)}
\lesssim &
\| u_{2,b1}
\|_{W^{0,s}(\partial\Omega_1,\rho_{2},\lambda)}
+  \| u_{1,b2} \|_{W^{0,s}(\partial\Omega_2,\rho_{1},\lambda)}\\
& + \| f 
\|_{W^{0,s}(\Omega,\rho_{2},\lambda)}.
   \end{aligned}
\end{equation}  

We now write $u_{2,b1} = 
 u_{2,b1}^+ + u_{2,b1}^0 + u_{2,b1}^-$
 in
the boundary condition 
\begin{equation*}
\Lb_1 u_2-c_{12}^2 u_2=0
\end{equation*}
and use
\begin{align*}
 R_1\circ (\Lb_1 -c_{12}^2) u_{2}^- =& R_1\circ \Psi^-_k \circ 
   (\Lb_1 -c_{12}^2) u_{2} + 
  R_1\circ\left(  \left[ (\Lb_1 -c_{12}^2), \Psi^-_k\right]
    u_2\right)\\
    =& \left[ (\Lb_1 -c_{12}^2), \Psi^-_k\right]
    u_{2,b1}.
\end{align*}
If we let $\varphi(x,\rho_2)\in 
 C{\infty}_0(\partial\Omega_1)$ be such that
the support of $\varphi$ is contained in the
 region where $\zeta\equiv 1$, then
the support of the symbol,
$
 \varphi(x,\rho_2)\sigma \left(
 	\left[ (\Lb_1 -c_{12}^2), \Psi^-_k\right]
 \right) 
$, is contained in the region where
   $\sigma(\Psi^0_{k_2}) \equiv 1$
for some $ k_2>k$.
Let us fix a notation for such an operator.
 Let $\Psi^r_{b1,\psi^0}$ denote an operator
of order $r$ on
 $\partial\Omega_1$ whose symbol has support
 in the region where
 $\sigma(\Psi^0_{k_2}) \equiv 1$.

For such an operator, we can
 use \eqref{EstPlusZero} to conclude
\begin{align*}
\nonumber
\big\|
\Psi_{b1,\psi^0}^0 \big\|_{W^{1,s}(\partial\Omega_1, \rho_2,\lambda)}
\lesssim& 
\| u_{2,b1}
\|_{W^{0,s}(\partial\Omega_1,\rho_{2},\lambda)}
+  \| u_{1,b2} \|_{W^{0,s}(\partial\Omega_2,\rho_{1},\lambda)}\\
\nonumber
& + \| f 
\|_{W^{0,s}(\Omega,\rho_{2},\lambda)}.
\end{align*}

We write in a similar manner
 to \eqref{bndryPlus} 
the boundary condition in the
 microlocal region determined by $\Psi^-_k$:
\begin{multline}
 \label{bndryCondn}
 \left(\frac{1}{\sqrt{2}}|D_{b1}| -i T_1^0
 \right) u_{2,b1}^- +
 \frac{1}{\sqrt{2}}\left(
 \sum_{k\notin K} |X_{k,b1}|_{\lrl_1}^2
 -  \sum_{k\in K} |X_{k,b1}|_{\lrl_1}^2
 \right) u_{2,b1}^- \\
+\Psi^0_{b1,\psi^0}
u_{2,b1}
 +\Psi^0_{\varepsilon,b1} u_b
+ E^{-}_2 u_b
= - \sqrt{2} R_1\circ 
  \Psi^-_k\circ \partial_{\rho_1} \circ  G_2f,
\end{multline}
(in the support of a cutoff, $\varphi$, as outlined above),
 where we use the notation,
$\Psi^0_{\varepsilon,b1}$ as outlined in 
\eqref{smallOps},
 the $\varepsilon$ signifying
the property
\begin{equation*}
\| \Psi_{\varepsilon,b1}^0 \phi\|_{W^s(\mathbb{R}^3)}
\le s.c. \| \phi \|_{W^{s}(\mathbb{R}^3)}
\end{equation*}
where the constant of inequality,
written as "s.c.", can be made arbitrarily
small by choosing an appropriately large 
constant, $k$, for the functions,
$\psi^+_k$, $\psi^0_k$, and $\psi^-_k$, used 
to divide Fourier space above. 
Recall that the sums of norms of vector fields on the
  left-hand side
 of \eqref{bndryCondn} come from the zero order term of the DNO
as in \eqref{forZeroDNO}.
 For these norms, we use the notation
  $\lrl_1$ to denote the Levi form is used
with respect to the defining function,
 $\rho_1$, on $\partial\Omega_1$.

From the estimates in 
 \eqref{EstPlusZero}, it suffices to consider 
\eqref{bndryCondn}, and get estimates for
 $u_b^-$.

On 
$\partial\Omega_1$, the 
vector field, $X_{2,b1}$ is
given by $L_{2,b1}$, while
$X_{1,b1}$ does not exist.  Thus 
\begin{equation*}
\sum_{k\notin K} |X_{k,bj}|_{\lrl_1}^2
-  \sum_{k\in K} |X_{k,bj}|_{\lrl_1}^2
=-|L_{2,b1}|_{\lrl_1}^2.
\end{equation*} 
The strict pseudoconvexity condition
gives 
\begin{equation*}
|L_{2,b1}|_{\lrl_1}^2= \beta > 0
\end{equation*}
on $\partial\Omega_1$. 
The boundary condition,
 \eqref{bndryCondn}, on
$\partial\Omega_1$ thus becomes
\begin{multline}
\label{case1Bndry1}
 \left(\frac{1}{\sqrt{2}}|D_{b1}| -i T_1^0
\right) u_{2,b1}^- -
\frac{\beta}{\sqrt{2}} u_{2,b1}^- \\
+\Psi^0_{b1,\psi^0}
u_{2,b1} +\Psi^0_{\varepsilon,b1} u_{2,b1}^-
  +E^-_2 u_{b}
= - \sqrt{2} R_1\circ \Psi^-_k \circ 
 \partial_{\rho_1} \circ  G_2f.
\end{multline}

The boundary condition on $\partial\Omega_2$ is of course
 symmetric, so it suffices to obtain
estimates for $u_{2}$ on $\partial\Omega_1$.

We now apply $\frac{1}{\sqrt{2}}|D_{b1}| +i T_1^0$
 to both sides of \eqref{case1Bndry1} above.
We use 
\begin{equation}
 \label{conjMult}
 \bigg(
\frac{1}{\sqrt{2}}|D_{b1}| +i T_1^0
 \bigg)
 \bigg(
\frac{1}{\sqrt{2}}|D_{b1}| -i T_1^0
 \bigg)
	=
\left(
\frac{1}{2} D_{b1}^2+(T_1^0)^2
\right)
+ \Psi^1_{b1},
\end{equation}
where $\Psi^1_{b1}$ is given
 by $\frac{i}{\sqrt{2}}[T_1^0, |D_{b1}|]$.
The symbol of this first order term has the
 property
\begin{equation*}
 \sigma_1\left(
  [T_1^0, |D_{b1}|]
 \right) = O\left(
   \sqrt{\eta_2^2+ \xi_{2}^2}
 \right) + O(x,\rho_2)
\end{equation*}
from \eqref{dxXi}
 (see also Proposition 3.4 in 
  \cite{CNS92} or 
 Section 6 of \cite{Eh18_dno}).
We thus write the first order operator in
 \eqref{conjMult} above as
$\Psi_{\varepsilon,b1}^1$.
 
We also note that 
 the operator,
$
	\frac{1}{2} D_{b1}^2+(T_1^0)^2
$, is given by
\begin{align*}
\frac{1}{2} D_{b1}^2+(T_1^0)^2
 =&   
 L_{2,b1} \Lb_{2,b1}
  +i \sqrt{2} |L_{2,b1}|_{\lrl_1}^2
   T_1^0
  +\Psi^1_{\varepsilon,b1}\\
   =&   
  L_{2,b1} \Lb_{2,b1}
  +i\sqrt{2} \beta T_1^0
  +\Psi^1_{\varepsilon,b1}
\end{align*}
(see the discussion in \cite{Eh18_dno}).

After applying
 $ \frac{1}{\sqrt{2}}|D_{b1}| +i T_1^0$ to both sides of
\eqref{case1Bndry1}, we thus have
\begin{align}
\nonumber
  L_{2,b1} \Lb_{2,b1} &
   u_{2,b1}^-
 +i\sqrt{2} \beta T_1^0 u_{2,b1}^-
- \frac{\beta}{\sqrt{2}} 
  \left(
     \frac{1}{\sqrt{2}}|D_{b1}| +i T_1^0
  \right)u_{2,b1}^- \\
\nonumber
  =& -\sqrt{2} R_1 \circ \left(
\frac{1}{\sqrt{2}}|D_{b1}| +i T_1^0
  \right)\circ \Psi^-_k \circ \partial_{\rho_1} G_2 f\\
\label{bndryEqnFormNull}
& 
+\Psi^1_{b1,\psi^0}
u_{2,b1}
+\Psi^1_{\varepsilon,b1} u_{2,b1}^-
+\Psi^{0}_{b1} u_{2,b1}
+\Psi^1_{b1} E^-_2 u_b  
  .
\end{align}
Expanding the (symbols of the) 
 last two terms on the left hand
 side of \eqref{bndryEqnFormNull} for
large (negative) $\xi_1$, we see the terms cancel,
 modulo operators of small operator
  norm:
\begin{align}
\nonumber
L_{2,b1} \Lb_{2,b1} 
u_{2,b1}^-
=& -\sqrt{2} R_1 \circ \left(
\frac{1}{\sqrt{2}}|D_{b1}| +i T_1^0
\right)\circ \Psi^-_k \circ \partial_{\rho_1} G_2 f
\\
\label{bndryEqnForm}
&
+\Psi^1_{b1,\psi^0}
u_{2,b1}
 +\Psi^1_{\varepsilon,b1} u_{2,b1}^-
+\Psi^{0}_{b1} u_{2,b1}
+\Psi^1_{b1} E^-_2 u_b  
.
\end{align}

\section{A priori weighted $L^2$ boundary estimates}
\label{secWeightedEst}

To show estimates for the boundary solution,
 $u_b$,
we start with the property
\begin{equation}
\label{baseSpace}
u_{k,bj}\in L^2
\left(
\partial\Omega_j \cap \partial\Omega
\right)
\end{equation} 
for $k=1,2$.

In the weighted estimates we
 start with the base case, $s=0$.  
 From \cite{MS98} (Theorem 3.1), we have
the solution, $u$, to the \dbar-Neumann problem
  is in $W^{1/2}(\Omega)$ as are
$\mdbar u$ and $\mdbar^{\ast}u$.  We 
now follow
 the proof in Lemma 5.2.3 of \cite{CS} to
show the boundary estimates in
 \eqref{baseSpace}.

We write
\begin{equation*}
\| u_b \|_{L^2(\partial\Omega)} 
:= \left\| u_{b1}
\right\|_{L^2(\partial\Omega_1\cap \partial\Omega)}
+\left\| u_{b2}
\right\|_{L^2(\partial\Omega_2\cap \partial\Omega)}.
\end{equation*}
We will also use the short-hand
\begin{equation*}
 \| \rho^k u_b\|_{W^r(\partial\Omega)}
  :=
\left\| \rho_2^k u_{b1}
\right\|_{W^r(\partial\Omega_1\cap \partial\Omega)}
+\left\|\rho_1^k u_{b2}
\right\|_{W^r(\partial\Omega_2\cap \partial\Omega)}.
\end{equation*}

We let $\Lambda_{bj}^{1/2} \in 
 \Psi^{1/2}(\partial\Omega_j\cap \partial\Omega)$
for $j=1,2$
denote the operator with symbol
\begin{equation*}
 \sigma(\Lambda_{bj}^{1/2})
  = \left(1+
   \xi^2 + \eta_k^2
  \right)^{1/2},
\end{equation*}
where $k\neq j$,
 and $\xi^2 = \xi_1^2 + \xi_2^2$,
i.e. $\sigma(\Lambda_{bj}^{1/2})
 \simeq 1+ \sqrt{\eta_{k}^2+
  \Xi_{bj}^2}$.
We recall the superscript
 $E_k$ meaning extension by zero
  across $\rho_k=0$
as in Section \ref{secNotation}, and
 we note
\begin{align}
\nonumber
 \|\Lambda_{bj}^{1/2} u_{bj}^{E_k}
\|_{L^{2}(\partial\Omega_j)}
 \simeq&
 \| u_{bj}^{E_k}
  \|_{W^{1/2}(\partial\Omega_j)}\\
\label{lambdaExtn}
  \simeq&
 \| u_{bj}
 \|_{W^{1/2}(\partial\Omega_j\cap\partial\Omega)}
\end{align}
by the Extension Theorem, 
 Theorem 1.4.2.4 in \cite{Gr}.

For each $u_{bj}$ we write
 as in \cite{CS}
\begin{align}
\nonumber
\left\| u_{bj} 
\right\|_{L^2(\partial\Omega_j\cap \partial\Omega)} 
 =&
\left\| u_{bj}^{E_k}
\right\|_{L^2(\partial\Omega_j)}  \\
\label{withTang}
\lesssim&
\left\| \Lambda_{bj}^{1/2} u^{E_k}
\right\|_{L^2(\Omega_j)}  
+ \left\| \Lambda_{bj}^{-1/2} 
\partial_{\rho_j}u^{E_k}
\right\|_{L^2(\Omega_j)}.
\end{align}
Note that we have
\begin{equation*}
\partial_{\rho_j}u^{E_k} 
 =\left(\partial_{\rho_j}u\right)^{E_k},
\end{equation*}
and from \eqref{lambdaExtn},
\begin{equation*}
\left\| \Lambda_{bj}^{1/2} u^{E_k}
\right\|_{L^2(\Omega_j)}  
 \simeq 
\left\| \left(\Lambda_{bj}^{1/2} u\right)^{E_k}
\right\|_{L^2(\Omega_j)}.
\end{equation*}
so that \eqref{withTang}
shows 
$\left\| u_{bj} 
\right\|_{L^2(\partial\Omega_j\cap \partial\Omega)} $ is bounded
 by a sum of terms of the form
\begin{align*}
 \left\| \left(\Lambda_{bj}^{1/2} u\right)^{E_k}
 \right\|_{L^2(\Omega_j)}+
 \left\| \Lambda_{bj}^{-1/2} 
 (\partial_{\rho_j} u)^{E_k}
\right\|_{L^2(\Omega_j)}
\end{align*}
Now write the operator, $\partial_{\rho_j}$, 
 as a combination of 
(components) of $\mdbar$, $\mdbar^{\ast}$,
 and tangential (to $\partial\Omega_j$) operators:
\begin{align*}
 \left\| \Lambda_{bj}^{-1/2} (\partial_{\rho_j} u)^{E_k}
\right\|_{L^2(\Omega_j)}
 \lesssim&  \left\| \Lambda_{bj}^{-1/2}
 \left( \mdbar u\right)^{E_k}
 \right\|_{L^2(\Omega_j)}
+
\left\| \Lambda_{bj}^{-1/2}\left( \mdbar^{\ast} u
\right)^{E_k}
\right\|_{L^2(\Omega_j)}\\
&+
\left\| \Lambda_{bj}^{1/2} u^{E_k}
\right\|_{L^2(\Omega_j)}\\
\lesssim&
  \left\|
 \mdbar u
 \right\|_{L^2(\Omega)}
 +
 \left\|  \mdbar^{\ast} u
 \right\|_{L^2(\Omega)}
 + \|u\|_{W^{1/2}(\Omega)}
.
\end{align*}
Combining this with \eqref{withTang}
 as well as Theorem 3.1 of \cite{MS98}, we get
$ u_{b} \in L^2(\partial\Omega)$, with estimates
\begin{equation}
\label{l2Base}
 \|u_b\|_{L^2(\partial\Omega)} 
  \lesssim \|f\|_{L^2(\Omega)}.
\end{equation}

We will concentrate on the 
 more difficult estimates in the 
microlocal region defined by 
 $\psi^-_k$; that is we assume
$-i\sigma(T_1) < - k \sqrt{|\sigma(T_2)|^2 +
	 |\sigma(\partial_{\rho_2})|^2}$.
We will also just consider the boundary,
$\partial\Omega_1$, the results being
 analogous on $\partial\Omega_2$.
We will thus drop the subscripts, writing
 $L_{b1}$ to mean $L_{2,b1}$
 in \eqref{bndryEqnForm}.
  From \eqref{bndryEqnForm}, the boundary 
equation reads
\begin{equation}
\label{bndrySub}
L_{b1} \Lb_{b1} u_{2,b1}^-
 = g
+\Psi^1_{b1,\psi^0}
u_{2,b1} 
 + \Psi^1_{\varepsilon,b1} u_{2,b1}^-
 +\Psi^{0}_{b1} u_{2,b1} + \Psi^1_{b1} E^-_2 u_b,
\end{equation}
where $g$ is of the form
 $ R_1\circ \Psi^1_{b1}\circ \partial_{\rho_1} \circ G_2f$.

 We first start with the assumption 
that $u_{b}\in C^{\infty}(\partial
 \Omega)$ so that integration by
parts can be performed.  
This is not
 a necessary assumption as the use
of regularizing operators could be used
 from the beginning (see \cite{K85}).  
We separate 
the regularizing argument in this paper, mainly
out of aesthetic concerns, but also due
 to the more complicated regularizing
of terms involving the 
 $E^-_2$ operator in \eqref{bndrySub}.
The assumption of smooth 
 forms will be removed in Section
\ref{secRegularize}.
 
Using an argument of Kohn \cite{K85},
we write
\begin{equation*}
T_1 = \frac{i}{\lambda}
\left(
L_{b1}\Lb_{b1} - \Lb_{b1} L_{b1}
\right)
\end{equation*} 
modulo $L_{b1}$ and $\Lb_{b1}$,
 where $\lambda>0$ on $\partial \Omega_1$.  
Sobolev $1/2$ estimates follow as in the
smooth case:
let $\varphi \in C^{\infty}(\partial\Omega_1)$ with 
 support away from 
$\rho_2\ge 0$, but still contained in 
 the region where $\zeta\equiv 1$, where
$\zeta$ is the cutoff defining the operator
 $\Psi^{-}_k$ as in 
  \eqref{symMic}.    Then we have by integration by 
 parts
\begin{equation}
\label{firstStepHalfEst}
\begin{aligned}
\|  \varphi u_{2,b1}^- \|_{1/2}^2
\lesssim& \left|(T_1 \varphi u_{2,b1}^-, \varphi u_{2,b1}^-) \right|\\
\lesssim& \| L_{b1} \varphi u_{2,b1}^- \|^2 +
\|\Lb_{b1} \varphi u_{2,b1}^-\|^2 + \|\varphi u_{2,b1}^-\|^2\\
\lesssim& \| \varphi  L_{b1}v \|^2 +
\| \varphi \Lb_{b1} u_{2,b1}^-\|^2 + \| u_{2,b1}^-\|^2.
\end{aligned}
\end{equation}
 The $L^2$-norms are with respect to 
$\partial\Omega_1$, and we
 write
$\| \cdot \|_{1/2}$ as a shorthand for
 $\| \cdot \|_{W^{1/2}(\partial\Omega_1)}$.  
In the case we need to 
 specify one boundary norm over another,
we will write explicitly the domain on which 
 the norms are calculated; otherwise 
the boundary, $\partial\Omega_1$, is to be
 the default.

We also have
\begin{align}
\nonumber
\|\varphi L_{b1} u_{2,b1}^- \|^2 +
\|\varphi &\Lb_{b1} u_{2,b1}^-\|^2 \lesssim\\
\nonumber
&-(\varphi \Lb_{b1} L_{b1} u_{2,b1}^-, \varphi u_{2,b1}^-)
+ O\left(
\|\varphi L_{b1}  u_{2,b1}^- \| \| u_{2,b1}^-\|
\right)\\
\label{estZero}
& -( \varphi L_{b1} \Lb_{b1}  u_{2,b1}^-, \varphi u_{2,b1}^-)
+ O\left(
\|\varphi \Lb_{b1}  u_{2,b1}^- \| \| u_{2,b1}^-\|
\right).
\end{align}

For the term
$\Lb_{b1} L_{b1} u_{2,b1}^-$, we write
 (in the support of $\varphi$)
\begin{align}
\nonumber
\Lb_{b1} L_{b1} u_{2,b1}^-
=& L_{b1}\Lb_{b1} u_{2,b1}^- + [\Lb_{b1},L_{b1}] u_{2,b1}^-\\
\label{commuteL}
=& g  + i\lambda T_1 u_{2,b1}^- 
+\Psi^1_{b1,\psi^0}
u_{2,b1} 
+
\Psi^{1}_{\varepsilon,b1}u_{2,b1}^-
+\Psi^{0}u_{2,b1}
 +\Psi^1_{b1}E^-_2 u_b.
\end{align}
In the support of $\psi^-_k$, the symbol
of $-iT_1$ is negative, hence we can write
\begin{equation*}
-(\varphi i\lambda T_1 u_{2,b1}^-,\varphi u_{2,b1}^-)
\lesssim \|u_{2,b1}^-\|^2,
\end{equation*}
which also follows from
G\aa rding's inequality.

Using this in \eqref{estZero} above,
we have
\begin{align*}
\|\varphi L_{b1} u_{2,b1}^- \|^2 +&
\|\varphi \Lb_{b1} u_{2,b1}^-\|^2 \lesssim\\
&2|(\varphi g, \varphi u_{2,b1}^-)|
+ | (\varphi
\Psi^1_{b1,\psi^0} u_{2,b1}, \varphi u_{2,b1})|
+ | (\varphi
\Psi^1_{\varepsilon,b1} u_{b}^-, \varphi u_{2,b1}^-)|\\
&+ \left| \left(\varphi \Psi^1_{b1} E^- u_b, u_{2,b1}^- \right)
 \right|  + \|u_{2,b1}^-\|^2
 \\
&+ O\left(
(
\|\varphi \Lb_{b1}  u_{2,b1}^- \| + \|\varphi L_{b1} u_{2,b1}^-\|)
\| u_{2,b1}^-\|
\right).
\end{align*}
Estimates of the terms on the right-side
 then yield estimates for
$L_{b1}$ and $\Lb_{b1}$ applied to $u_{2,b1}^-$
 which in turn yield $1/2$-estimates 
  for $u_{2,b1}^-$
(when combined with estimates for
 $u_{1,b2}^-$).  We apply this 
approach, but with
 weights, in order to obtain estimates
near a singularity.

We now look at
estimates as in 
 \eqref{firstStepHalfEst}
but without the assumption of 
 support away from the boundary
  singularities. 
For higher order estimates,
we work with the operator,
$\Lambda$, with symbol
\begin{equation*}
|\sigma(\Lambda)| \simeq 
\sqrt{|\sigma(T_1)|^2 + 
	|\sigma(T_2)|^2+
	|\sigma(\partial_{\rho_2})|^2},
\end{equation*}
on $\partial\Omega_1$.
   So that integration
by parts can still be used, we 
 multiply by factors of $\rho_2$
(recall we work with estimates on
 $\partial\Omega_1$);
in \eqref{firstStepHalfEst} we let
 the cutoff, $\varphi$, have support in a 
  neighborhood of the singularity, assumed to be
   at 0, and introduce a 
factor, $\rho_2^{2\alpha}$ for
 some integer $2\alpha\ge 1$,
and we also replace $u_{2,b1}^-$ with
 $\Lambda^{\alpha} u_{2,b1}^-$:
\begin{multline}
\label{halfEst}
\| \rho_2^{2\alpha+1} \Lambda^{\alpha} u_{2,b1}^-\|_{1/2}^2
\lesssim\\ \|\rho_2^{2\alpha+1}  L_{b1}  \Lambda^{\alpha} u_{2,b1}^-\|^2 +
\|\rho_2^{2\alpha+1}\Lb_{b1} \Lambda^{\alpha} u_{2,b1}^-\|^2 + \|\rho_2^{2\alpha}
\Lambda^{\alpha} u_{2,b1}^-\|^2.
\end{multline}
The 
 vanishing of the boundary terms arising
in the integration by parts
  occurs due 
to our assumption that
 $u_b \in C^{\infty}(\partial\Omega)$.
  We omit writing the cutoffs $\varphi$; they could
also be understood to be part of the $\Lambda$ operators.

For integer $\alpha$,
we use
\begin{align*}
\|\rho_2^{2\alpha} \Lambda^{\alpha} 
u_{2,b1}^-\|^2
\lesssim \sum_{l\le \alpha} \|\Psi_{b1}^{\alpha-l}
 \rho_2^{2\alpha-2l}  u_{2,b1}^-\|^2
\lesssim \sum_{l\le \alpha} \| \rho_2^{2\alpha-2l}  u_{2,b1}^-
\|_{\alpha-l}^2,
\end{align*}
whereas for $\alpha$ of the form
$\alpha = (2k+1)/2$, we use
\begin{align*}
\|\rho_2^{2k+1} \Lambda^{k+1/2} 
u_{2,b1}^-\|^2
\lesssim& \sum_{l\le k} \|\rho_2 \Psi_{b1}^{k-l+1/2} \rho_2^{2k-2l}  
u_{2,b1}^-\|^2\\
\lesssim& \sum_{l\le k} \| \rho_2^{2k-2l+1}  u_{2,b1}^-\|_{k-l+1/2}^2
+ \| u_{2,b1}^-\|^2\\
\lesssim& \sum_{l \le \lfloor \alpha
	\rfloor} 
\| \rho_2^{2\alpha-2l}  u_{2,b1}^-\|_{\alpha-l}^2
+ \| u_{2,b1}^-\|^2
.
\end{align*}
In each case we write
\begin{align}
\label{case21}
\|\rho_2^{2\alpha} \Lambda^{\alpha} 
 u_{2,b1}^-\|^2
\lesssim
\sum_{l \le \lfloor\alpha\rfloor} \| \rho_2^{2\alpha-2l}  u_{2,b1}^-\|_{\alpha-l}
+ \| u_{2,b1}^-\|
.
\end{align}

We consider the first term on the
 right-hand side of 
\eqref{halfEst}, and integrate by
 parts (again, for the time being,
 assuming the
  vanishing of the arising boundary 
   integrals at $\rho_2=0$):
\begin{align*}
(\rho_2^{2\alpha+1} L_{b1} \Lambda^{\alpha} 
 u_{2,b1}^-,
\rho_2^{2\alpha+1} &L_{b1} \Lambda^{\alpha}
 u_{2,b1}^-)\\
=&- (\rho_2^{2\alpha+1} \Lb_{b1} L_{b1} \Lambda^{\alpha} u_{2,b1}^-,
 \rho_2^{2\alpha+1}
\Lambda^{\alpha} u_{2,b1}^-)\\
&+
O\left(\|\rho_2^{2\alpha+1} L_{b1} \Lambda^{\alpha} u_{2,b1}^-\| 
\|\rho_2^{2\alpha}
  \Lambda^{\alpha}u_{2,b1}^-\|\right)\\
\lesssim& -(\rho_2^{2\alpha+1} 
\Lb_{b1} L_{b1} \Lambda^{\alpha} u_{2,b1}^-, 
\rho_2^{2\alpha+1} \Lambda^{\alpha} u_{2,b1}^-)\\
&+ s.c. \|\rho_2^{2\alpha+1}
 L_{b1}\Lambda^{\alpha} u_{2,b1}^-\|^2 +
\|\rho_2^{2\alpha}\Lambda^{\alpha}
 u_{2,b1}^-\|^2,
\end{align*}
or
\begin{align*}
\|\rho_2^{2\alpha+1} L_{b1}
 \Lambda^{\alpha} u_{2,b1}^-\|^2
\lesssim& - (\rho_2^{2\alpha+1} 
\Lb_{b1} L_{b1} \Lambda^{\alpha} u_{2,b1}^-, 
\rho_2^{2\alpha+1} \Lambda^{\alpha} u_{2,b1}^-)+
\|\rho_2^{2\alpha}\Lambda^{\alpha} u_{2,b1}^-\|^2.
\end{align*}

We now commute the $\Lambda^{\alpha}$ operator through the
$L$ derivatives 
 and use G\aa rding's inequality:
\begin{align}
\nonumber
-\big(\rho_2^{2\alpha+1} \Lb_{b1}& L_{b1} \Lambda^{\alpha} u_{2,b1}^-, 
\rho_2^{2\alpha+1} \Lambda^{\alpha} u_{2,b1}^-\big)
\\
\nonumber \lesssim&
-(\rho_2^{2\alpha+1}  L_{b1} \Lb_{b1} \Lambda^{\alpha} u_{2,b1}^-, 
\rho_2^{2\alpha+1}\Lambda^{\alpha} u_{2,b1}^-) \\
\nonumber
\lesssim& 
\big|\big(\rho_2^{2\alpha+1}\Lambda^{\alpha} L_{b1} \Lb_{b1}u_{2,b1}^-,
 \rho_2^{2\alpha+1} \Lambda^{\alpha}
  u_{2,b1}^-\big)\big|+
s.c. \|  \rho_2^{2\alpha+1}
L_{b1} \Psi_{b1}^{\alpha} u_{2,b1}^- \|^2\\
\label{LLGarding}
&+s.c. \|  \rho_2^{2\alpha+1}\Psi_{b1}^{\alpha}
\Lb_{b1}  u_{2,b1}^- \|^2
+
\|\rho_2^{2\alpha+1}\Lambda^{\alpha} u_{2,b1}^-\|^2.
\end{align}

The next relation
can be derived in a similar
manner as we did \eqref{case21}:
\begin{align}
\label{case211}
 \rho_2^{2\alpha+1} \Psi_{b1}^{\alpha} w
=&
\sum_{l\le \lfloor\alpha\rfloor} 
\Psi_{b1}^{\alpha-l} \rho_2^{2\alpha+1-2l}
w  + 
\Psi_{b1}^{-1/2}\rho_2w+ \Psi_{b1}^{-1}w
\end{align}
and in particular, when 
$w = L_{b1} u_{2,b1}^-$, 
\begin{multline}
\label{scL}
\rho_2^{2\alpha+1} \Psi_{b1}^{\alpha} 
L_{b1} u_{2,b1}^-
=\\
\sum_{l\le \lfloor\alpha\rfloor}
\Psi_{b1}^{\alpha-l} \rho_2^{2\alpha+1-2l}
L_{b1} u_{2,b1}^-  + 
 \Psi_{b1}^{1/2}\rho_2 u_{2,b1}^-+
\Psi_{b1}^{0} u_{2,b1}^-
\end{multline}  
with similar inequalities
in which $\Lb_{b1}$ replaces the $L_{b1}$
derivative.

We now use
 \eqref{bndrySub}, noting that the
   $\varphi$ cutoff allows us to use an operator 
$\Psi^1_{b1,\psi^0}$ as before,
in the first term on 
the right of \eqref{LLGarding}:
\begin{align}
\nonumber
\big|\big(\rho_2^{2\alpha+1} \Lb_{b1}& L_{b1} \Lambda^{\alpha} u_{2,b1}^-, 
\rho_2^{2\alpha+1} \Lambda^{\alpha} u_{2,b1}^-\big)\big|
\\ \lesssim&  
\nonumber
\big|(\rho_2^{2\alpha+1}\Lambda^{\alpha} g, \rho_2^{2\alpha+1} \Lambda^{\alpha}
  u_{2,b1}^-)\big|+
\big|(\rho_2^{2\alpha+1}\Lambda^{\alpha}
\Psi^1_{b1,\psi^0}u_{2,b1} ,
\rho_2^{2\alpha+1} \Lambda^{\alpha}
u_{2,b1}^-)\big|\\  
\label{LLTerm}
&+
\big|(\rho_2^{2\alpha+1}\Lambda^{\alpha}
 \Psi^1_{\varepsilon,b1}u_{2,b1}^- ,
  \rho_2^{2\alpha+1} \Lambda^{\alpha}
u_{2,b1}^-)\big|\\
\nonumber
&+
 \big|(\rho_2^{2\alpha+1}\Lambda^{\alpha}
 \Psi^1_{b1} E^-_2 u_{b} ,
 \rho_2^{2\alpha+1} \Lambda^{\alpha}
 u_{2,b1}^-)\big|\\
 \nonumber
&+
s.c. \|  \rho_2^{2\alpha+1}
\Psi_{b1}^{\alpha} L_{b1}  u_{2,b1}^-\|^2
+s.c. \|  \rho_2^{2\alpha+1}\Psi_{b1}^{\alpha}
\Lb_1 u_{2,b1}^- \|^2+
\|\rho_2^{2\alpha+1}\Psi_{b1}^{\alpha}u_{b}\|^2
.
\end{align}
Using (a cruder form of \eqref{case211})
\begin{equation*}
 \rho_2^{2\alpha+1} \Psi_{b1}^{\alpha} 
=
\sum_{l\le \lceil\alpha\rceil}
\Psi_{b1}^{\alpha-l} \rho_2^{2\alpha+1-2l}
+ 
\Psi_{b1}^{-1} 
\end{equation*}
and
\begin{equation*}
 \rho_2^{2\alpha+1}\Psi^{\alpha+1}_{\varepsilon,b1}
  = 
\Psi^{\alpha+1}_{\varepsilon,b1}  \rho_2^{2\alpha+1}+
 \sum_{1\le l\le \lceil\alpha\rceil}  
 \Psi_{b1}^{\alpha-l+1} \rho_2^{2\alpha+1-2l}
   + 
 \Psi_{b1}^{0} 
\end{equation*}
in the third term on the right-hand side of
 \eqref{LLTerm}
we obtain
\begin{multline*}
 \big|
   (\rho_2^{2\alpha+1}\Lambda^{\alpha}
   \Psi^1_{\varepsilon,b1}u_{2,b1}^- ,
   \rho_2^{2\alpha+1} \Lambda^{\alpha}
   u_{2,b1}^-)
 \big| \lesssim
  s.c.\| \rho_2^{2\alpha+1} u_{2,b1}^-
  \|_{1/2+\alpha}^2\\
+ \sum_{1\le l\le \lceil\alpha\rceil}   
\| \rho_2^{2\alpha+1-2l} u_{2,b1}^-
\|_{1/2+\alpha-l}^2
 + \|  u_{b}
 \|^2.
\end{multline*}
Similarly, the second term on the right-hand side of
\eqref{LLTerm} can be estimated by
\begin{align*}
\big|(\rho_2^{2\alpha+1}\Lambda^{\alpha}
\Psi^1_{b1,\psi^0}u_{2,b1} ,
\rho_2^{2\alpha+1} \Lambda^{\alpha}
u_{2,b1}^-)\big|
 \lesssim& 
\sum_{l \le \lceil\alpha\rceil} \| \rho_2^{2\alpha+1-2l}  u_{2,b1}^0\|_{\alpha+1-l}^2
+ \| u_{2,b1}^0\|^2
\\
&+
\sum_{l \le \lfloor\alpha\rfloor} \| \rho_2^{2\alpha-2l}  u_{2,b1}^-\|_{\alpha-l}^2
+ \| u_{2,b1}^-\|^2.
\end{align*}

For the fourth term on the right
 of \eqref{LLTerm}
we distinguish the two cases
 for each of the terms composing 
$E^-_2u_b$:
$\alpha$ is of the form
 $i)$ 
$\alpha = k +1/2$, or $ii)$ $\alpha = k$,
for $k$ an integer.
We recall that
\begin{multline*}
E^-_2 u_b
= 
2 \Psi^-\circ
|D_{b1}|\circ R_1 \circ 
\Theta_2^+ \circ R_2 \circ 
\Theta_1^+ u_{2,b1}\\
+\Psi^{-1}_{b1} u_{2,b1}
+\lre^{21}_{-1/2} u_{b2}
+ \lre^{21}_{-3/2} \left(\partial_{\rho_2} 
P_2(u_b) \big|_{\partial\Omega_2}\right)
+R_{b1}^{-\infty}.
\end{multline*} 

Consider the term
 $|D_{b1}|\circ R_1 \circ 
 \Theta_2^+ \circ R_2 \circ 
 \Theta_1^+ u_{2,b1}$, which we will
write as $\lre^{11}_0u_b$.  
In case $i)$ we estimate
\begin{align*}
\big| (\rho_2^{2\alpha+1}\Lambda^{\alpha}
\Psi^1_b &\lre^{11}_0 u_{b} ,
\rho_2^{2\alpha+1} \Lambda^{\alpha}
u_{2,b1}^-)\big|
\\
\lesssim &
\sum_{0\le l\le \lfloor\alpha\rfloor+1}   
\| \rho_2^{2\alpha+2-2l} \lre^{11}_0 u_{b}
\|_{1+\alpha-l}^2
+ \|  \lre^{11}_0u_{b} \|^2\\
&+ \sum_{0\le l\le \lfloor\alpha\rfloor}   
\| \rho_2^{2\alpha-2l} u_{2,b1}^-
\|_{\alpha-l}^2
+ \|  u_{2,b1}^-\|^2
.
\end{align*}
Note that
\begin{equation*}
 \rho_2 |D_{b1}|\circ R_1 \circ 
 \Theta_2^+ \circ R_2 \circ 
 \Theta_1^+ u_{2,b1}
= \lre^{11}_{-1,1/2} u_{2,b1}
\end{equation*}
by \eqref{lowerOrder}.

In case $i)$ for the
 first term on the right of the
inequality we have
\begin{align*}
\sum_{l \le k+1}
\Big\| \rho^{2k -2l +3} 
\lre^{11}_{0} u_b
\Big\|_{1/2 + k +1 -l}
\simeq& \left\| \rho_2 \lre^{11}_{0} u_b
\right\|_{W^{1/2,k+1}(\partial\Omega_1 
	\cap\partial\Omega,\rho_2,2)}\\
\lesssim& 
\left\| \lre^{11}_{-1,1/2} u_b
\right\|_{W^{1/2,k+1}(\partial\Omega_1 
	\cap\partial\Omega,\rho_2,2)}\\
\lesssim& 
\left\|  u_{2,b1}
\right\|_{W^{-1/2,k+1}(\partial\Omega_1 
	\cap\partial\Omega,\rho_2,2)}\\
\lesssim
& s.c.
\sum_{l \le \lfloor \alpha \rfloor}
\| \rho_2^{2\alpha -2l } u_{2,b1}
\|_{\alpha-l}
 + s.c.\|u_{2,b1}\|.
\end{align*}

In case $ii)$ we estimate
\begin{align*}
\big| (\rho_2^{2\alpha+1}\Lambda^{\alpha}
\Psi^1_b &\lre^{11}_0 u_{b},
\rho_2^{2\alpha+1} \Lambda^{\alpha}
u_{2,b1}^-)\big|
\\
\lesssim &
\sum_{0\le l\le \lceil\alpha\rceil}   
\| \rho_2^{2\alpha+1-2l} \lre^{11}_0 u_{b}
\|_{1/2+\alpha-l}^2
+ \|  \lre^{11}_0u_{b} \|^2\\
&+ s.c. 
 \sum_{0\le l\le \lceil\alpha\rceil}   
\| \rho_2^{2\alpha+1-2l} u_{2,b1}^-
\|_{1/2+\alpha-l}^2
+ \|  u_{2,b1}^-\|^2
.
\end{align*}
The first term on the right of
 the inequality can be estimated 
(in case $ii)$) by
\begin{align*}
\sum_{l \le k}
\Big\| \rho^{2k -2l +1} 
\lre^{11}_{0} u_b
\Big\|_{1/2 + k -l}
\simeq& \left\| \rho_2 \lre^{11}_{0} u_b
\right\|_{W^{1/2,k}(\partial\Omega_1 
	\cap\partial\Omega,\rho_2,2)}\\
\lesssim& 
\left\| \lre^{11}_{-1,1/2} u_b
\right\|_{W^{1/2,k}(\partial\Omega_1 
	\cap\partial\Omega,\rho_2,2)}\\
\lesssim& 
\left\|  u_b
\right\|_{W^{-1/2,k}(\partial\Omega_1 
	\cap\partial\Omega,\rho_2,2)}\\
\lesssim
& 
\sum_{l \le \lfloor \alpha \rfloor}
\| \rho^{2\alpha -2l} u_{b}
\|_{\alpha-l}.
\end{align*}

It remains to write estimates 
 for the $\lre^{21}_{-\beta}$ terms in
the $E^-_2$ operator.  We will go
through
 the estimates for the 
$\lre^{21}_{-3/2} \partial_{\rho_2}P_2(u_b) \big|_{\partial\Omega_2}$
 term, the remaining estimates for
the terms involving $\lre^{21}_{-1/2} u_b$
 being proved similarly.

We write
\begin{align}
\nonumber
\big| (\rho_2^{2\alpha+1}\Lambda^{\alpha}
\Psi^1_b &
\lre^{21}_{-3/2} \partial_{\rho_2}P_2(u_b) ,
\rho_2^{2\alpha+1} \Lambda^{\alpha}
u_{2,b1}^-)\big|
\\
\nonumber
\lesssim &
\sum_{0\le l\le \lceil\alpha\rceil}   
\| \rho_2^{2\alpha+1-2l} 
\lre^{21}_{-3/2} \partial_{\rho_2}P_2(u_b)
\|_{1/2+\alpha-l}^2
+ \|  \lre^{21}_{-3/2} \partial_{\rho_2}P_2(u_b)\|^2\\
\label{caseiiDrho}
&+ s.c. 
\sum_{0\le l\le \lceil\alpha\rceil}   
\| \rho_2^{2\alpha+1-2l} u_{2,b1}^-
\|_{1/2+\alpha-l}^2
+ \|  u_{2,b1}^-\|^2
.
\end{align}
The first term on the right of
the inequality can be estimated 
(in case $i)$) by
\begin{align*}
\sum_{l \le k+1}
\Big\| &\rho_2^{2k -2l +2} 
\lre^{21}_{-3/2} \partial_{\rho_2}P_2(u_b)\big|_{\partial\Omega_2}
\Big\|_{ k+1 -l}\\
\lesssim & 
 \sum_{l \le k}
 \Big\| \rho_2^{2k -2l+2} 
 \lre^{21}_{-3/2} \partial_{\rho_2}P_2(u_b)\big|_{\partial\Omega_2}
 \Big\|_{ k+1 -l}
+  \Big\| 
\lre^{21}_{-3/2} \partial_{\rho_2}
P_2(u_b)\big|_{\partial\Omega_2}
\Big\|
\\
\lesssim & 
\sum_{l \le k}
\Big\| \rho_2^{2k -2l+1} 
\lre^{21}_{-3/2} \partial_{\rho_2}
P_2(u_b)\big|_{\partial\Omega_2}
\Big\|_{ k -l}\\
&+ \sum_{l \le k}
\Big\| \rho_2^{2k -2l+2} 
\lre^{21}_{-1/2} \partial_{\rho_2}
P_2(u_b)\big|_{\partial\Omega_2}
\Big\|_{ k -l}
+  \Big\| 
\lre^{21}_{-3/2} \partial_{\rho_2}
P_2(u_b)\big|_{\partial\Omega_2}
\Big\|.
\end{align*}
Note that the
 $\lre^{21}_{-1/2} 
$ operator is
of the form
\begin{equation*}
|D_1|\circ R_1 \circ \Psi^{-2}+
 |D_1|^2 \circ R_1\circ \Psi^{-1}
  \circ R_2\circ \Psi^{-1}
  \circ R_1\circ \Psi^{-2},
\end{equation*}
by
\eqref{form3/2}, 
and hence
with this $\lre^{21}_{-1/2}$ operator
 we have
\begin{equation*}
 \rho_2 \lre^{21}_{-1/2} 
  \partial_{\rho_2}P_2(u_b)\big|_{\partial\Omega_2}
 =  \lre^{21}_{-3/2} 
 \partial_{\rho_2}P_2(u_b)\big|_{\partial\Omega_2}.
\end{equation*}
With this property, 
 we continue the estimates:
\begin{align*}
\sum_{l \le k+1}
\Big\| \rho_2^{2k -2l +2} 
\lre^{21}_{-3/2} \partial_{\rho_2}&
 P_2(u_b)\big|_{\partial\Omega_2}
\Big\|_{ k+1 -l}\\
\lesssim & 
\sum_{l \le k}
\Big\| \rho_2^{2k -2l+1} 
\lre^{21}_{-3/2} \partial_{\rho_2}
P_2(u_b)\big|_{\partial\Omega_2}
\Big\|_{ k -l}
+  s.c. \|u_b\| \\
\lesssim& 
s.c. 
 \left\| \lre^{21}_{-3/2} \partial_{\rho_2}
 P_2(u_b)\big|_{\partial\Omega_2}
\right\|_{W^{0,k}(\partial\Omega_1 
	\cap\partial\Omega,\rho_2,2)}
+s.c. \|u_b\|
\\
\lesssim& 
s.c. 
\left\| u_{b}
\right\|_{W^{-1/2,k}(\partial\Omega,\rho,2)}
+s.c. \|u_b\|\\
\lesssim
& s.c.
 \sum_{l \le \lfloor \alpha \rfloor}
 \| \rho^{2\alpha -2l} u_{b}
 \|_{W^{\alpha-l}(
 	\partial\Omega)}
+s.c.\|u_{b}\|,
\end{align*}
modulo estimates of smooth terms,
\begin{equation*}
 \| P(u_b) \|_{-\infty}+\left\|
 \partial_{\rho}   P(u_b) \big|_{\partial\Omega}
 \right\|_{-\infty} 
 + \|
 u_b \|_{-\infty} \lesssim \|u_b\|_{L^2(\partial\Omega)}.
\end{equation*}

On the other hand, in case $ii)$ we
 can estimate the first term on 
right of \eqref{caseiiDrho}
by
\begin{align*}
 \sum_{0\le l\le k}   
 \| \rho_2^{2k+1-2l} 
 \lre^{21}_{-3/2} \partial_{\rho_2}&
 P_2(u_b)\big|_{\partial\Omega_2}
 \|_{1/2+k-l}^2\\
\lesssim&
 s.c. \|  \lre^{21}_{-3/2} \partial_{\rho_2}
 P_2(u_b)\big|_{\partial\Omega_2}
 \|_{W^{1/2,k}(\partial\Omega_1 
 	\cap\partial\Omega,\rho_2,2)}\\
\lesssim&
s.c. \|  \partial_{\rho_2}
P_2(u_b)\big|_{\partial\Omega_2}
\|_{W^{-1,k}(\partial\Omega_2
	\cap\partial\Omega,\rho_2,2)}\\
\lesssim&
s.c. \|  u_b
\|_{W^{0,k}(\partial\Omega,\rho,2)}\\
\lesssim
& 
\sum_{l \le \lfloor \alpha \rfloor}
\| \rho^{2\alpha -2l} u_{b}
\|_{W^{\alpha-l}(\partial\Omega)}.
\end{align*}
In fact, the estimates for the
 $\lre^{21}_{-3/2} \partial_{\rho_2}
 P_2(u_b)\big|_{\partial\Omega_2}$
could be improved by taking into account the
 vanishing of $u_2$ along $\partial\Omega_2$, but
the weaker estimates we have above suffice for
 our purposes.

Putting this together, we have
\begin{align*}
\|\rho_2^{2\alpha+1} L_{b1}
\Lambda^{\alpha} u_{2,b1}^-\|^2
\lesssim& 
 \| \rho_2^{2\alpha+1} \Lambda^{\alpha} 
 g\|_{-1/2}^2
+
  s.c.\| \rho_2^{2\alpha+1} u_{2,b1}^-
\|_{1/2+\alpha}^2\\
&+ \sum_{1\le l\le \lceil\alpha\rceil}   
\| \rho_2^{2\alpha+1-2l} u_{2,b1}^-
\|_{1/2+\alpha-l}^2
+ \|  u_{2,b1}^-
\|^2\\
&+\sum_{l \le \lceil\alpha\rceil} \| \rho_2^{2\alpha+1-2l}  u_{2,b1}^0\|_{\alpha+1-l}^2
+ \| u_{2,b1}^0\|^2
\\
 &
 +s.c. \|  \rho_2^{2\alpha+1}
L_{b1} \Psi_b^{\alpha} u_{2,b1}^- \|^2
 +s.c. \|  \rho_2^{2\alpha+1}
\Lb_{b1} \Psi_b^{\alpha} u_{2,b1}^- \|^2\\
&+\sum_{l \le \lfloor\alpha\rfloor} \| \rho^{2\alpha-2l}  u_{b}\|_{W^{\alpha-l}(\partial\Omega)}^2
+ \| u_{b}\|^2.
\end{align*}

Hence, using \eqref{case21}, \eqref{case211}, and
 \eqref{scL}, we have
\begin{align}
\nonumber
\|\rho_2^{2\alpha+1}& L_{b1}
\Lambda^{\alpha} u_{2,b1}^-\|^2
\lesssim\\
\nonumber
&  \sum_{l\le \lceil \alpha \rceil}
\| \rho_2^{2\alpha-2l+1}  g\|_{\alpha-l-1/2}^2
+  \|\Psi_b^{-1/2} g\|_{-1/2}^2+
s.c.\| \rho_2^{2\alpha+1} u_{2,b1}^-
\|_{1/2+\alpha}^2\\
\nonumber
&+ \sum_{1\le l\le \lceil\alpha\rceil}   
\| \rho_2^{2\alpha+1-2l} u_{2,b1}^-
\|_{1/2+\alpha-l}^2
+\sum_{l \le \lfloor\alpha\rfloor} \| \rho^{2\alpha-2l}  u_{b}\|_{W^{\alpha-l}(\partial\Omega)}^2\\
\nonumber
&+\sum_{l \le \lceil\alpha\rceil} \| \rho_2^{2\alpha+1-2l}  u_{2,b1}^0\|_{\alpha+1-l}^2+ \|  u_{b}
\|^2
\\
&+
s.c. \sum_{l\le \lfloor \alpha \rfloor}
\| \rho_2^{2\alpha-2l+1}  \Lb_{b1} u_{2,b1}^-\|_{\alpha-l}^2+
s.c. \sum_{l\le \lfloor \alpha \rfloor}
\| \rho_2^{2\alpha-2l+1}  L_{b1} u_{2,b1}^-\|_{\alpha-l}^2
\label{rLLamSum}
.
\end{align}

Similarly, for any $\Psi_{b1}^{\alpha}$, we can estimate
$\|\rho_2^{2\alpha+1} L_{b1} \Psi_{b1}^{\alpha} u_{2,b1}^-\|^2$ by the
 right-hand side of \eqref{rLLamSum}.
And since
\begin{align*}
\|\rho_2^{2\alpha+1} L_{b1} u_{2,b1}^-\|^2_{\alpha} \lesssim&
\sum_{l\le \lfloor\alpha \rfloor}
\| \rho_2^{2\alpha-2l+1}  L_{b1} 
 \Psi_{b1}^{\alpha-l} u_{2,b1}^-\|^2
+\|\rho_2 u_{2,b1}^-\|_{1/2}^2\\
&+ \sum_{l\le \lfloor\alpha \rfloor} \| 
 \rho_2^{2\alpha-2l+1} \Psi_{b1}^{\alpha-l}
u_{2,b1}^-\|^2
+\|u_{2,b1}^-\|^2\\
\lesssim&
\sum_{l\le \lfloor\alpha \rfloor}
\| \rho_2^{2\alpha-2l+1}  L_{b1}
  \Psi_{b1}^{\alpha-l} u_{2,b1}^-\|^2
+\|\rho_2 u_{2,b1}^-\|_{1/2}^2\\
&+ \sum_{l\le \lfloor\alpha \rfloor} \| \rho_2^{2\alpha-2l+1} 
u_{2,b1}^-\|_{\alpha-l}^2
+\| u_{2,b1}^-\|^2
,
\end{align*}
we have 
\begin{align}
\nonumber
\sum_{l\le \lfloor\alpha \rfloor} \|& \rho_2^{2\alpha-2l+1}  L_{b1} u_{2,b1}^-\|_{\alpha-l}^2
\lesssim\\ 
\nonumber
&  \sum_{l\le \lceil \alpha \rceil}
\| \rho_2^{2\alpha-2l+1}  g\|_{\alpha-l-1/2}^2
+  \|\Psi_{b1}^{-1}g\|^2+
s.c.\| \rho_2^{2\alpha+1} u_{2,b1}^-
\|_{1/2+\alpha}^2
\\
\nonumber
&+ \sum_{1\le l\le \lceil\alpha\rceil}   
\| \rho_2^{2\alpha+1-2l} u_{2,b1}^-
\|_{1/2+\alpha-l}^2
+\sum_{l \le \lfloor\alpha\rfloor} \| \rho^{2\alpha-2l}  u_{b}\|_{W^{\alpha-l}(\partial\Omega)}^2
\\
\nonumber
&+\sum_{l \le \lceil\alpha\rceil} \| \rho_2^{2\alpha+1-2l}  u_{2,b1}^0\|_{\alpha+1-l}^2+ \|  u_{b}
\|^2
\\
\label{boundLTerm}
&+
s.c. \sum_{l\le \lfloor \alpha \rfloor}
\| \rho_2^{2\alpha-2l+1}  \Lb_{b1} u_{2,b1}^-\|_{\alpha-l}^2+
s.c. \sum_{l\le \lfloor \alpha \rfloor}
\| \rho_2^{2\alpha-2l+1}  L_{b1} u_{b}\|_{\alpha-l}^2
.
\end{align}

We similarly have that
\begin{equation*}
\sum_{l\le \lfloor\alpha \rfloor} \| \rho_2^{2\alpha-2l+1} 
 \Lb_{b1} u_{2,b1}^-\|_{\alpha-l}^2
\end{equation*}
is bounded by the right-hand side of 
 \eqref{boundLTerm}.

In particular, from \eqref{halfEst},
we have
\begin{align}
\nonumber
\| \rho_2^{2\alpha+1}& \Psi_{b1}^{\alpha} u_{2,b1}^- \|_{1/2}^2
\lesssim\\
\label{psiVEst}
&  \sum_{l\le \lceil \alpha \rceil}
\| \rho_2^{2\alpha-2l+1}  g\|_{\alpha-l-1/2}^2
+  \|\Psi_{b1}^{-1}g\|^2+
s.c.\| \rho_2^{2\alpha+1} u_{2,b1}^-
\|_{1/2+\alpha}^2\\
\nonumber
&+ \sum_{1\le l\le \lceil\alpha\rceil}   
\| \rho_2^{2\alpha+1-2l} u_{2,b1}^-
\|_{1/2+\alpha-l}^2
+\sum_{l \le \lfloor\alpha\rfloor} \| \rho^{2\alpha-2l}  u_{b}\|_{W^{\alpha-l}(\partial\Omega)}^2
\\
\nonumber
&+\sum_{l \le \lceil\alpha\rceil} \| \rho_2^{2\alpha+1-2l}  u_{2,b1}^0\|_{\alpha+1-l}^2+ \|  u_{b}
\|^2.
\end{align}

In analogy with
 \eqref{case211}, we write
\begin{equation*}
 \Psi_{b1}^{\alpha} \rho_2^{2\alpha+1}
   = \sum_{l \le \lceil \alpha\rceil}
    \rho_2^{2\alpha-2l+1} \Psi_{b1}^{\alpha-l}
    + \Psi_{b1}^{-1}
\end{equation*}
so that we have
\begin{align*}
 \| \rho_2^{2\alpha+1} u_{2,b1}^- 
 \|_{\alpha+1/2}^2
 \lesssim& 
\sum_{l \le \lfloor \alpha\rfloor}
 \| \rho_2^{2\alpha-2l+1} 
 \Psi_{b1}^{\alpha-l} u_{2,b1}^- \|_{1/2}^2
 + \| u_{2,b1}^-\|^2.
\end{align*}
Then using \eqref{psiVEst} for the
 terms in the summation on 
the right yields
\begin{align}
\nonumber
\| \rho_2^{2\alpha+1} & u_{2,b1}^-\|_{\alpha+1/2}^2
\lesssim\\
\label{multEst}
&  \sum_{l\le \lceil \alpha \rceil}
\| \rho_2^{2\alpha-2l+1}  g\|_{\alpha-l-1/2}^2
+  \|\Psi_b^{-1}g\|^2+
s.c.\| \rho_2^{2\alpha+1} u_{2,b1}^-
\|_{1/2+\alpha}^2\\
\nonumber
&+ \sum_{1\le l\le \lceil\alpha\rceil}   
\| \rho_2^{2\alpha+1-2l} u_{2,b1}^-
\|_{1/2+\alpha-l}^2
+\sum_{l \le \lfloor\alpha\rfloor} \| \rho^{2\alpha-2l}  u_{b}\|_{W^{\alpha-l}(\partial\Omega)}^2
\\
\nonumber
&+\sum_{l \le \lceil\alpha\rceil} \| \rho_2^{2\alpha+1-2l}  u_{2,b1}^0\|_{\alpha+1-l}^2+ \|  u_{b}
\|^2
.
\end{align}

The estimates will 
 be obtained by induction.  
To illustrate the process we calculate 
the first few estimates,
starting with
$\| \rho_2 u_{2,b1}^-\|_{1/2}^2$.  
 We continue to refrain from writing the boundary over 
which the norms are taken in the cases
 where it is clear from the boundary distribution.
Thus, for instance we write
$\| \rho_2 u_{2,b1}\|_{1/2}^2$ to mean
 the $W^{1/2}(\partial\Omega_1\cap\partial\Omega)$
norm, and
 $\| \rho_1 u_{1,b2}\|_{1/2}^2$ to mean
 the $W^{1/2}(\partial\Omega_2\cap\partial\Omega)$
 norm.

From
\eqref{multEst}, we have
\begin{align}
\nonumber
\| \rho_2 u_{2,b1}^-\|_{1/2}^2
\lesssim& \|g\|_{-1/2}^2 + \|u_{b}\|^2
+\| \rho_2 u_{2,b1}^0\|_{1}^2\\
\label{guEst}
\lesssim& \|g\|_{-1/2}^2 + \|u_{b}\|^2
 + \|f\|^2_{L^2(\Omega)},
\end{align}
where we use the estimates in \eqref{EstPlusZero} 
 for the $u_{2,b1}^0$ term.
To estimate 
 $g = R_1\circ \Psi^1_{b1}\circ \partial_{\rho_1} G_2(f)$,
 we use Corollary \ref{corGreenNormDerEst}:
\begin{align*}
 \|R_1\circ \Psi^1_{b1}\circ 
 \partial_{\rho_1} \circ G_2(f)\|_{-1/2} \lesssim&  
  \| \partial_{\rho_1} \circ G_2(f)
   \|_{W^{1/2}(\partial\Omega_1\cap
   	 \partial\Omega)}\\
\lesssim&  
\| f \|_{L^2(\Omega)}.
\end{align*} 
Together with \eqref{l2Base} to estimate
 $\|u_{b}\|_{L^2(\partial\Omega)}$
in \eqref{guEst}, this yields
\begin{equation*}
 \| \rho_2 u_{2,b1}^-\|_{1/2}^2
  \lesssim \| f\|_{L^2(\Omega)}^2.
\end{equation*} 
Combining this with weighted $1/2$
 estimates for
 $u_{2,b1}^{+}$ and $u_{2,b1}^0$
as well as the analogous estimates for
 $u_{1,b2}$ we get
\begin{equation}
\label{firstStep}
\| \rho_2 u_{2,b1}\|_{1/2}^2
+\| \rho_1 u_{1,b2}\|_{1/2}^2
\lesssim \| f\|_{L^2(\Omega)}^2.
\end{equation} 
 
Next, we calculate from
 \eqref{multEst}
\begin{align*}
\| \rho_2^2u_{2,b1}^-\|_{1}^2
\lesssim& 
 \| \rho_2^2 g\|^2 + \|g\|_{-1}^2+
 \|\rho_2 
u_{2,b1}\|_{1/2}^2
+
\|\rho_1 
u_{1,b2}\|_{1/2}^2\\
& + 
 \|\rho_2^2 u_{2,b1}^0 \|_{3/2} + \| u_{2,b1}^0 \|_{1/2}
+\|u_b\|^2.
\end{align*}
To estimate $\| \rho_2^2 g\|^2 $,
we use
\begin{align*}
 \|\rho_2^2 \Psi^1_b\circ 
  \partial_{\rho_1} \circ G_2(f)
 \| \lesssim&
\|\rho_2^2 \circ 
\partial_{\rho_1} \circ G_2(f)
\|_1
 + \|
 \partial_{\rho_1} \circ G_2(f)
 \|\\
\lesssim&
 \| f \|_{W^{0,1}(\Omega,\rho_2, 2)},
\end{align*}
by Corollary \ref{corGreenNormDerEst}.
We also use, from \eqref{EstPlusZero},
\begin{align*}
  \|\rho_2^2 u_{2,b1}^0 \|_{3/2}^2 + \| u_{2,b1}^0 \|_{1/2}^2
   \lesssim& s.c. \|u_{2,b1}^0\|^2_{W^{1,1}(\partial\Omega_1\cap
   	 \partial\Omega,\rho_2,2)} + \|u_b\|^2\\
  \lesssim&
  s.c.\| u_{2,b1}
  \|_{W^{0,1}(\partial\Omega_1,\rho_{2},2)}^2
  + s.c. \| u_{1,b2} 
  \|_{W^{0,1}(\partial\Omega_2,\rho_{1},2)}^2\\
  & + \| f 
  \|_{W^{0,1}(\Omega,\rho_{2},2)}^2.
\end{align*}

Combining these estimates with those
 for $\rho_2^2 u_{2,b1}^+$ and $\rho_2^2 u_{2,b1}^0$,
 as well as estimates for $u_{1,b2}$, we have
\begin{equation*}
 \| \rho_2^2u_{2,b1}^-\|_{1}^2
  \lesssim
  \| f \|_{W^{0,1}(\Omega,\rho_2, 2)}^2,
\end{equation*}
and
\begin{equation*}
\| \rho_2^2u_{2,b1}\|_{1}^2
+\| \rho_1^2u_{1,b2}\|_{1}^2
\lesssim
\| f \|_{W^{0,1}(\Omega,\rho_2, 2)}^2.
\end{equation*}

Next we can calculate
$\|\rho_2^3 u_{2,b1}^-\|_{3/2}$
using
\eqref{multEst} (with $\alpha=1$).
\begin{align*}
\| \rho_2^{3}  u_{2,b1}^-\|_{3/2}^2
\lesssim& 
 \| \rho_2^3 g\|_{1/2}^2 + 
  \|\rho_2 g\|_{-1/2}^2+
\| g\|_{-1}^2\\
&+
\| \rho_2^{2}  u_{2,b1}\|_{1}^2 
+
\| \rho_1^{2}  u_{1,b2}\|_{1}^2 \\
&+\| \rho_2  u_{2,b1}\|_{1/2}^2
+\| \rho_1  u_{1,b2}\|_{1/2}^2\\
&+\|\rho_2^3 u_{2,b1}^0 \|_{2}^2 + 
 \| \rho_2u_{2,b1}^0 \|_{1}^2 
+ \|u_{b}\|^2
.
\end{align*}
The terms involving $u_{2,b1}^0$ can be handled as 
 before with
\begin{align*}
   \|\rho_2^3 u_{2,b1}^0 \|_{2}^2 + 
   \| \rho_2u_{2,b1}^0 \|_{1}^2 
 \lesssim&
 s.c.\| u_{2,b1}
 \|_{W^{0,1}(\partial\Omega_1,\rho_{2},2)}^2
 + s.c. \| u_{1,b2} 
 \|_{W^{0,1}(\partial\Omega_2,\rho_{1},2)}^2\\
 & + \| f 
 \|_{W^{0,1}(\Omega,\rho_{2},2)}^2.
\end{align*}

The only remaining term we need to estimate is the 
$\| \rho_2^3 g\|_{1/2}^2$ term.
 This follows as above:
\begin{align*}
 \| \rho_2^3 \Psi^1\circ 
 \partial_{\rho_1} \circ G_2(f)\|_{1/2}
  \lesssim& 
\| \rho_2^2
\partial_{\rho_1} \circ G_2(f)\|_{3/2}
+  \| 
\partial_{\rho_1} \circ G_2(f)\|_{1/2}\\
 \simeq &
\| R_1 \circ \partial_{\rho_1} \circ G_2(f)
 \|_{W^{1/2,1}(\partial\Omega_1 
 	 \cap \partial\Omega,\rho_2,2)}\\
 \lesssim&
  \| f \|_{W^{0,1}(\Omega,\rho_2, 2)}.
\end{align*}
We thus have
\begin{equation*}
 \| \rho_2^{3} u_{2,b1}^-\|_{3/2}^2
 \lesssim \| f \|_{W^{0,1}(\Omega,\rho_2, 2)}^2
\end{equation*}
which can be combined with the
 other weighted $3/2$ estimates in the usual
way.

For a last illustrative step, we
 estimate
 $\| \rho_2^4 u_{2,b1}^-\|_2^2$,
  for which we use
\eqref{multEst} (with $\alpha=3/2$):
\begin{align*}
\| \rho_2^{4}  u_{2,b1}^-\|_{2}^2
\lesssim& 
\| \rho_2^4 g\|_{1}^2 + 
\|\rho_2^2 g\|^2+
\| g\|_{-1}^2\\
&
+
\| \rho_2^{3}  u_{2,b1}\|_{3/2}^2 
+
\| \rho_1^{3}  u_{1,b2}\|_{3/2}^2 
+
\| \rho_2^{2}  u_{2,b1}\|_{1}^2 
+
\| \rho_1^{2}  u_{1,b2}\|_{1}^2 \\
&+\| \rho_2  u_{2,b1}\|_{1/2}^2 
+\| \rho_1  u_{1,b2}\|_{1/2}^2 \\
&+ s.c. \|u_{2,b1}^0\|^2_{W^{1,2}(\partial\Omega_1\cap
	\partial\Omega,\rho_2,2)} + \|u_b\|^2
.
\end{align*}
We note
\begin{align*}
\| \rho_2^4 g\|_{1}
\lesssim& 
\| \rho_2^4
\partial_{\rho_1} \circ G_2(f)\|_{2}
+  \| \rho_2^2
\partial_{\rho_1} \circ G_2(f)\|_{1}
+\| 
\partial_{\rho_1} \circ G_2(f)\|\\
\simeq &
\| R_1 \circ \partial_{\rho_1} \circ G_2(f)
\|_{W^{0,2}(\partial\Omega_1 
	\cap \partial\Omega,\rho_2,2)}\\
\lesssim&
\| f \|_{W^{0,2}(\Omega,\rho_2, 2)},
\end{align*}
and, from \eqref{EstPlusZero},
\begin{align*}
 \|u_{2,b1}^0\|^2_{W^{1,2}(\partial\Omega_1\cap
 	\partial\Omega,\rho_2,2)}
  \lesssim& 
\| u_{2,b1}
\|_{W^{0,2}(\partial\Omega_1,\rho_{2},2)}^2
+ \| u_{1,b2} 
\|_{W^{0,2}(\partial\Omega_2,\rho_{1},2)}^2\\
& + \| f 
\|_{W^{0,2}(\Omega,\rho_{2},2)}^2,
\end{align*}
to
conclude
\begin{multline*}
\|  u_{2,b1}\|_{W^{0,2}
 (\partial\Omega_1\cap\partial\Omega,
 \rho_2,2)} 
+\|  u_{1,b2}\|_{W^{0,2}
	(\partial\Omega_2\cap\partial\Omega,
	\rho_1,2)}
\lesssim\\
\| f \|_{W^{0,2}(\Omega,\rho_2, 2)}
+\| f \|_{W^{0,2}(\Omega,\rho_1, 2)}.
\end{multline*}

Higher order (weighted) estimates are
calculated with repeated application
of \eqref{multEst} and the already
obtained lower order
estimates.

We put together the estimates for the
 boundary solution in the following
\begin{thrm}
Assume $u_b$ is a smooth
 form on $\partial\Omega$
satisfying \eqref{bndrySub}
 with analogous expression for 
other microlocal regions.  Then 
 we have the estimates
\begin{equation*}
	\sum_j \left\|
	u_{bj}
	\right\|_{W^{0,s}
		\left(
		\partial\Omega_j \cap \partial\Omega,
		\rho_{k},2
		\right)}
	\lesssim
	\sum_j \| f\|_{W^{0,s}
		\left( \Omega,
		\rho_{j},2 \right)}
.
\end{equation*}
\end{thrm}

\section{Regularizing operators}
\label{secRegularize}

Here we sketch the argument
 validating the regularity 
  assumption in Section
  \ref{secWeightedEst} we used to
integrate by parts.  We adopt the 
 regularizing operators used in
\cite{K85}.  
 We let $\zeta(x,\rho_2)$ be a cutoff
function with support in a neighborhood
 of the origin intersected with
  $\partial\Omega_1$ (i.e. a neighborhood
of the boundary singularity in which we
 are working).
For $\delta>0$ we write
 $\lrp_{\delta}^-$ to denote an operator
  in $\Psi^{0}(\partial\Omega_1)$
whose symbol,
 $p_{\delta}^-(x,\rho_2, \eta_2,\xi) = \sigma(\lrp_{\delta}^-)$
is of the form
\begin{equation*}
 p_{\delta}^- = c_{\delta}\left(
 \sqrt{\eta_2^2+\xi^2}\right)
  \zeta(x,\rho_2) 
  \psi^{-}_k(\eta_2,\xi), 
\end{equation*}
where $c_{\delta}(t) \in 
 C^{\infty}_0(\overline{\mathbb{R}_+})$
with 
 $c_{\delta}(t)\equiv 1$ for  $t\in[0,\delta]$
and $c_{\delta}(t)\equiv 0$ 
 for $t>\delta+1$.
Furthermore, the symbol estimates,
 realizing $p_{\delta}^-$ as a function
in class $\lrs^{0}(\partial\Omega_1
 \times \mathbb{R}^3)$,
 are independent of the parameter $\delta$.
We also write $\lrp^{'-}_{\delta}$ 
 for operators which dominate
$\lrp_{\delta}^-$, by which we mean that
 the symbol,  $p_{\delta}^{'-}$
can be written
\begin{equation*}
p_{\delta}^{'-} = c_{\delta'}\left(
\sqrt{\eta_2^2+\xi^2}\right)
\zeta'(x,\rho_2) 
\psi^{-}_{k'}(\eta_2,\xi), 
\end{equation*}
where $\delta'>\delta$, 
 $\zeta'\in C^{\infty}(\partial\Omega_1)$
with support near the origin and
 $\zeta'\equiv 1$ on the support of 
$\zeta$, and $k'<k$.

We note 
 $\lrp_{\delta}^- u_{2,b1} \in 
  C^{\infty}(\partial\Omega_1)$ and
hence the integration by parts 
leading to
\eqref{halfEst} is valid 
 for  $\lrp_{\delta}^- u_{2,b1}$.  We recall 
that we consider $\Lambda^{\alpha}$ to be an operator
 with symbol
\begin{equation*}
 \sigma(\Lambda^{\alpha}) = \varphi(x,\rho_2) 
  \left( 1 + \xi^2 + \eta_2^2\right)^{\alpha/2},
\end{equation*}
where $\varphi$ has support where $\zeta\equiv 1$.
\begin{equation*}
\begin{aligned}
\|  \rho_2^{2\alpha+1} \Lambda^{\alpha}
 \lrp_{\delta}^-
  u_{2,b1} \|_{1/2}^2
\lesssim& \left|\left((T_1 
\rho_2^{2\alpha+1} \Lambda^{\alpha}
 \lrp_{\delta}^-
  u_{2,b1}, 
  \rho_2^{2\alpha+1} \Lambda^{\alpha}
   \lrp_{\delta}^-
   u_{2,b1}\right) \right|\\
\lesssim& \| L_{b1} 
\rho_2^{2\alpha+1} \Lambda^{\alpha}
\lrp_{\delta}^-
 u_{2,b1} \|^2 +
\|\Lb_{b1} 
\rho_2^{2\alpha+1} \Lambda^{\alpha}
\lrp_{\delta}^-
 u_{2,b1}\|^2\\
 & + \|\rho_2^{2\alpha+1} \Lambda^{\alpha}
 \lrp_{\delta}^- u_{2,b1}\|^2\\
\lesssim& \| \rho_2^{2\alpha+1}  L_{b1} 
\Lambda^{\alpha}
\lrp_{\delta}^-
u_{2,b1} \|^2 +
\|\rho_2^{2\alpha+1}\Lb_{b1} 
 \Lambda^{\alpha}
\lrp_{\delta}^-
u_{2,b1}\|^2\\
& + \|\rho_2^{2\alpha} \Lambda^{\alpha}
\lrp_{\delta}^- u_{2,b1}\|^2.
\end{aligned}
\end{equation*}

In estimating the first two terms
 on the right, we follow our 
previous analysis to write
\begin{align}
\nonumber
(\rho_2^{2\alpha+1} L_{b1} \Lambda^{\alpha} 
\lrp_{\delta}^- u_{2,b1},
&\rho_2^{2\alpha+1} L_{b1} \Lambda^{\alpha}
\lrp_{\delta}^- u_{2,b1})\\
\nonumber
\lesssim& -(\rho_2^{2\alpha+1} 
\Lb_{b1} L_{b1} \Lambda^{\alpha} 
\lrp_{\delta}^- u_{2,b1}, 
\rho_2^{2\alpha+1} \Lambda^{\alpha} \lrp_{\delta}^-u_{2,b1})\\
\label{LLReg}
&+ s.c. \|\rho_2^{2\alpha+1}
L_{b1}\Lambda^{\alpha} 
\lrp_{\delta}^- u_{2,b1}\|^2 +
\|\rho_2^{2\alpha}\Lambda^{\alpha}
\lrp_{\delta}^- u_{2,b1}\|^2.
\end{align}
We can commute the two $L$ derivatives past
 the $\Lambda^{\alpha}$ operator as before,
but we need then an expression for
 the resulting operator applied to 
$\lrp_{\delta}^- u_{2,b1}^-$.
 For this, we need a replacement for
\eqref{bndrySub}:
\begin{multline}
\label{LLReplace}
L_{b1} \Lb_{b1} \lrp_{\delta}^- u_{2,b1}
= \lrp_{\delta}^- g+ 
\Psi^1_{b1,\psi^0}\lrp_{\delta}^-
u_{2,b1} +
\Psi_{\varepsilon,b1}^1  
\lrp_{\delta}^-u_{2,b1}
 + \Psi^1_{b1} E^-_2 u_b\\
+ \Psi^0_{b1}\lrp_{\delta}^{'-} L_{b1} u_{2,b1}
+ \Psi^0_{b1}\lrp_{\delta}^{'-} \Lb_{b1} u_{2,b1}
+ \Psi^0_{b1}\lrp_{\delta}^{'-} u_{2,b1}
.
\end{multline}
Again, \eqref{LLReplace} holds in the support
 of $\varphi$.  That the support of 
  $\varphi$ is contained in $\zeta\equiv 1$ gives rise to
the $\Psi^1_{b1,\phi^0}$ operator.

We can then estimate the term
  in 
\eqref{LLReg}
by
\begin{align}
\nonumber
(\rho_2^{2\alpha+1} L_{b1} \Lambda^{\alpha} 
\lrp_{\delta}^- u_{2,b1},
&\rho_2^{2\alpha+1} L_{b1} \Lambda^{\alpha}
\lrp_{\delta}^- u_{2,b1})\\
\nonumber
\lesssim&
-(\rho_2^{2\alpha+1} 
\Lambda^{\alpha}  L_{b1} \Lb_{b1}
\lrp_{\delta}^- u_{2,b1}, 
\rho_2^{2\alpha+1} \Lambda^{\alpha} \lrp_{\delta}^- u_{2,b1})\\
\nonumber
&+
 s.c. \| \rho_2^{2\alpha+1}  L_{b1} 
 \Lambda^{\alpha}
 \lrp_{\delta}^-
 u_{2,b1} \|^2 +
s.c. \|\rho_2^{2\alpha+1}\Lb_{b1} 
 \Lambda^{\alpha}
 \lrp_{\delta}^-
 u_{2,b1}\|^2\\
\label{LLReg2}
 & + \|\rho_2^{2\alpha} \Lambda^{\alpha}
 \lrp_{\delta}^- u_{2,b1}\|^2 
 ,
\end{align} 
again using G\aa rding's inequality to handle
  the commutation term,
$[\Lb_{b1},L_{b1}]$, and using \eqref{LLReplace}
then gives for the first term on the right
\begin{align*}
 (\rho_2^{2\alpha+1} 
 \Lambda^{\alpha}  L_{b1} \Lb_{b1}
 \lrp_{\delta}^- u_{2,b1}, &
 \rho_2^{2\alpha+1} \Lambda^{\alpha} \lrp_{\delta}^- u_{2,b1})\\
  \lesssim&
(\rho_2^{2\alpha+1} 
\Lambda^{\alpha} 
\lrp_{\delta}^- g, 
\rho_2^{2\alpha+1} \Lambda^{\alpha} \lrp_{\delta}^- u_{2,b1})\\
&+ (\rho_2^{2\alpha+1} 
\Lambda^{\alpha} 
\Psi^1_{b1,\psi^0} \lrp_{\delta}^- 
u_{2,b1}, 
\rho_2^{2\alpha+1} \Lambda^{\alpha} \lrp_{\delta}^- u_{2,b1})\\
&+ (\rho_2^{2\alpha+1} 
\Lambda^{\alpha} 
\Psi^1_{\varepsilon,b1} \lrp_{\delta}^- 
 u_{2,b1}, 
\rho_2^{2\alpha+1} \Lambda^{\alpha} \lrp_{\delta}^- u_{2,b1})\\
&+ (\rho_2^{2\alpha+1} 
\Lambda^{\alpha} 
\Psi^1_{b1} E^-_2
u_b, 
\rho_2^{2\alpha+1} \Lambda^{\alpha} \lrp_{\delta}^- u_{2,b1})\\
&+ (\rho_2^{2\alpha+1} 
\Lambda^{\alpha} 
 \lrp_{\delta}^{'-} L_{b1}
u_{2,b1}, 
\rho_2^{2\alpha+1} \Lambda^{\alpha} \lrp_{\delta}^- u_{2,b1})
\\
&+ (\rho_2^{2\alpha+1} 
\Lambda^{\alpha} 
\lrp_{\delta}^{'-} \Lb_{b1}
u_{2,b1}, 
\rho_2^{2\alpha+1} \Lambda^{\alpha} \lrp_{\delta}^- u_{2,b1})\\
 & + \|\rho_2^{2\alpha} \Lambda^{\alpha}
\lrp_{\delta}^{'-} u_{2,b1}\|^2 .
\end{align*}

From Section \ref{secWeightedEst}, the
 first three terms on the 
  right can be bounded by
\begin{align*}
\sum_{l\le \lceil \alpha \rceil}
\| \rho_2^{2\alpha-2l+1}
\lrp_{\delta}^{-}   g&\|_{\alpha-l-1/2}^2
+  \|\lrp_{\delta}^{-} g\|^2+
s.c.\| \rho_2^{2\alpha+1} \lrp_{\delta}^{-} 
 u_{2,b1}
\|_{1/2+\alpha}^2\\
&+
\sum_{l \le \lceil\alpha\rceil} \| \rho_2^{2\alpha+1-2l}  
 \lrp_{\delta}^{-} u_{2,b1}^0\|_{\alpha+1-l}^2
+ \| \lrp_{\delta}^{-} u_{2,b1}^0\|^2
\\
\nonumber
&+ \sum_{1\le l\le \lceil\alpha\rceil}   
\| \rho_2^{2\alpha+1-2l} 
\lrp_{\delta}^{-} u_{2,b1}
\|_{1/2+\alpha-l}^2
+ \|  \lrp_{\delta}^{-}  u_{2,b1}
\|^2.
\end{align*}

We note that in 
the second and third to last
 terms, the $\lrp^{'-}_{\delta}$ and
$\lrp^{-}_{\delta}$ can be switched, 
 the error involving only lower order terms:
\begin{align*}
 (\rho_2^{2\alpha+1} 
 \Lambda^{\alpha} 
 \lrp_{\delta}^{'-} L_{b1}
 u_{2,b1}, &
 \rho_2^{2\alpha+1} \Lambda^{\alpha} \lrp_{\delta}^- u_{2,b1})\\
  \lesssim&
(\rho_2^{2\alpha+1} 
\Lambda^{\alpha} 
\lrp_{\delta}^{-} L_{b1}
u_{2,b1}, 
\rho_2^{2\alpha+1} \Lambda^{\alpha} \lrp_{\delta}^{'-} u_{2,b1})\\
& + 
\sum_{1\le l\le \lceil\alpha\rceil}
 \|\rho_2^{2\alpha-2l-1} \Psi_b^{\alpha-l}
 \lrp_{\delta}^{''-} L_{b1}
 u_{2,b1}\|^2
\\
& + 
\sum_{1\le l\le \lceil\alpha\rceil}
\|\rho_2^{2\alpha-2l-1} \Psi_b^{\alpha-l}
\lrp_{\delta}^{''-} \Lb_{b1}
u_{2,b1}\|^2 \\
& + s.c. 
\| \rho_2^{2\alpha+1} \Lambda^{\alpha} 
 \lrp_{\delta}^{''-} u_{2,b1}\|^2
 + \|u_b\|^2
 ,
\end{align*}
where $\lrp_{\delta}^{''-}$
 dominates $\lrp_{\delta}^{'-}$
(similar relations hold for 
 $\Lb_{b1}$), and
where we use
\begin{equation*}
\rho_2^{2\alpha+1} \lrp_{\delta}^{'-} \Psi^{\alpha}_b 
=
\lrp_{\delta}^{'-} \rho_2^{2\alpha+1} 
  \Psi^{\alpha}_b
+
\sum_{1\le l\le \lceil\alpha\rceil}
\rho_2^{2\alpha+1-2l}
\Psi^{\alpha-l}_b \lrp_{\delta}^{''-}
+ 
\Psi^{-1}_b.
\end{equation*}

To handle the terms with
 $\Psi_b^1 E_2^- u_b$ we argue as in
Section \ref{secWeightedEst} to reduce to 
 lower order terms:
\begin{align*}
(\rho_2^{2\alpha+1} 
\Lambda^{\alpha} 
\Psi_b^1 E_2^- u_b, &
\rho_2^{2\alpha+1} \Lambda^{\alpha} \lrp_{\delta}^- u_{2,b1})\\
\lesssim&
 \sum_{l \le \lfloor \alpha \rfloor}
 \| \rho^{2\alpha -2l} u_{b}
 \|_{\alpha-l}
 +\|u_{b}\|+s.c.\| \rho_2^{2\alpha+1} \lrp_{\delta}^{-} 
u_{2,b1}
\|_{1/2+\alpha}^2\\
&+ \sum_{1\le l\le \lceil\alpha\rceil}   
\| \rho_2^{2\alpha+1-2l} 
\lrp_{\delta}^{-} u_{2,b1}
\|_{1/2+\alpha-l}^2.
\end{align*}
Putting this together yields
\begin{align}
\nonumber
(\rho_2^{2\alpha+1} 
\Lambda^{\alpha}  L_{b1} \Lb_{b1}
\lrp_{\delta}^-& u_{2,b1}^-, 
\rho_2^{2\alpha+1} \Lambda^{\alpha} \lrp_{\delta}^- u_{2,b1}^-)\\
\nonumber
\lesssim&
\sum_{l\le \lceil \alpha \rceil}
\| \rho_2^{2\alpha-2l+1}
\lrp_{\delta}^{-}   g\|_{\alpha-l-1/2}^2
+  \|\lrp_{\delta}^{-} g\|^2
\\
\nonumber
&+
s.c.\| \rho_2^{2\alpha+1} \lrp_{\delta}^{-} 
u_{2,b1}
\|_{1/2+\alpha}^2\\
\nonumber
&+
\sum_{l \le \lceil\alpha\rceil} \| \rho_2^{2\alpha+1-2l}  
\lrp_{\delta}^{-} u_{2,b1}^0\|_{\alpha+1-l}^2
+ \| \lrp_{\delta}^{-} u_{2,b1}^0\|^2\\
\nonumber
&+ \sum_{1\le l\le \lceil\alpha\rceil}   
\| \rho_2^{2\alpha+1-2l} 
\lrp_{\delta}^{-} u_{2,b1}
\|_{1/2+\alpha-l}^2
\\
\nonumber
&+
s.c. \| \rho_2^{2\alpha+1}  L_{b1} 
\Lambda^{\alpha}
\lrp_{\delta}^-
u_{2,b1} \|^2 +
s.c. \|\rho_2^{2\alpha+1}\Lb_{b1} 
\Lambda^{\alpha}
\lrp_{\delta}^{-}
u_{2,b1}\|^2
\\
\nonumber
& + 
\sum_{1\le l\le \lceil\alpha\rceil}
\|\rho_2^{2\alpha-2l-1} \Lambda^{\alpha-l}
\lrp_{\delta}^{''-} L_{b1}
u_{2,b1}^-\|^2
\\
\nonumber
& + 
\sum_{1\le l\le \lceil\alpha\rceil}
\|\rho_2^{2\alpha-2l-1} \Lambda^{\alpha-l}
\lrp_{\delta}^{''-} \Lb_{b1}
u_{2,b1}^-\|^2 \\
& + \sum_{l \le \lfloor \alpha \rfloor}
\| \rho^{2\alpha -2l} u_{b}
\|_{\alpha-l}
+ \|u_b\|^2
\label{LLReg3}
\end{align}

Genuine estimates can be obtained 
 from the following argument. 
From \eqref{LLReg2}, and the corresponding
 estimates for its analogue with
the $\Lb_{b1}$ operator replacing
 $L_{b1}$, and from
\eqref{LLReg3}, we can obtain 
 estimates for
\begin{equation}
\label{LLSumReg}
 \| \rho_2^{2\alpha+1}  L_{b1} 
 \Lambda^{\alpha}
 \lrp_{\delta}^-
 u_{2,b1} \|^2 +
 \|\rho_2^{2\alpha+1}\Lb_{b1} 
 \Lambda^{\alpha}
 \lrp_{\delta}^-
 u_{2,b1}\|^2
\end{equation}
in terms of 
 a $s.c.\| \rho_2^{2\alpha+1} \lrp_{\delta}^{-} 
 u_{2,b1}
 \|_{1/2+\alpha}^2$ plus terms of 
(weighted) lower order of
 $u_b$ as well as weighted lower order
terms of $L_{b1} u_{2,b1}$ and
 $\Lb_{b1} u_{2,b1}$.  The higher order 
  norms involving $u_{2,b1}^0$ can
 be handled using \eqref{EstPlusZero}:
\begin{equation*}
 \sum_{l \le \lceil\alpha\rceil} \| \rho_2^{2\alpha+1-2l}  
 \lrp_{\delta}^{-} u_{2,b1}^0\|_{\alpha+1-l}^2
  \lesssim 
 \sum_{l \le \lfloor \alpha \rfloor}
\| \rho^{2\alpha -2l} u_{b}
\|_{\alpha-l}
 + \|f\|_{W^{0,\lceil\alpha\rceil}(\Omega,\rho_2,2)}.
\end{equation*} 
 
 An induction argument
thus gives an estimate for 
 $\| \rho_2^{2\alpha+1} \lrp_{\delta}^{-} 
 u_{2,b1}
 \|_{1/2+\alpha}^2$ which can then be
used to deduce estimates for
 \eqref{LLSumReg}.
Letting $\delta\rightarrow \infty$ and 
 combining estimates for
$u_{b1}^+$ and $u_{b1}^0$ yields estimates
 for 
$\| \rho_2^{2\alpha+1} 
u_{2,b1}
\|_{1/2+\alpha}^2$ ,
as well as
\begin{equation*}
\| \rho_2^{2\alpha+1}  L_{b1} 
\Lambda^{\alpha}
u_{2,b1} \|^2 +
\|\rho_2^{2\alpha+1}\Lb_{b1} 
\Lambda^{\alpha}
u_{2,b1}\|^2.
\end{equation*}
All these
  estimates can in turn be used in the next induction 
 step to estimate
$\| \rho_2^{2\alpha+2} \lrp_{\delta}^{-} 
u_{2,b1}
\|_{1+\alpha}^2$ and so forth.
 
Recalling the notation
 $g=R_1\circ \Psi^1_b\circ \partial_{\rho_1} G_2(f)$
we have
\begin{thrm}
	\label{ubEst}
Let $\Omega\subset\mathbb{C}^2$
  be the piecewise smooth
intersection domain, $\Omega = \Omega_1\cap\Omega_2$
with generic corners.
Let $u = u_1 \omegab_1 + u_2 \omegab_2$
 be the solution to the 
  \dbar-Neumann problem \eqref{dbarNEqn}
with boundary conditions \eqref{bndryCndns}.
 Let $u_{bj} = u|_{\partial\Omega_j\cap \partial\Omega}$.
Then 
	\begin{equation*}
	\sum_j \left\|
	u_{bj}
	\right\|_{W^{0,s}
		\left(
		\partial\Omega_j \cap \partial\Omega,
		\rho_{k},2
		\right)}
	\lesssim
	\sum_j \| f\|_{W^{0,s}
		\left( \Omega,
		\rho_{j},2 \right)}
	.
	\end{equation*}
\end{thrm}
 
\section{Estimates of solution operator}

By definition of the Poisson and
 Green operators, the solution to 
\eqref{dbarNEqn} with boundary conditions
 \eqref{bndryCndns} can be written
\begin{equation*}
 u = G(2f) + P(u_b).
\end{equation*}
From the weighted estimates for 
 Green's operator (see 
  Theorem \ref{weightedGreen}), we have 
\begin{equation*}
\sum_j \left\|  G(2f)
\right\|_{W^{2,s}(\Omega,\rho_{j},2)}
\lesssim
\sum_j \| f\|_{W^{0,s}
	\left( \Omega,
	\rho_{j},2 \right)}.
\end{equation*}
In addition, we have the weighted estimates
 for the Poisson operator from Theorem
 \ref{thrmWeightedPoisson}:
\begin{align*}
 \left\|
 P(u_b)  \right\|_{W^{1/2,s}
 	\left( \Omega, \rho,2
 	\right)}
 \lesssim& \sum_j \left\|
 u_{bj} 
 \right\|_{W^{0,s}
 	\left(
 	\partial\Omega_j\cap \partial\Omega,
 	\rho_{k},2
 	\right) }\\
\lesssim&
\sum_j \| f\|_{W^{0,s}
	\left( \Omega,
	\rho_{j},2 \right)}.
\end{align*}
We thus have estimates for the \dbar-Neumann
 problem:
\begin{thrm}
\label{dbarNEst}
Let $\Omega$ be as in Theorem \ref{ubEst}.
Let $f\in W^{0,s}(\Omega,\rho_j,2)$ for
 $j=1,2$.
	Let $N$ be solution operator 
to the \dbar-Neumann problem \eqref{dbarNEqn},
 \eqref{bndryCndns}.
Then
\begin{equation*}
 \sum_j \|Nf\|_{W^{1/2,s}
 \left(
\Omega,
 \rho_j,2
 \right)} \lesssim
\sum_j \|f\|_{W^{0,s}
	\left(
\Omega,
	\rho_j,2
	\right)}.
\end{equation*}
\end{thrm}
When $s=0$ this is the same result
 of Michel and Shaw in 
\cite{MS98} (Theorem 1.2).

If we also suppose that the data form
 $f\in W^s(\Omega)$.  Then since
\begin{equation*}
 \sum_j \|f\|_{W^{0,s}
 	\left(
\Omega,
 	\rho_{j},2
 	\right)}
\lesssim 
  \|f\|_{W^{s}
  	\left(
  \Omega
  	\right)},
\end{equation*}
Theorem \ref{dbarNEst}
 implies
\begin{equation*}
 \|Nf\|_{W^{1/2,s}
 	\left(
 	\Omega,
 	\rho,2
 	\right)} \lesssim
\|f\|_{W^{s}
 	\left(
 \Omega
 	\right)}.
\end{equation*}

\end{document}